\documentclass[10pt, a4paper,
oneside, headinclude,footinclude]{scrartcl}
\usepackage[utf8]{inputenc}

\usepackage[
nochapters, 
pdfspacing, 
dottedtoc 
]{classicthesis} 
\usepackage{amsopn}
\usepackage[T1]{fontenc} 
\usepackage{multirow}
\usepackage[utf8]{inputenc} 
\usepackage{hhline}
\usepackage{graphicx} 
\graphicspath{{Figures/}} 
\usepackage{indentfirst}
\usepackage{enumitem} 
\usepackage{savesym}
\usepackage{mathrsfs}
\usepackage{geometry}
\usepackage{esint}
\usepackage{longtable}

\usepackage[
    backend=biber,
    style=numeric,
    natbib=false,
    url=false, 
    doi=false,
    eprint=false,
    maxnames=50
]{biblatex}

\usepackage{subfig} 

\usepackage{amsmath, amssymb,amsthm,amsfonts} 

\usepackage[english,capitalize]{cleveref}

\usepackage{thmtools}
\usepackage{thm-restate}

\usepackage{hyperref}
\newcommand{\vertiii}[1]{{\left\vert\kern-0.25ex\left\vert\kern-0.25ex\left\vert #1 
    \right\vert\kern-0.25ex\right\vert\kern-0.25ex\right\vert}}

\theoremstyle{plain}
\newtheorem{teorema}{Theorem}[section]
\newtheorem{proposizione}[teorema]{Proposition}
\newtheorem{lemma}[teorema]{Lemma}
\newtheorem{corollario}[teorema]{Corollary}
\newtheorem*{theorem*}{Theorem}

\theoremstyle{definition}
\newtheorem{definizione}{Definition}[section]

\theoremstyle{remark}
\newtheorem{osservazione}{Remark}[section]

\newcommand{\Tan}{\mathrm{Tan}}
\newcommand{\N}{\mathbb{N}}

\newcommand{\R}{\mathbb{R}}

\newcommand{\G}{Gr}
\newcommand{\res}

\newcounter{const}
\newcommand{\newC}{\refstepcounter{const}\ensuremath{C_{\theconst}}}
\newcommand{\oldC}[1]{\ensuremath{C_{\ref{#1}}}}

\newcounter{eps}
\newcommand{\newep}{\refstepcounter{eps}\ensuremath{\varepsilon_{\theeps}}}
\newcommand{\oldep}[1]{\ensuremath{\varepsilon_{\ref{#1}}}}

\DeclareMathOperator*{\lip}{Lip_1^+}

\DeclareMathOperator*{\diam}{diam}

\DeclareMathOperator*{\dist}{dist}

\hypersetup{
colorlinks=true, breaklinks=true,bookmarksnumbered,
urlcolor=webbrown, linkcolor=RoyalBlue, citecolor=webgreen, 
pdftitle={}, 
pdfauthor={\textcopyright}, 
pdfsubject={}, 
pdfkeywords={}, 
pdfcreator={pdfLaTeX}, 
pdfproducer={LaTeX with hyperref and ClassicThesis} 
} 

\title{On rectifiable measures in Carnot groups: existence of density} 
\author{Gioacchino Antonelli\textsuperscript{*} and Andrea Merlo\textsuperscript{**}}

\date{}

\begin{document}

\renewcommand{\sectionmark}[1]{\markright{\spacedlowsmallcaps{#1}}} 
\lehead{\mbox{\llap{\small\thepage\kern1em\color{halfgray} \vline}\color{halfgray}\hspace{0.5em}\rightmark\hfil}} 
\pagestyle{scrheadings}
\maketitle 
\setcounter{tocdepth}{2}
\paragraph*{Abstract}
In this paper we start a detailed study of a new notion of rectifiability in Carnot groups: we say that a Radon measure is $\mathscr{P}_h$-rectifiable, for $h\in\mathbb N$, if it has positive $h$-lower density and finite $h$-upper density almost everywhere, and, at almost every point, it admits a unique tangent measure up to multiples. 

First, we compare  $\mathscr{P}_h$-rectifiability with other notions of rectifiability previously known in the literature in the setting of Carnot groups, and we prove that it is strictly weaker than them. Second, we prove several structure properties of $\mathscr{P}_h$-rectifiable measures. Namely, we prove that the support of a $\mathscr P_h$-rectifiabile measure is almost everywhere covered by sets satisfying a cone-like property, and in the particular case of $\mathscr P_h$-rectifiabile measures with complemented tangents, we show that they are supported on the union of intrinsically Lipschitz and differentiable graphs. Such a covering property is used to prove the main result of this paper: we show that a $\mathscr{P}_h$-rectifiable measure has almost everywhere positive and finite $h$-density whenever the tangents admit at least one complementary subgroup. 

{\let\thefootnote\relax\footnotetext{* \textit{Scuola Normale Superiore, Piazza dei Cavalieri, 7, 56126 Pisa, Italy,}}}
{\let\thefootnote\relax\footnotetext{** \textit{Univerité Paris-Saclay, 307 Rue Michel Magat Bâtiment, 91400 Orsay, France.}}}


\section{Introduction}

In Euclidean spaces a Radon measure $\phi$ is said to be $k$-\emph{rectifiable} if it is absolutely continuous with respect to the $k$-dimensional Hausdorff measure $\mathcal{H}^k$ and it is supported on a countable union of $k$-dimensional Lipschitz surfaces, for a reference see \cite[\S 3.2.14]{Federer1996GeometricTheory}. This notion of regularity for a measure is an established, thoroughly studied and well understood concept and its versatility is twofold. On the one hand it can be effortlessly extended to general metric spaces. On the other, it can be shown, at least in Euclidean spaces, that the global regularity properties arise as a consequence of the local structure of the measure, as it is clear from the following classical proposition, see e.g., \cite[Theorem 16.7]{Mattila1995GeometrySpaces}.

\begin{proposizione}\label{prop:loctoglob}
Assume $\phi$ is a Radon measure on $\R^n$ and $k$ is a natural number such that $1\leq k\leq n$. Then, $\phi$ is a $k$-rectifiable measure if and only if  for $\phi$-almost every $x\in\mathbb{R}^n$ we have
\begin{itemize}
     \item[(i)]$0<\Theta^k_*(\phi,x)\leq\Theta^{k,*}(\phi,x)<+\infty$,
    \item[(ii)]$\mathrm{Tan}_k(\phi,x) \subseteq \{\lambda\mathcal{H}^k\llcorner V:\lambda>0,\,\,\text{and $V$ is a $k$-dimensional vector subspace}\}$,
\end{itemize}
where $\Theta^k_*(\phi,x)$ and $\Theta^{k,*}(\phi,x)$ are, respectively, the lower and the upper $k$-density of $\phi$ at $x$, see \cref{def:densities}, and $\mathrm{Tan}_k(\phi,x)$ is the set of $k$-tangent measures to $\phi$ at $x$, see \cref{def:TangentMeasure}, while $\mathcal{H}^k$ is the Hausdorff measure.
\end{proposizione}

As mentioned above one can define rectifiable measures in arbitrary metric spaces: however one quickly understands that there are some limitations to what the classical rectifiability can achieve.

The first example of this is the curve in $L^1([0,1])$ that at each $t\in [0,1]$ assigns the indicator function of the interval $[0,t]$. This curve is Lipschitz continuous, however it fails to be Fr\'echet differentiable at every point of $[0,1]$, thus not admitting a tangent. This shows that we cannot expect anything like \cref{prop:loctoglob} to hold in infinite dimension.

For the second example we need to briefly introduce Carnot groups, see \cref{sec:Prel} for details. A Carnot group $\mathbb{G}$ is a simply connected nilpotent Lie group, whose Lie algebra is stratified and generated by its first layer. Carnot groups are a generalization of Euclidean spaces, and we remark that (quotients of) Carnot groups arise as the infinitesimal models of sub-Riemannian manifolds and their geometry, even at an infinitesimal scale, might be very different from the Euclidean one. We endow $\mathbb G$ with an arbitrary left-invariant homogeneous distance $d$, and we recall that any two of them are bi-Lipschitz equivalent. These groups have finite Hausdorff dimension, that is commonly denoted by $Q$, and any Lipschitz map $f:\R^{Q-1}\to(\mathbb{G},d)$ has $\mathcal{H}^{Q-1}$-null image, unless $\mathbb G$ is an Euclidean space, see for instance \cite{AK00} and \cite[Theorem 1.1]{MagnaniUnrect}. This from an Euclidean perspective means that there are no Lipschitz regular parametrized one-codimensional surfaces inside $(\mathbb{G},d)$, unless $\mathbb G$ is an Euclidean space. However, as shown in the fundational papers \cite{Serapioni2001RectifiabilityGroup, step2}, in Carnot groups there is a well defined notion of finite perimeter set and in Carnot groups of step 2 their reduced boundary can be covered up to $\mathcal{H}^{Q-1}$-negligible sets by countably many \textbf{intrinsic} $C^1$-regular hypersurfaces, $C^1_{\mathrm{H}}$ hypersurfaces from now on, see \cite[Definition 6.4]{Serapioni2001RectifiabilityGroup}.
The success of the approach attempted in \cite{Serapioni2001RectifiabilityGroup} has started an effort to study Geometric Measure Theory in sub-Riemannian Carnot groups, and in particular to study various notion(s) of rectifiability, see, e.g., \cite{step2, FSSC03b, FSSC06, FSSC07, MagnaniJEMS, MatSerSC, FSSC11, FMS14, DLDMV19, LDM20, JNGV20, Merlo1, MarstrandMattila20, IMM20, DDFO20, JNGV21Area}.
The big effort represented by the aformentioned papers in trying to understand rectifiability in Carnot groups has given rise to a multiplication of definitions, each one suiting some particular cases.

As we shall see in the subsequent paragraphs, not only one could consider our approach reversed with respect to the ones known in the literature but it also has a twofold advantage. On the one hand the definition of $\mathscr{P}$-rectifiable measure is natural and intrinsic with respect to the (homogeneous) structure of Carnot groups and it is equivalent to the usual one in the Euclidean setting; on the other hand we do not have to handle the problem of distinguishing, in the definition, between the low-dimensional and the low-codimensional rectifiability. 

Nevertheless, for arbitrary Carnot groups, we prove non-trivial structure results for rectifiable measures, see \cref{mainresultsintro}. These structure results will be used to prove the main Theorem of this paper, see \cref{coroll:DENSITYIntro}. Moreover, the study of the structure properties proved in \cref{mainresultsintro} is completed in the subsequent paper \cite{antonelli2020rectifiable2}. 

In a companion paper \cite{antonelli2020rectifiableB}, which roughly corresponds to an elaboration of Section 5 of the Preprint \cite{antonelli2020rectifiable}, we shall prove a Marstrand--Mattila type rectifiability criterion for $\mathscr{P}$-rectifiable measures that in turn will lead to the proof of the one-dimensional Preiss's theorem for the first Heisenberg group $\mathbb H^1$ endowed with the Koranyi distance. 

\vspace{0.2cm}
\textit{Additional remark.} The present work consists of an elaboration of Sections 2, 3, 4, and 6 of the Preprint \cite{antonelli2020rectifiable}. This is the first of two companion papers derived from \cite{antonelli2020rectifiable}. The second one is an elaboration of Sections 2, and 5 of \cite{antonelli2020rectifiable}. We stress that the results in Sections 2, 3, 4, and 6 of \cite{antonelli2020rectifiable} do not use the results in Section 5 of \cite{antonelli2020rectifiable}. As a result this paper can be read fully and independently from its companion paper.

\subsection{$\mathscr{P}$-rectifiable measures}\label{sub:P}
In this paper we study structure results in the class of $\mathscr{P}$-rectifiable measures, which have been introduced in \cite[Definition 3.1 \& Definition 3.2]{MarstrandMattila20}. Let $\mathbb G$ be a Carnot group of Hausdorff dimension $Q$.

\begin{definizione}[$\mathscr{P}$-rectifiable measures]\label{def:PhRectifiableMeasureINTRO}
Let $1\leq h\leq Q$ be a natural number. A Radon measure $\phi$ on $\mathbb G$ is said to be {\em $\mathscr{P}_h$-rectifiable} if for $\phi$-almost every $x\in \mathbb{G}$ we have \begin{itemize}
    \item[(i)]$0<\Theta^h_*(\phi,x)\leq\Theta^{h,*}(\phi,x)<+\infty$,
    \item[(\hypertarget{due}{ii})]$\mathrm{Tan}_h(\phi,x) \subseteq \{\lambda\mathcal{S}^h\llcorner \mathbb V(x):\lambda\geq 0\}$, where $\mathbb V(x)$ is a homogeneous subgroup of $\mathbb G$ of Hausdorff dimension $h$,
\end{itemize}
where $\Theta^h_*(\phi,x)$ and $\Theta^{h,*}(\phi,x)$ are, respectively, the lower and the upper $h$-density of $\phi$ at $x$, see \cref{def:densities}, $\mathrm{Tan}_h(\phi,x)$ is the set of $h$-tangent measures to $\phi$ at $x$, see \cref{def:TangentMeasure}, and $\mathcal{S}^h$ is the spherical Hausdorff measure of dimension $h$, see \cref{def:HausdorffMEasure}. Furthermore, we say that $\phi$ is {\em $\mathscr{P}_h^*$-rectifiable} if (\hyperlink{due}{ii})  is replaced with the weaker
\begin{equation*}
    \mathrm{(ii)^*}\,\, \mathrm{Tan}_h(\phi,x) \subseteq \{\lambda \mathcal{S}^h\llcorner \mathbb V: ~\lambda \geq 0, \text{ $\mathbb V$ is a homogeneous } \text{subgroup}\text{  of $\mathbb G$ of Hausdorff dimension $h$}\}.
\end{equation*}
If we impose more regularity on the tangents we can define different subclasses of $\mathscr{P}$-rectifiable or $\mathscr{P}^*$-rectifiable measures, see \cref{def:SubclassesPh} for details. We notice that, a posteriori, in the aformentioned definitions we can and will restrict to $\lambda>0$, see \cref{rem:AboutLambda=0}.
\end{definizione}

The definition of $\mathscr{P}$-rectifiable measure is natural in the setting of Carnot groups. Indeed, we have on $\mathbb G$ a family of dilations $\{\delta_{\lambda}\}_{\lambda>0}$, see \cref{sec:Prel}, that we can use to give a good definition of blow-up of a measure. Hence we ask, for a measure to be rectifiable, that the tangents are \textbf{flat}. The natural class of flat spaces, i.e., the analogous of vector subspaces of the Euclidean space, seems to be the class of homogeneous subgroups of $\mathbb G$. This latter assertion is suggested also from the result in \cite[Theorem 3.2]{Mattila2005MeasuresGroups} according to which on every locally compact group $G$ endowed with dilations and isometric left translations, if a Radon measure $\mu$ has a unique (up to multiplicative constants) tangent $\mu$-almost everywhere then this tangent is $\mu$-almost everywhere (up to multiplicative constants) the left Haar measure on a closed dilation-invariant subgroup of $G$. As a consequence, in the definition of $\mathscr{P}_h$-rectifiable measure we can equivalently substitute item (ii) of \cref{def:PhRectifiableMeasureINTRO} with the weaker requirement
$$
\mathrm{(ii)'} \,\, \mathrm{Tan}_h(\phi,x)\subseteq \{\lambda\nu_x:\lambda>0\},\,\,\text{where $\nu_x$ is a Radon measure on $\mathbb G$}.
$$
Moreover, we stress that if a metric group is locally compact, isometrically homogeneous and admits one dilation, as it is for the class of metric group studied in \cite{Mattila2005MeasuresGroups}, and moreover the distance is geodesic, then it is a sub-Finsler Carnot group, see \cite[Theorem 1.1]{EnricoCarnot}. 

As already mentioned, according to one of the approaches to rectifiability in Carnot groups, the good parametrizing objects for the notion of rectifiability are \textbf{$C^1_{\mathrm{H}}$-regular surfaces with complemented tangents} in $\mathbb G$, i.e., sets that are locally the zero-level sets of $C^1_{\mathrm H}$ functions $f$ - see \cref{def:C1h} - with surjective Pansu differential $df$, and such that $\mathrm{Ker}(df)$ admits a complementary subgroup in $\mathbb G$. This approach has been taken to its utmost level of generality through the works \cite{MagnaniJEMS, MagnaniTwoardsDiff, JNGV20}. In particular in \cite[Definition 2.18]{JNGV20} the authors give the most general, and available up to now, definition of $(\mathbb G,\mathbb G')$-rectifiable sets, see \cref{def:C1Hmanifold} and \cref{def:GG'Rect}, and they prove area and coarea formulae within this class of rectifiable sets. We stress that an improvement of the area formula in \cite{JNGV20} is obtained by the two authors of this work in \cite[Theorem 1.3]{antonelli2020rectifiable2}. Related results are in \cite{CM20}.

We remark that our definition of rectifiability is strictly weaker than the one in \cite{JNGV20}, see \cref{prop:GG'IsPh} and \cref{oss:LackOfGenerality}. Moreover for discussions on the converse of the following \cref{prop:GG'Intro} we refer the reader to \cref{rem:C1GG'PhQuandoPossono}. We stress that, as a result of the subsequent work \cite[Corollary 5.3]{antonelli2020rectifiable2}, at least in the co-horizontal setting, the notion of $\mathscr{P}$-rectifiable measure and the notion of rectifiability given in terms of $(\mathbb G,\mathbb G')$-rectifiable sets coincide.

\begin{proposizione}\label{prop:GG'Intro}
Let us fix $\mathbb G$ and $\mathbb G'$ two arbitrary Carnot groups of homogeneous dimensions $Q$ and $Q'$ respectively. Let us take $\Sigma\subseteq \mathbb G$ a $(\mathbb G,\mathbb G')$-rectifiable set. Then $\mathcal{S}^{Q-Q'}\llcorner \Sigma$ is a $\mathscr{P}_{Q-Q'}$-rectifiable measure with complemented tangents. Moreover, there exists $\mathbb G$ a Carnot group, $\Sigma\subseteq \mathbb G$, and $1\leq h\leq Q$ such that $\mathcal{S}^h\llcorner \Sigma$ is a $\mathscr{P}_h$-rectifiable measure and, for every Carnot group $\mathbb G'$, $\Sigma$ \textbf{is not} $(\mathbb G,\mathbb G')$-rectifiable.
\end{proposizione}

Let us stress that the second part of \cref{prop:GG'Intro} is not surprising. Indeed, the approach to rectifiability described above and used in \cite{JNGV20} is selecting rectifiable sets whose tangents are \textbf{complemented normal} subgroups of $\mathbb G$, see \cite[Section 2.5]{JNGV20} for a more detailed discussion. This can be easily understood if one thinks that, according to this approach to rectifiability, the parametrizing class of objects is given by $C^1_{\mathrm{H}}$-regular surfaces $\Sigma$ with complemented tangents $\mathrm{Ker}(df_p)$ at $p\in\Sigma$, which are complemented (and normal) subgroups.

In some sense we could say that the approach of \cite{JNGV20} is covering, in the utmost generality known up to now, the case of low-codimensional rectifiable sets in a Carnot group $\mathbb G$. It has been clear since the works \cite{FSSC07, MatSerSC} that, already in the Heisenberg groups $\mathbb H^n$, one should approach the low-dimensional rectifiability in a different way with respect to the low-codimensional one. Indeed, in the low-dimensional case in $\mathbb H^n$, the authors in \cite{FSSC07, MatSerSC} choose as a parametrizing class of objects the images of $C^1_{\mathrm{H}}$-regular (or Lipschitz-regular) functions from subsets of $\mathbb R^d$ to $\mathbb H^n$, with $1\leq d\leq n$, see \cite[Definition 3.1 \& Definition 3.2]{FSSC07}, and \cite[Definition 2.10 and Definition 3.13]{MatSerSC}. 

The bridge between the definition of $\mathscr{P}$-rectifiability and the ones disscused above  is done in \cite{MatSerSC} \textbf{in the setting of Heisenberg groups} and in \cite{IMM20} in arbitrary homogeneous groups but in the case of \textbf{horizontal tangents}. Let us stress that the result in \cite[(i)$\Leftrightarrow$(iv) of Theorem 3.15]{MatSerSC} shows that in the Heisenberg groups the $\mathscr{P}$-rectifiability with tangents that are vertical subgroups is equivalent to the rectifiability given in terms of $C^1_{\mathrm H}$-regular surfaces. Moreover \cite[(i)$\Leftrightarrow$(iv) of Theorem 3.14]{MatSerSC} shows that in the Heisenberg groups the $\mathscr{P}$-rectifiability with tangents that are horizontal subgroups is equivalent to the rectifiability given in terms of Lipschitz-regular images. 

Moreover, very recently, in \cite[Theorem 1.1]{IMM20}, the authors prove a generalization of \cite[Theorem 3.14]{MatSerSC} in arbitrary homogeneous groups. Namely they prove that in a homogeneous group the $k$-rectifiability of a set in the sense of Federer can be characterized with the fact that the tangent measures to the set are horizontal subgroups, or equivalently with the fact that there exists an approximate tangent plane that is a horizontal subgroup almost everywhere. In our setting this implies that the $\mathscr{P}$-rectifiability with tangents that are horizontal subgroups is equivalent to the rectifiability given in terms of Lipschitz-regular images, which is Federer's one. We observe that in the subsequent paper \cite{antonelli2020rectifiable2} we shall exploit the results proved in this paper and we shall further develop the theory of $\mathscr{P}$-rectifiable measures thus obtaining generalizations of \cite[Theorem 3.14 and 3.15]{MatSerSC} in arbitrary Carnot groups and in all dimensions. See the introduction of \cite{antonelli2020rectifiable2}, and \cite[Theorem 1.1]{antonelli2020rectifiable2}. For results similar to the ones of \cite{MatSerSC, IMM20, antonelli2020rectifiable2} but in the different setting of the parabolic $\mathbb R^n$ and in all the codimensions, we point out the recent \cite{MattilaParabolic}.

We stress that the previous results leave open the challenging question of understanding what is the precise structure of a measure $\phi$ on $\mathbb H^1$ such that the tangents are $\phi$-almost everywhere the vertical line. 

\subsection{Results}\label{mainresultsintro}

The main contribution of this paper is the proof of the fact that a $\mathscr{P}_h$-rectifiable measure with complemented tangents has density, see \cref{coroll:DENSITY}, and \cref{prop:density} for the last part of the following statement. We recall that when we say that a homogeneous subgroup $\mathbb V$ of a Carnot group $\mathbb G$ {\em admits a complementary subgroup}, we mean that there exists a homogeneous subgroup $\mathbb L$ such that $\mathbb G=\mathbb V\cdot \mathbb L$ and $\mathbb V\cap\mathbb L=\{0\}$. 

\begin{teorema}[Existence of the density]\label{coroll:DENSITYIntro}
Let $\phi$ be a $\mathscr{P}_h$-rectifiable measure \textbf{with complemented tangents} on $\mathbb G$, and assume $d$ is a homogeneous left-invariant metric on $\mathbb{G}$. Let $B(x,r)$ be the closed metric ball relative to $d$ of centre $x$ and radius $r$. Then, for $\phi$-almost every $x\in\mathbb G$ we have
$$
0<\liminf_{r\to 0}\frac{\phi(B(x,r))}{r^h}=\limsup_{r\to 0}\frac{\phi(B(x,r))}{r^h}<+\infty.
$$
Moreover, 
for $\phi$-almost every $x\in\mathbb G$ we have
$$
r^{-h}T_{x,r}\phi \rightharpoonup \Theta^h(\phi,x)\mathcal{C}^h\llcorner\mathbb V(x),\qquad \text{as $r$ goes to $0$},
$$
where the map $T_{x,r}$ is defined in \cref{def:TangentMeasure}, the convergence is understood in the duality with the continuous functions with compact support on $\mathbb G$, $\Theta^h(\phi,x)$ is the $h$-density with respect to the distance $d$, and $\mathcal{C}^h\llcorner\mathbb V(x)$ is the $h$-dimensional centered Hausdorff measure, with respect to the distance $d$, restricted to $\mathbb V(x)$, see \cref{def:HausdorffMEasure}.
\end{teorema}

A way of reading the previous theorem is the following: we prove that whenever a Radon measure on a Carnot group has strictly positive $h$-lower density and finite $h$-upper density, and at almost every point all the blow-up measures are supported on the same (depending on the point) $h$-dimensional homogeneous complemented subgroup, then the measure has $h$-density. 

We observe here that, as a non-trivial consequence of the results that will be developed in \cite{antonelli2020rectifiable2}, see \cite[Theorem 1.1]{antonelli2020rectifiable2}, we have that whenever $\Gamma\subseteq \mathbb G$ is a Borel set such that $0<\mathcal{S}^h(\Gamma)<+\infty$, and $\mathcal{C}^h\llcorner\Gamma$ is $\mathscr{P}_h$-rectifiable with complemented tangents, then  $\Theta^h(\mathcal{C}^h\llcorner\Gamma,x)=1$ for $\mathcal{C}^h$-almost every $x\in\Gamma$. See \cref{rem:DensityOne}. We remark that the fact that $\mathcal{C}^h\llcorner\Gamma$ has density one is not a straightforward consequence of \cref{coroll:DENSITYIntro}, and it requires additional work, cf. \cite[Proposition 3.9]{antonelli2020rectifiable2}.

Let us remark that the previous \cref{coroll:DENSITYIntro} solves the implication (ii)$\Rightarrow$(i) of the density problem formulated in \cite[page 50]{MarstrandMattila20} in the setting of $\mathscr{P}_h$-rectifiable measures with complemented tangents. In Euclidean spaces the proof of Theorem \ref{coroll:DENSITYIntro} is an almost immediate consequence of the fact that projections on linear spaces are $1$-Lipschitz in conjunction with the area formula. In our context we do not have at our disposal the Lipschitz property of projections and an area formula for $\mathscr{P}_h$-rectifiable measures with complemented tangents is obtained in \cite[Theorem 1.2]{antonelli2020rectifiable2} as a consequence of a more refined study of such measures. So the proof require new ideas. In order to obtain Theorem \ref{coroll:DENSITYIntro} first of all one reduces to the case of the surface measure on an intrinsically Lipschitz graph with very small Lipschitz constant thanks to the structure result \cref{coroll:IdiffMeasureIntroNEW} below. Secondly, one needs to show that the surface measures of the tangents and their push-forward on the graph are mutually absolutely continuous. For this last point to hold it will be crucial on the one hand that a $\mathscr P_h$-rectifiable measure with complemented tangents can be covered almost everyhwere with \textbf{intrinsic graphs}, see the forthcoming \cref{coroll:IdiffMeasureIntroNEW}, and on the other hand that intrinsic Lipschitz graphs have big projections on their bases, see \cref{prop:proj}. Third, one exploits the fact that the density exists for the surface measures on the tangents to infer its existence for the original measure. 

We remark that with the tools developed in the subsequent paper \cite{antonelli2020rectifiable2} and pushing forward the study started in this paper, we shall show an area formula for $\mathscr{P}_h$-rectifiable measures with complemented tangents, see \cite[Theorem 1.2 and Theorem 1.3]{antonelli2020rectifiable2}. 

Other contributions of this paper are structure results for $\mathscr P$-rectifiable measures. Since they will be given in terms of sets that satisfy a cone property, let us clarify which cones we are choosing. For any  $\alpha>0$ and any homogeneous subgroup $\mathbb V$ of $\mathbb G$, the {\em cone} $C_{\mathbb V}(\alpha)$ is the set of points $w\in\mathbb G$ such that $\dist(w,\mathbb V)\leq \alpha\|w\|$, where $\|\cdot\|$ is the homogeneous norm relative to the fixed distance $d$ on $\mathbb G$. Moreover a set $E\subseteq \mathbb G$ is a {\em $C_{\mathbb V}(\alpha)$-set} if $
E\subseteq pC_{\mathbb V}(\alpha)$ for every $p\in E$. We refer the reader to \cref{sub:Cones} for such definitions and some properties of them. We stress that the cones $C_{\mathbb V}(\alpha)$ are used to give the definition of intrinsically Lispchitz graphs and functions, see \cite[Definition 11 and Proposition 3.1]{FranchiSerapioni16}. The first result reads as follows, see \cref{thm:MainTheorem2}.

\begin{teorema}\label{thm:MainTheorem1Intro}
Let $\mathbb G$ be a Carnot group endowed with an arbitrary left-invariant homogeneous distance. Let $\phi$ be a $\mathscr{P}_h$-rectifiable measure on $\mathbb G$. Then $\mathbb G$ can be covered $\phi$-almost everywhere with countably many compact sets with the cone property with arbitrarily small opening. In other words for every $\alpha>0$ we have
$$
\phi\bigg(\mathbb G \setminus \bigcup_{i=1}^{+\infty} \Gamma_i\bigg)=0,
$$
where $\Gamma_i$ are compact $C_{\mathbb V_i}(\alpha)$-sets, where $\mathbb V_i$ are homogeneous subgroups of $\mathbb G$ of Hausdorff dimension $h$.
\end{teorema}
If we ask that the tangents are \textbf{complemented} subgroups, we can improve the previous result. In particular we can take the $\Gamma_i$'s to be \textbf{intrinsic Lipschitz graphs}, see \cref{thm:MainTheorem1} and \cref{prop:ConeAndGraph}. For the definition of intrinsically Lipschitz function, we refer the reader to \cref{def:iLipfunctions}. Let us remark that the fact that the $\Gamma_i$'s can be taken to be graphs will be crucial for the proof of the existence of the density in \cref{coroll:DENSITYIntro}. Actually, by pushing a little bit further the information about the fact that the tangent measures at $\phi$-almost every $x$ are constant multiples of $\mathcal{S}^h\llcorner\mathbb V(x)$, we can give a structure result within the class of intrinsically differentiable graphs. Roughly speaking we say that the graph of a function between complementary subgroups $\varphi:U\subseteq \mathbb V\to\mathbb L$ is intrinsically differentiable at $a_0\cdot\varphi(a_0)$ if $\mathrm{graph}(\varphi)$ admits a homogeneous subgroup as Hausdorff tangent at $a_0\cdot\varphi(a_0)$, see \cref{defiintrinsicdiffgraph} for details. For the forthcoming theorem, see \cref{coroll:IdiffMeasure}.

\begin{teorema}\label{coroll:IdiffMeasureIntroNEW}
Let $\mathbb G$ be a Carnot group of homogeneous dimension $Q$ endowed with an arbitrary left-invariant homogeneous distance. Let $h\in\{1,\dots,Q\}$, and let $\phi$ be a $\mathscr{P}_h^c$-rectifiable, i.e., a $\mathscr{P}_h$-rectifiable measure with tangents that are complemented almost everywhere. 

Then $\mathbb G$ can be covered $\phi$-almost everywhere with countably many compact graphs that are simultaneously intrinsically Lipschitz with arbitrarily small constant, and intrinsically differentiable almost everywhere. In other words, for every $\alpha>0$, we can write
$$
\phi\bigg(\mathbb G \setminus \bigcup_{i=1}^{+\infty} \Gamma_i\bigg)=0,
$$
where $\Gamma_i=\mathrm{graph}(\varphi_i)$ are compact sets, with $\varphi_i:A_i\subseteq \mathbb V_i\to\mathbb L_i$ being a function between a compact subset $A_i$ of $\mathbb V_i$, which is a homogeneous subgroup of $\mathbb G$ of homogeneous dimension $h$, and $\mathbb L_i$, which is a subgroup complementary to $\mathbb V_i$; in addition $\mathrm{graph}(\varphi_i)$ is a $C_{\mathbb V_i}(\alpha)$-set, and it is an intrinsically differentiable graph at $a\cdot\varphi_i(a)$ for $\mathcal{S}^h\llcorner A_i$-almost every $a\in\mathbb V_i$, see \cref{defiintrinsicdiffgraph}. 
\end{teorema}


Let us briefly remark that when a Rademacher-type theorem holds, i.e., if an intrinsically Lipschitz function is intrinsically differentiable almost everywhere, the full result in \cref{coroll:IdiffMeasureIntroNEW} would simply be deduced by the analogous result but only requiring a covering with intrinsic Lipschitz graphs. We remark that a Rademacher-type theorem at such level of generality, i.e., between arbitrary complementary subgroups of a Carnot group, is now known to be false, see the counterexample in \cite{JNGV20a}. On the other hand, some positive results in particular cases have been provided in \cite{FSSC11, FMS14, AM20} for intrinsically Lipschitz functions with one-dimensional target in groups in which De Giorgi $C^1_{\mathrm H}$-rectifiability for finite perimeter sets holds, and for functions with normal targets in arbitrary Carnot groups. We stress that very recently in \cite{Vittone20} the author proves the Rademacher theorem at any codimension in the Heisenberg groups $\mathbb H^n$.

We stress that, as a consequence of the result \cite[Theorem 1.1]{antonelli2020rectifiable2}, we get that the $\mathscr{P}_h$-rectifiability of measures of the type $\mathcal{S}^h\llcorner\Gamma$ is equivalent to the fact that $\mathcal{S}^h$-almost every $\Gamma$ is covered by countably many intrinsic differentiable graphs. Thus the negative result of \cite{JNGV20a} gives as a consequence that we cannot substitute \textit{intrinsic differentiable} with \textit{intrinsic Lipschitz} in the latter sentence. Ultimately, the general notion of rectifiability by means of coverings with countably many intrinsic Lipschitz graphs is not equivalent to the infinitesimal notion of rectifiability (namely, the $\mathscr{P}$-rectifiability) that one can give by asking that the tangents are almost everywhere unique (and then, as a consequence, homogeneous subgroups).

Let us briefly comment on the results listed above.
 \cref{coroll:DENSITYIntro} extends the implication in \cite[(iv)$\Rightarrow$(ii) of Theorem 3.15]{MatSerSC} to the setting of $\mathscr{P}_h$-rectifiable measures whose tangents are complemented in arbitrary Carnot groups. Indeed, in \cite[(iv)$\Rightarrow$(ii) of Theorem 3.15]{MatSerSC} the authors prove that if $n+1\leq h\leq 2n$, and $\mathcal{S}^h\llcorner \Gamma$ is a $\mathscr{P}_h$-rectifiable measure with tangents that are vertical subgroups in the Heisenberg group $\mathbb H^n$, then the $h$-density of $\mathcal{S}^h\llcorner \Gamma$ exists almost everywhere and the tangent is unique almost everywhere. The analogous property in $\mathbb H^n$, but with $\mathscr{P}_h$-rectifiable measures with tangents that are horizontal subgroups is obtained in \cite[(iv)$\Rightarrow$(ii) of Theorem 3.14]{MatSerSC}, and in arbitrary homogeneous groups in the recent \cite[(iii)$\Rightarrow$(ii) of Theorem 1.1]{IMM20}. However, in this special horizontal case treated in \cite[Theorem 3.14]{MatSerSC} and \cite[Theorem 1.1]{IMM20} the authors do not assume $\Theta_*^h(\mathcal{S}^h\llcorner\Gamma,x)>0$ since it comes from the existence of an approximate tangent, see \cite[Theorem 3.10]{MatSerSC}, while the authors in \cite{IMM20} are able to overcome this issue by adapting \cite[Lemma 3.3.6]{Federer1996GeometricTheory} in \cite[Theorem 4.4]{IMM20}. We do not address in this paper the question of obtaining the same general results as in \cref{thm:MainTheorem1Intro}, \cref{coroll:DENSITYIntro}, and \cref{coroll:IdiffMeasureIntroNEW} removing the hypothesis on the strictly positive lower density in item (i) of \cref{def:PhRectifiableMeasureINTRO} when the tangent is unique (up to a mutiplicative constant). Neverthless we stress that the results obtained in \cite{MatSerSC, IMM20} are for sets, while our results hold for arbitrary Radon measures.

We finally mention that, as a consequence of the machinery developed in the subsequent paper \cite{antonelli2020rectifiable2}, the covering property with intrinsically differentiable graphs proved in \cref{coroll:IdiffMeasureIntroNEW} actually characterizes the $\mathscr{P}_h$-rectifiability with complemented tangents, see \cite[3. $\Rightarrow$ 1. of Theorem 1.1]{antonelli2020rectifiable2}. 
The previous characterization is in fact obtained through a delicate rectifiability results for intrinsically differentiable graphs, see \cite[Theorem 1.3]{antonelli2020rectifiable2}, and more precisely \cite[Lemma 3.9 and Lemma 3.10]{antonelli2020rectifiable2}. Moreover, as a consequence of \cite[Theorem 1.1]{antonelli2020rectifiable2}, the $\mathscr{P}_h$-rectifiability with complemented tangents is equivalent to asking that Preiss's tangents are complemented homogeneous subgroups without any requirement on the $h$-lower and upper densities. We refer the reader to \cite{antonelli2020rectifiable2} for details.

\vspace{0.3cm}
\textbf{Acknowledgments}: The first author is partially supported by the European Research Council (ERC Starting Grant 713998 GeoMeG `\emph{Geometry of Metric Groups}'). The second author is supported by the Simons Foundation Wave Project, grant 601941, GD. The authors wish tho express their gratitude to the anonymous referee for the careful reading of the paper. 

\medskip

\section{Preliminaries}\label{sec:Prel}
\subsection{Carnot Groups}\label{sub:Carnot}
In this subsection we briefly introduce some notations on Carnot groups that we will extensively use throughout the paper. For a detailed account on Carnot groups we refer to \cite{LD17}.

A Carnot group $\mathbb{G}$ of step $\kappa$ \label{num:step} is a simply connected Lie group whose Lie algebra $\mathfrak g$ admits a stratification $\mathfrak g=V_1\, \oplus \, V_2 \, \oplus \dots \oplus \, V_\kappa$. We say that $V_1\, \oplus \, V_2 \, \oplus \dots \oplus \, V_\kappa$ is a {\em stratification} of $\mathfrak g$ if $\mathfrak g = V_1\, \oplus \, V_2 \, \oplus \dots \oplus \, V_\kappa$,
$$
[V_1,V_i]=V_{i+1}, \quad \text{for any $i=1,\dots,\kappa-1$}, \quad  \quad [V_1,V_\kappa]=\{0\}, \quad \text{and}\,\, V_\kappa\neq\{0\},
$$ 
where $[A,B]:=\mathrm{span}\{[a,b]:a\in A,b\in B\}$. We call $V_1$ the \emph{horizontal layer} of $\mathbb G$. We denote by $n$ the topological dimension of $\mathfrak g$, by $n_j$ the dimension of $V_j$ for every $j=1,\dots,\kappa$.
Furthermore, we define $\pi_i:\mathbb{G}\to V_i$ to be the projection maps on the $i$-th strata. 
We will often shorten the notation to $v_i:=\pi_iv$.

For a Carnot group $\mathbb G$, the exponential map $\exp :\mathfrak g \to \mathbb{G}$ is a global diffeomorphism from $\mathfrak g$ to $\mathbb{G}$.
Hence, if we choose a basis $\{X_1,\dots , X_n\}$ of $\mathfrak g$,  any $p\in \mathbb{G}$ can be written in a unique way as $p=\exp (p_1X_1+\dots +p_nX_n)$. This means that we can identify $p\in \mathbb{G}$ with the $n$-tuple $(p_1,\dots , p_n)\in \R^n$ and the group $\mathbb{G}$ itself with $\R^n$ endowed with the group operation $\cdot$ determined by the Baker-Campbell-Hausdorff formula. From now on, we will always assume that $\mathbb{G}=(\R^n,\cdot)$ and, as a consequence, that the exponential map $\exp$ acts as the identity.

For any $p\in \mathbb{G}$, we define the left translation $\tau _p:\mathbb{G} \to \mathbb{G}$ as
\begin{equation*}
q \mapsto \tau _p q := p\cdot q.
\end{equation*}
As already remarked above, the group operation $\cdot$ is determined by the Campbell-Hausdorff formula, and it has the form (see \cite[Proposition 2.1]{step2})
\begin{equation*}
p\cdot q= p+q+\mathscr{Q}(p,q), \quad \mbox{for all }\, p,q \in  \R^n,
\end{equation*} 
where $\mathscr{Q}=(\mathscr{Q}_1,\dots , \mathscr{Q}_\kappa):\R^n\times \R^n \to V_1\oplus\ldots\oplus V_\kappa$, and the $\mathscr{Q}_i$'s have the following properties. For any $i=1,\ldots \kappa$ and any $p,q\in \mathbb{G}$ we have\label{tran}
\begin{itemize}
    \item[(i)]$\mathscr{Q}_i(\delta_\lambda p,\delta_\lambda q)=\lambda^i\mathscr{Q}_i(p,q)$,
    \item[(ii)] $\mathscr{Q}_i(p,q)=-\mathscr{Q}_i(-q,-p)$,
    \item[(iii)] $\mathscr{Q}_1=0$ and $\mathscr{Q}_i(p,q)=\mathscr{Q}_i(p_1,\ldots,p_{i-1},q_1,\ldots,q_{i-1})$.
\end{itemize}
Thus, we can represent the product $\cdot$ as
\begin{equation}\label{opgr}
p\cdot q= (p_1+q_1,p_2+q_2+\mathscr{Q}_2(p_1,q_1),\dots ,p_\kappa +q_\kappa+\mathscr{Q}_\kappa (p_1,\dots , p_{\kappa-1} ,q_1,\dots ,q_{\kappa-1})). 
\end{equation}

The stratificaton of $\mathfrak{g}$ carries with it a natural family of dilations $\delta_\lambda :\mathfrak{g}\to \mathfrak{g}$, that are Lie algebra automorphisms of $\mathfrak{g}$ and are defined by\label{intrdil}
\begin{equation}
     \delta_\lambda (v_1,\dots , v_\kappa):=(\lambda v_1,\lambda^2 v_2,\dots , \lambda^\kappa v_\kappa), \quad \text{for any $\lambda>0$},
     \nonumber
\end{equation}
where $v_i\in V_i$. The stratification of the Lie algebra $\mathfrak{g}$  naturally induces a gradation on each of its homogeneous Lie sub-algebras $\mathfrak{h}$, i.e., sub-algebras that are $\delta_{\lambda}$-invariant for any $\lambda>0$, that is
\begin{equation}
    \mathfrak{h}=V_1\cap \mathfrak{h}\oplus\ldots\oplus V_\kappa\cap \mathfrak{h}.
    \label{eq:intr1}
\end{equation}
We say that $\mathfrak h=W_1\oplus\dots\oplus W_{\kappa}$ is a {\em gradation} of $\mathfrak h$ if $[W_i,W_j]\subseteq W_{i+j}$ for every $1\leq i,j\leq \kappa$, where we mean that $W_\ell:=\{0\}$ for every $\ell > \kappa$.
Since the exponential map acts as the identity, the Lie algebra automorphisms $\delta_\lambda$ can be read also as group automorphisms of $\mathbb{G}$.

\begin{definizione}[Homogeneous subgroups]\label{homsub}
A subgroup $\mathbb V$ of $\mathbb{G}$ is said to be \emph{homogeneous} if it is a Lie subgroup of $\mathbb{G}$ that is invariant under the dilations $\delta_\lambda$.
\end{definizione}

We recall the following basic terminology: a {\em horizontal subgroup} of a Carnot group $\mathbb G$ is a homogeneous subgroup of it that is contained in $\exp(V_1)$; a {\em Carnot subgroup} $\mathbb W=\exp(\mathfrak h)$ of a Carnot group $\mathbb G$ is a homogeneous subgroup of it such that the first layer $V_1\cap\mathfrak h$ of the grading of $\mathfrak h$ inherited from the stratification of $\mathfrak g$ is the first layer of a stratification of $\mathfrak h$.

Homogeneous Lie subgroups of $\mathbb{G}$ are in bijective correspondence through $\exp$ with the Lie sub-algebras of $\mathfrak{g}$ that are invariant under the dilations $\delta_\lambda$. 
For any Lie algebra $\mathfrak{h}$ with gradation $\mathfrak h= W_1\oplus\ldots\oplus W_{\kappa}$, we define its \emph{homogeneous dimension} as
$$\text{dim}_{\mathrm{hom}}(\mathfrak{h}):=\sum_{i=1}^{\kappa} i\cdot\text{dim}(W_i).$$
Thanks to \eqref{eq:intr1} we infer that, if $\mathfrak{h}$ is a homogeneous Lie sub-algebra of $\mathfrak{g}$, we have $\text{dim}_{\mathrm{hom}}(\mathfrak{h}):=\sum_{i=1}^{\kappa} i\cdot\text{dim}(\mathfrak{h}\cap V_i)$. We introduce now the class of homogeneous and left-invariant distances.

\begin{definizione}[Homogeneous left-invariant distance]
A metric $d:\mathbb{G}\times \mathbb{G}\to \R$ is said to be homogeneous and left invariant if for any $x,y\in \mathbb{G}$ we have
\begin{itemize}
    \item[(i)] $d(\delta_\lambda x,\delta_\lambda y)=\lambda d(x,y)$ for any $\lambda>0$,
    \item[(ii)] $d(\tau_z x,\tau_z y)=d(x,y)$ for any $z\in \mathbb{G}$.
\end{itemize}
\end{definizione}

We remark that two homogeneous left-invariant distances on a Carnot group are always bi-Lipschitz equivalent.
It is well-known that the Hausdorff dimension (for a definition of Hausdorff dimension see for instance \cite[Definition 4.8]{Mattila1995GeometrySpaces}) of a graded Lie group $\mathbb G$ with respect to an arbitrary left-invariant homogeneous distance coincides with the homogeneous dimension of its Lie algebra. For a reference for the latter statement, see \cite[Theorem 4.4]{LDNG19}. \textbf{From now on, if not otherwise stated, $\mathbb G$ will be a fixed Carnot group}. We recall that a \emph{homogeneous norm} $\|\cdot\|$ on $\mathbb G$ is a function $\|\cdot\|:\mathbb G\to [0,+\infty)$ such that $\|\delta_\lambda x\|=\lambda\|x\|$ for every $\lambda>0$ and $x\in\mathbb G$; $\|x\cdot y\|\leq \|x\|+\|y\|$ for every $x,y\in\mathbb G$; and $\|x\|=0$ if and only if $x=0$. We introduce now a distinguished homogeneous norm on $\mathbb G$.

\begin{definizione}[Smooth-box metric]\label{smoothnorm}
For any $g\in \mathbb{G}$, we let
$$\lVert g\rVert:=\max\{\varepsilon_1\lvert g_1\rvert,\varepsilon_2\lvert g_2\rvert^{1/2},\ldots, \varepsilon_{\kappa}\lvert g_\kappa\rvert^{1/{\kappa}}\},$$
where $\varepsilon_1=1$ and $\varepsilon_2,\ldots \varepsilon_{\kappa}$ are suitably small parameters depending only on the group $\mathbb{G}$. For the proof of the fact that we can choose the $\varepsilon_i$'s in such a way that $\lVert\cdot\rVert$ is a homogeneous norm on $\mathbb{G}$ that induces a left-invariant homogeneous distance we refer to \cite[Section 5]{step2}. 
\end{definizione}
Given an arbitrary homogeneous norm $\|\cdot\|$ on $\mathbb G$, the distance $d$ induced by $\|\cdot\|$ is defined as follows
$$
d(x,y):=\lVert x^{-1}\cdot y\rVert.
$$
Vice-versa, given a homogeneous left-invariant distance $d$, it induces a homogeneous norm through the equality $\|x\|:=d(x,e)$ for every $x\in\mathbb G$, where $e$ is the identity element of $\mathbb G$.

Given a homogeneous left-invariant distance $d$ we let $U(x,r):=\{z\in \mathbb{G}:d(x,z)<r\}$ be the open metric ball relative to the distance $d$ centred at $x$ and with radius $r>0$. The closed ball will be denoted with $B(x,r):=\{z\in \mathbb{G}:d(x,z)\leq r\}$. Moreover, for a subset $E\subseteq \mathbb G$ and $r>0$, we denote with $B(E,r):=\{z\in\mathbb G:\dist(z,E)\leq r\}$ the {\em closed $r$-tubular neighborhood of $E$}  and with $U(E,r):=\{z\in\mathbb G:\dist(z,E)< r\}$ the {\em open $r$-tubular neighborhood of $E$}.

The following estimate on the norm of the conjugate will be useful later on.
\begin{lemma}\label{lem:EstimateOnConjugate}
For any homogeneous norm $\|\cdot\|$ and any $\ell>0$, there exists a constant $\newC\label{c:1}(\ell)>1$ such that for every $x,y\in B(0,\ell)$ we have
$$
\|y^{-1}\cdot x\cdot y\| \leq \oldC{c:1}\|x\|^{1/\kappa}.
$$
\end{lemma}

\begin{proof}
This follows immediately from \cite[Lemma 3.12]{FranchiSerapioni16}.
\end{proof}

\begin{definizione}[Hausdorff Measures]\label{def:HausdorffMEasure}
Throughout the paper we define the $h$-dimensional {\em spherical Hausdorff measure} relative to a left invariant homogeneous metric $d$ as\label{sphericaldhausmeas}
$$
\mathcal{S}^{h}(A):=\sup_{\delta>0}\inf\bigg\{\sum_{j=1}^\infty  r_j^h:A\subseteq \bigcup_{j=1}^\infty B(x_j,r_j),~r_j\leq\delta\bigg\},
$$
for every $A\subseteq \mathbb G$.
We define the $h$-dimensional {\em Hausdorff measure}\label{hausmeas} relative to $d$ as
$$
\mathcal{H}^h(A):=\sup_{\delta>0}\inf \left\{\sum_{j=1}^{\infty} 2^{-h}(\diam E_j)^h:A \subseteq \bigcup_{j=1}^{\infty} E_j,\, \diam E_j\leq \delta\right\},
$$
for every $A\subseteq \mathbb G$.
We define the $h$-dimensional {\em centered Hausdorff measure} relative to $d$ as\label{centredhausmeas}
$$
\mathcal{C}^{h}(A):=\underset{E\subseteq A}{\sup}\,\,\mathcal{C}_0^h(E),
$$
for every $A\subseteq \mathbb G$, 
where
$$
\mathcal{C}^{h}_0(E):=\sup_{\delta>0}\inf\bigg\{\sum_{j=1}^\infty  r_j^h:E\subseteq \bigcup_{j=1}^\infty B(x_j,r_j),~ x_j\in E,~r_j\leq\delta\bigg\},
$$
for every $E\subseteq \mathbb G$.
We stress that $\mathcal{C}^h$ is an outer measure, and thus it defines a Borel regular measure, see \cite[Proposition 4.1]{EdgarCentered}, and that the measures $\mathcal{S}^h,\mathcal{H}^h,\mathcal{C}^h$ are all equivalent measures, see \cite[Section 2.10.2]{Federer1996GeometricTheory} and \cite[Proposition 4.2]{EdgarCentered}.
\end{definizione}

\begin{definizione}[Hausdorff distance]\label{def:Haus}
Given a left-invariant homogeneous distance $d$ on $\mathbb G$, for any couple of sets $A,B\subseteq \mathbb{G}$, we define the \emph{Hausdorff distance} of $A$ from $B$ as
$$d_{H,\mathbb G}(A,B):=\max\Big\{\sup_{x\in A}\text{dist}(x,B),\sup_{y\in B}\text{dist}(A,y)\Big\},$$
where 
$$
\text{dist}(x,A):=\inf_{y\in A} d(x,y),
$$
for every $x\in\mathbb G$ and $A\subseteq \mathbb G$.
\end{definizione}

\subsection{Densities and tangents of Radon measures}

Throughout the rest of  the paper we will always assume that $\mathbb{G}$ is a fixed Carnot group endowed with an arbitrary left-invariant homogeneous distance $d$. Some of the forthcoming results will be proved in the particular case in which $d$ is the distance induced by the distinguished homogeneous norm defined in \cref{smoothnorm}, and we will stress this when it will be the case. 

The homogeneous, and thus Hausdorff, dimension with respect to $d$ will be denoted with $Q$. Furthermore as discussed in the previous subsection, we will assume without loss of generality that $\mathbb{G}$ coincides with $\R^n$ endowed with the product induced by the Baker-Campbell-Hausdorff formula relative to $\text{Lie}(\mathbb{G})$.

\begin{definizione}[Weak convergence of measures]\label{def:WeakConvergence}
Given a family $\{\phi_i\}_{i\in\N}$ of Radon measures on $\mathbb{G}$ we say that $\phi_i$ weakly converges to a Radon measure $\phi$, and we write $\phi_i\rightharpoonup \phi$, if
$$
\int fd \phi_i \to \int fd\phi, \qquad\text{for any } f\in C_c(\mathbb G).
$$
\end{definizione}

\begin{definizione}[Tangent measures]\label{def:TangentMeasure}
Let $\phi$ be a Radon measure on $\mathbb G$. For any $x\in\mathbb G$ and any $r>0$ we define the measure
$$
T_{x,r}\phi(E):=\phi(x\cdot\delta_r(E)), \qquad\text{for any Borel set }E.
$$
Furthermore, we define $\mathrm{Tan}_{h}(\phi,x)$, the $h$-dimensional tangents to $\phi$ at $x$, to be the collection of the Radon measures $\nu$ for which there is an infinitesimal sequence $\{r_i\}_{i\in\N}$ such that
$$r_i^{-h}T_{x,r_i}\phi\rightharpoonup \nu.$$
\end{definizione}
\begin{osservazione}(Zero as a tangent measure)\label{rem:TangentZero}
We remark that our definition potentially admits the zero measure as a tangent measure, as in \cite{DeLellis2008RectifiableMeasures}, while the definitions in \cite{Preiss1987GeometryDensities} and \cite{MatSerSC} do not.
\end{osservazione}

\begin{definizione}[Lower and upper densities]\label{def:densities}
If $\phi$ is a Radon measure on $\mathbb{G}$, and $h>0$, we define
$$
\Theta_*^{h}(\phi,x):=\liminf_{r\to 0} \frac{\phi(B(x,r))}{r^{h}},\qquad \text{and}\qquad \Theta^{h,*}(\phi,x):=\limsup_{r\to 0} \frac{\phi(B(x,r))}{r^{h}},
$$
and we say that $\Theta_*^{h}(\phi,x)$ and $\Theta^{h,*}(\phi,x)$ are the lower and upper $h$-density of $\phi$ at the point $x\in\mathbb{G}$, respectively. Furthermore, we say that measure $\phi$ has $h$-density if
$$
0<\Theta^h_*(\phi,x)=\Theta^{h,*}(\phi,x)<\infty,\qquad \text{for }\phi\text{-almost any }x\in\mathbb{G}.
$$
\end{definizione}

Lebesgue theorem holds for measures with positive lower density and finite upper density, and thus local properties are stable under restriction to Borel subsets.

\begin{proposizione}\label{prop:Lebesuge}
Suppose $\phi$ is a Radon measure on $\mathbb{G}$ with $0<\Theta^h_*(\phi,x)\leq \Theta^{h,*}(\phi,x)<\infty $ for $\phi$-almost every $x\in \mathbb{G}$. Then, for any Borel set $B\subseteq \mathbb{G}$ and for $\phi$-almost every $x\in B$ we have
$$\Theta^h_*(\phi\llcorner B,x)=\Theta^h_*(\phi,x),\qquad \text{and}\qquad\Theta^{h,*}(\phi\llcorner B,x)=\Theta^{h,*}(\phi,x).$$
\end{proposizione}

\begin{proof}
This is a direct consequence of Lebesgue differentiation Theorem of \cite[page 77]{HeinonenKoskelaShanmugalingam}, that can be applied since $(\mathbb G,d,\phi)$ is a Vitali metric measure space due to \cite[Theorem 3.4.3]{HeinonenKoskelaShanmugalingam}.
\end{proof}

We stress that whenever the $h$-lower density of $\phi$ is stricly positve and the $h$-upper density of $\phi$ is finite $\phi$-almost everywhere, the set $\mathrm{Tan}_h(\phi,x)$ is nonempty for $\phi$-almost every $x\in\mathbb G$, see \cite[Proposition 1.12]{MarstrandMattila20}. The following proposition has been proved in \cite[Proposition 1.13]{MarstrandMattila20}.

\begin{proposizione}[Locality of tangents]\label{prop:LocalityOfTangent}
Let $h>0$, and let $\phi$ be a Radon measure such that for $\phi$-almost every $x\in\mathbb G$ we have
$$
0<\Theta^h_*(\phi,x)\leq \Theta^{h,*}(\phi,x)<\infty.
$$
Then for every $\rho\in L^1(\phi)$ that is nonnegative $\phi$-almost everywhere we have
$\mathrm{Tan}_h(\rho\phi,x)=\rho(x)\mathrm{Tan}_h(\phi,x)$ for $\phi$-almost every $x\in\mathbb G$. More precisely the following holds: for $\phi$-almost every $x\in\mathbb G$ then
\begin{equation}
    \begin{split}
        \text{if $r_i\to 0$ is such that}\quad r_i^{-h}T_{x,r_i}\phi \rightharpoonup \nu\quad \text{then}\quad r_i^{-h}T_{x,r_i}(\rho\phi)\rightharpoonup \rho(x)\nu.  \\
    \end{split}
\end{equation}
\end{proposizione}

Let us introduce a useful split of the support of a Radon measure $\phi$ on $\mathbb G$. 
\begin{definizione}\label{def:EThetaGamma}
Let $\phi$ be a Radon measure on $\mathbb G$ that is supported on the compact set $K$. For any $\vartheta,\gamma\in\N$ we define
\begin{equation}
    E(\vartheta,\gamma):=\big\{x\in K:\vartheta^{-1}r^h\leq \phi(B(x,r))\leq \vartheta r^h\text{ for any }0<r<1/\gamma\big\}.
    \label{eq:A1}
\end{equation}
\end{definizione}

Let us stress that \cref{def:EThetaGamma} does not only depend on $\vartheta,\gamma$ but obviously also on $h$. Anyway throughout the proofs of this paper we always assume $h$ to be fixed, and hence we will not stress this dependence in the notation $E(\vartheta,\gamma)$.

\begin{proposizione}\label{prop:cpt}
For any $\vartheta,\gamma\in\N$, the set $E(\vartheta,\gamma)$ defined in \cref{def:EThetaGamma} is compact.
\end{proposizione}

\begin{proof}
This is \cite[Proposition 1.14]{MarstrandMattila20}.
\end{proof}

\begin{proposizione}\label{prop::E}
Assume $\phi$ is a Radon measure supported on the compact set $K$ such that $
0<\Theta^h_*(\phi,x)\leq \Theta^{h,*}(\phi,x)<\infty
$ for $\phi$-almost every $x\in\mathbb G$. Then $\phi(\mathbb{G}\setminus \bigcup_{\vartheta,\gamma\in\N} E(\vartheta,\gamma))=0$.
\end{proposizione}

\begin{proof}
Let $w\in K\setminus \bigcup_{\vartheta,\gamma} E(\vartheta,\gamma)$ and note that this implies that either $\Theta^h_*(\phi,x)=0$ or $\Theta^{h,*}(\phi,x)=\infty$. Since $
0<\Theta^h_*(\phi,x)\leq \Theta^{h,*}(\phi,x)<\infty
$ for $\phi$-almost every $x\in\mathbb G$, this concludes the proof.
\end{proof}

We recall here a useful proposition about the structure of Radon measures.
\begin{proposizione}[{\cite[Proposition 1.17 and Corollary 1.18]{MarstrandMattila20}}]\label{prop:MutuallyEthetaGamma}
Let $\phi$ be a Radon measure supported on a compact set on $\mathbb{G}$ such that $0<\Theta^h_*(\phi,x)\leq \Theta^{h,*}(\phi,x)<\infty $ for $\phi$-almost every $x\in \mathbb{G}$. For every $\vartheta,\gamma\in \mathbb N$ we have that $\phi\llcorner E(\vartheta,\gamma)$ is mutually absolutely continuous with respect to $\mathcal{S}^h\llcorner E(\vartheta,\gamma)$.

\end{proposizione}

\subsection{Intrinsic Grassmannian in Carnot groups}

Let us recall the definition of the Euclidean Grassmannian, along with some of its properties.
\begin{definizione}[Euclidean Grassmannian]\label{def:EuclideanGrassmannian}
Given $k\leq n$ we let $\G(n,k)$ to be the set of the $k$-vector subspaces of $\mathbb R^n$. We endow $\G(n,k)$ with the following distance
$$
d_{\mathrm{eu}}(V_1,V_2):=d_{H,\mathrm{eu}}\left(V_1\cap B_{\mathrm{eu}}(0,1),V_2\cap B_{\mathrm{eu}}(0,1)\right),
$$
where $B_{\mathrm{eu}}(0,1)$ is the (closed) Euclidean unit ball, and $d_{H,\mathrm{eu}}$ is the Hausdorff distance between sets induced by the Euclidean distance on $\R^n$.
\end{definizione}

\begin{osservazione}[Euclidean Grassmannian and convergence]\label{rem:CompactnessEuclideanGrassmannianAndApproximatingVn}
It is well-known that the metric space $(\G(n,k),d_{H,\mathrm{eu}})$ is compact. 
Moreover, the following hold
\begin{enumerate}
    \item[(i)] if $V_n\to V$, then for every $v\in V$ there exist $v_n\in V_n$ such that $v_n\to v$;
    \item[(ii)] if $V_n\to V$ and there is a sequence $v_n\in V_n$ such that $v_n\to v$, then $v\in V$.
\end{enumerate}
The proof of the two items above is left to the reader as an exercise. 
\end{osservazione}

We now give the definition of the intrinsic Grassmannian on Carnot groups and introduce the class of complemented homogeneous subgroups.

\begin{definizione}[Intrinsic Grassmanian on Carnot groups]\label{def:Grassmannian}
For any $1\leq h\leq Q$, we define $\G(h)$ to be the family of homogeneous subgroups $\mathbb V$ of $\mathbb{G}$ that have Hausdorff dimension $h$. 

Let us recall that if $\mathbb{V}$ is a homogeneous subgroup of $\mathbb{G}$, any other homogeneous subgroup $\mathbb L$ such that
$$
\mathbb{V}\cdot\mathbb{L}=\mathbb{G}\qquad \text{and}\qquad \mathbb{V}\cap \mathbb{L}=\{0\}.
$$
is said to be a \emph{complement} of $\mathbb{G}$. We let $\G_c(h)$ to be the subfamily of those $\mathbb V\in\G(h)$ that have a complement and we will refer to $\G_c(h)$ as the $h$-dimensional \emph{complemented} Grassmanian. 
\end{definizione}
Let us introduce the stratification vector of a homogeneous subgroup.
\begin{definizione}[Stratification vector]\label{def:stratification}
Let $h\in\{1,\ldots,Q\}$ and for any $\mathbb{V}\in\G(h)$ we denote with $\mathfrak{s}(\mathbb{V})$ the vector
$$
\mathfrak{s}(\mathbb{V}):=(\dim(V_1\cap\mathbb{V}),\ldots, \dim(V_\kappa\cap \mathbb{V})),
$$
that with abuse of language we call the {\em stratification}, or the {\em stratification vector}, of $\mathbb{V}$. Furthermore, we define
$$
\mathfrak{S}(h):=\{\mathfrak{s}(\mathbb{V})\in\N^\kappa:\mathbb{V}\in \G(h)\}.
$$
We remark that the cardinality of $\mathfrak{S}(h)$ is bounded by $\prod_{i=1}^\kappa (\dim V_i +1)$ for any $h\in\{1,\ldots,Q\}$.
\end{definizione}

We now collect in the following result some topological properties of the Grassmanians introduced above.

\begin{proposizione}[Compactness of the Grassmannian]\label{prop:CompGrassmannian}
For any $1\leq h\leq Q$ the function
$$
d_{\mathbb G}(\mathbb W_1,\mathbb W_2):=d_{H,\mathbb G}(\mathbb W_1\cap B(0,1),\mathbb W_2\cap B(0,1)),
$$
with $\mathbb W_1,\mathbb W_2\in \G(h)$, is a distance on $\G(h)$. Moreover $(\G(h),d_{\mathbb G})$ is a compact metric space. 
\end{proposizione}

\begin{proof}
The fact that $d_{\mathbb G}$ is a distance comes from well-known properties of the Hausdorff distance. 
Let us consider a sequence $\{\mathbb W_j\}_{j\in\N}\subseteq \G(h)$, with $\mathbb W_j=W_{j,1}\oplus\dots\oplus W_{j,\kappa}$, where $W_{j,i}:=V_i\cap \mathbb{W}_j$ for any $j\in\N$ and $1\leq i\leq \kappa$. By extracting a (non re-labelled) subsequence we can suppose that there exist $\{k_i\}_{i=1,\dots,\kappa}$ natural numbers such that the topological dimension is $\dim W_{j,i}=k_i$ for all $j\in\mathbb N$, and for all $1\leq i \leq \kappa$. In particular the topological dimension of $\mathbb W_j$ is constant. 
Exploiting the compactness of the Euclidean Grassmannian, see \cref{rem:CompactnessEuclideanGrassmannianAndApproximatingVn}, we get that up to a (non re-labelled) subsequence, 
\begin{equation}\label{eqn:PartialConvergence}
W_{j,i}\to W_i, \quad \mbox{i.e.} \quad d_{\mathrm{eu}}(W_{j,i},W_i)\to0 \quad \text{for any } 1\leq i\leq \kappa,
\end{equation}
where the convergence is meant in the Euclidean Grassmannian $\G(k_i,V_i)$. As a consequence
\begin{equation}\label{eqn:TotalConvergence}
W_j=W_{j,1}\oplus\dots\oplus W_{j,\kappa} \to W=W_1\oplus\dots\oplus W_\kappa, \quad \mbox{i.e.,} \quad d_{H,\mathrm{eu}}(W_j,W)\to 0, 
\end{equation}
where the convergence is meant in the Euclidean Grassmannian $\G(\sum_{i=1}^\kappa k_i,n)$. The previous equality is a consequence of \eqref{eqn:PartialConvergence} and the following observation: if $V$ and $W$ are two orthogonal linear subspaces such that $\mathbb R^n=V\oplus W$, and $A,B$ are vector subspaces of $V$, and $C,D$ are vector subspaces of $W$, then $$d_{\mathrm{eu}}(A\oplus C,B\oplus D)\leq d_{\mathrm{eu}}(A,B)+d_{\mathrm{eu}}(C,D),$$ where the direct sums above are orthogonal too. 
Let us notice that, from \eqref{eqn:TotalConvergence} it follows that
\begin{equation}\label{eqn:ConvergenceToW}
d_{H,\mathrm{eu}}(W_j\cap B(0,1),W\cap B(0,1))\to 0,
\end{equation}
where we stress that $B(0,1)$ is the closed unit ball in the homogeneous left-invariant metric $d$. The proof of \eqref{eqn:ConvergenceToW} can be reached by contradiction exploiting \eqref{eqn:TotalConvergence} and the fact that $B(0,1)$ is compact. We leave the routine details to the reader.

In order to conclude the proof, we need to show that 
\begin{equation}\label{eqn:ConvergenceToW2}
d_{\mathbb G}(W_j\cap B(0,1),W\cap B(0,1))\to 0.
\end{equation}
Indeed, on the compact set $B(0,1)$, one has $d\leq Cd_{\mathrm{eu}}^{1/\kappa}$ for some constant $C>0$, see for instance \cite[Proposition 2.15]{SC16}. This means that for subsets contained in $B(0,1)$ one has $d_{H}\leq Cd_{H,\mathrm{eu}}^{1/s}$. This last inequality with \eqref{eqn:ConvergenceToW} gives \eqref{eqn:ConvergenceToW2}. Finally from \eqref{eqn:ConvergenceToW2} we get, by the very definition of $d_{\mathbb G}$,
$$
d_{\mathbb G}(\mathbb W_j,\mathbb W)\to 0.
$$
If we show that $\mathbb W$ is a homogeneous subgroup of homogeneous dimension $h$ we are done. The homogeneity comes from the fact that $W$ admits a stratification \eqref{eqn:TotalConvergence}, while the homogeneous dimension is fixed because it depends on the dimensions of $W_i$ that are all equal to $k_i$. Let us prove $\mathbb W$ is a subgroup. First of all $\mathbb W$ is inverse-closed, because $\mathbb W=\exp W$, and $W$ is a vector space. Now take $a,b\in \mathbb W$. By the first point of \cref{rem:CompactnessEuclideanGrassmannianAndApproximatingVn} we find $a_n,b_n\in \mathbb W_n$ such that $a_n\to a$, and $b_n\to b$. Then, by continuity of the operation, $a_n\cdot b_n \to a\cdot b$, and $a_n\cdot b_n\in \mathbb W_n$. Then from the second point of \cref{rem:CompactnessEuclideanGrassmannianAndApproximatingVn} we get that $a\cdot b\in \mathbb W$.
\end{proof}

\begin{proposizione}\label{prop:htagliato}
There exists a constant $\hbar_\mathbb{G}>0$, depending only on $\mathbb{G}$, such that if $\mathbb{W},\mathbb{V}\in\G(h)$ and $d_{\mathbb{G}}(\mathbb{V},\mathbb{W})\leq \hbar_{\mathbb{G}}$, then
$\mathfrak{s}(\mathbb{V})=\mathfrak{s}(\mathbb{W})$.
\end{proposizione}

\begin{proof}
Let us fix $1\leq h\leq Q$. Let us suppose by contradiction that there exist $\mathbb V_i$ and $\mathbb W_i$ in $\G(h)$ such that, for every $i\in\mathbb N$, the stratification of $\mathbb V_i$ is different from $\mathbb W_i$ and such that $d_{\mathbb{G}}(\mathbb V_i,\mathbb W_i)\to 0$. Up to extract two (non re-labelled) subsequences we can assume that the $\mathbb V_i$'s have the same stratification for every $i\in \mathbb N$, as well as the $\mathbb W_i$'s. Then, by compactness, see the proof of \cref{prop:CompGrassmannian}, we can assume up to passing to a (non re-labelled) subsequence that $\mathbb W_i\to \mathbb W$ where $\mathbb W$ has the same stratification of the $\mathbb W_i$'s, and $\mathbb V_i\to \mathbb V$ where $\mathbb V$ has the same stratification of the $\mathbb V_i$'s. Since $d_{\mathbb{G}}(\mathbb V_i,\mathbb W_i)\to 0$ we get that $d_{\mathbb{G}}(\mathbb V,\mathbb W)=0$ and then $\mathbb V=\mathbb W$ but this is a contradiction since they have different stratifications. This proves the existence of a constant $\hbar$ that depends both on $\mathbb{G}$ and $h$. However, taking the minimum over $h$ of such $\hbar$'s, the dependence on $h$ is eliminated. 
\end{proof}

\begin{proposizione}\label{prop:haar}
    Suppose $\mathbb{V}\in\G(h)$ is a homogeneous subgroup of topological dimension $d$. Then $\mathcal{S}^h\llcorner \mathbb{V}$, $\mathcal{H}^h\llcorner\mathbb{V}$, $\mathcal{C}^h\llcorner\mathbb{V}$ and $\mathcal{H}^d_{\mathrm{eu}}\llcorner \mathbb{V}$ are Haar measures of $\mathbb{V}$. Furthermore, any Haar measure $\lambda$ of $\mathbb{V}$ is $h$-homogeneous in the sense that
    $$\lambda(\delta_r(E))=r^h\lambda(E),\qquad\text{for any Borel set }E\subseteq \mathbb{V}.$$
\end{proposizione}

\begin{proof}
This follows from the fact that the Hausdorff, the  spherical Hasudorff, and the centered Hausdorff measures introduced in \cref{def:Haus} are invariant under left-translations and thus on the one hand they are Haar measures of $\mathbb{V}$. Furthermore, one can show by an explicit computation that the Lebesgue measure $\mathcal{L}^d$ restricted to the vector space $\exp^{-1}(\mathbb{V})$ is a Haar measure. Indeed, this last assertion comes from the fact that for every $v\in\mathbb V$ the map $p\to v\cdot p$ has unitary Jacobian determinant when seen as a map from $\mathbb V$ to $\mathbb V$, see \cite[Lemma 2.20]{FranchiSerapioni16}. Thus since when seen $\mathbb{V}$ as immersed in $\R^n$ we have that the Lebesgue measure of $\mathbb{V}$ coincides with $\mathcal{H}^d_{\mathrm{eu}}\llcorner \mathbb{V}$, we conclude that $\mathcal{H}^d_{\mathrm{eu}}\llcorner \mathbb{V}$ is a Haar measure of $\mathbb{V}$ as well. The last part of the proposition comes from the fact that the property is obvious by definition for the spherical Hausdorff measure, and the fact that all the Haar measures are the same up to a constant.
\end{proof}

We now introduce the projections related to a splitting $\mathbb G=\mathbb V\cdot\mathbb L$ of the group.

\begin{definizione}[Projections related to a splitting]\label{def:Projections}
For any $\mathbb V\in \G_c(h)$ with a homogeneous complement $\mathbb L$, we can find two unique elements $g_{\mathbb V}:=P_\mathbb V g\in \mathbb V$ and $g_{\mathbb L}:=P_{\mathbb L}g\in \mathbb L$ such that
$$
g=P_\mathbb V (g)\cdot P_{\mathbb L}(g)=g_{\mathbb V}\cdot g_{\mathbb L}.
$$
We will refer to $P_{\mathbb V}(g)$ and $P_{\mathbb L}(g)$ as the \emph{splitting projections}, or simply {\em projections}, of $g$ onto $\mathbb V$ and $\mathbb L$, respectively.
\end{definizione}
We recall here below a very useful fact on splitting projections.
\begin{proposizione}\label{prop:InvarianceOfProj}
Let us fix $\mathbb V\in\G_c(h)$ and $\mathbb L$ two complementary homogeneous subgroups of a Carnot group $\mathbb G$. Then, for any $x\in\mathbb{G}$ the map $\Psi:\mathbb{V}\to\mathbb{V}$ defined as $\Psi(z):=P_\mathbb{V}(xz)$ is invertible and it has unitary Jacobian. As a consequence $\mathcal{S}^h(P_{\mathbb V}(\mathcal{E}))=\mathcal{S}^h(P_{\mathbb V}(xP_{\mathbb V}(\mathcal{E})))=\mathcal{S}^h(P_{\mathbb V}(x\mathcal{E}))$ for every $x\in \mathbb G$ and $\mathcal{E}\subseteq \mathbb G$ Borel.
\end{proposizione}
\begin{proof}
The first part is a direct consequence of \cite[Proof of Lemma 2.20]{FranchiSerapioni16}. For the second part it is sufficient to use the first part and the fact that for every $x,y\in\mathbb G$ we have $P_{\mathbb V}(xy)=P_{\mathbb V}(xP_{\mathbb V}y)$.
\end{proof}
The following proposition holds for the distance $d$ induced by the norm introduced in \cref{smoothnorm}.
\begin{proposizione}\label{prop:TopDimMetricDim}
Let $\mathbb G$ be a Carnot group endowed with the homogeneous norm $\|\cdot\|$ introduced in \cref{smoothnorm}. Let $\mathbb W\in\G(h)$ be a homogeneous subgroup of Hausdorff dimension $h$ and of topological dimension $d$. Then
\begin{enumerate}
\item[(i)] there exists a constant $\newC\label{c:2}:=\oldC{c:2}(\mathfrak{s}(\mathbb W))$ such that for any $p\in \mathbb{W}$ and any $r>0$ we have
\begin{equation}
    \mathcal{H}_{\mathrm{eu}}^{d}\left(B(p,r)\cap \mathbb W\right)=\oldC{c:2}r^{h},
\end{equation}
\item[(ii)] there exists a constant $\beta(\mathbb W)$ such that $    \mathcal{C}^{h}\llcorner \mathbb W =\beta(\mathbb{W})\mathcal{H}_{\mathrm{eu}}^{d}\llcorner \mathbb W$,
\item[(iii)] $\beta(\mathbb{W})=\mathcal{H}^d_{\mathrm{eu}}\llcorner \mathbb{W}(B(0,1))^{-1}$ and in particular $\beta(\mathbb{W})=\beta(\mathfrak{s}(\mathbb{W}))$.
\end{enumerate}
\end{proposizione}

\begin{proof}
Thanks to \cref{prop:haar}, we have
$$
\mathcal{H}^d_{\mathrm{eu}}(B(p,r)\cap \mathbb{W})=\mathcal{H}^d_{\mathrm{eu}}(B(0,r)\cap \mathbb{W})=\mathcal{H}^d_{\mathrm{eu}}(\delta_r(B(0,1)\cap \mathbb{W}))=r^h\mathcal{H}^d_{\mathrm{eu}}(B(0,1)\cap \mathbb{W}).$$
Furthermore, if $\mathbb{V}$ is another homogeneous subgroup such that $\mathfrak{s}(\mathbb{W})=\mathfrak{s}(\mathbb{V})$, we can find a linear map $T$ that acts as an orthogonal transformation on each of the $V_i$'s and that maps $\mathbb{W}$ to $\mathbb{V}$. Since we are endowing $\mathbb G$ with the box metric \cref{smoothnorm}, we get that $T(B(0,1)\cap \mathbb W)=B(0,1)\cap\mathbb V$. Since $T$ is an orthogonal transformation itself, it is an isometry of $\R^n$ and this implies that
$$\mathcal{H}^d_{\mathrm{eu}}(B(0,1)\cap \mathbb{W})=\mathcal{H}^d_{\mathrm{eu}}(T(B(0,1)\cap \mathbb{W}))=\mathcal{H}^d_{\mathrm{eu}}(B(0,1)\cap \mathbb{V}).$$

Concerning (ii) thanks to \cref{prop:haar} we have that both $\mathcal{C}^h\llcorner \mathbb{W}$ and $\mathcal{H}^d_{\mathrm{eu}}\llcorner \mathbb{W}$ are Haar measures of $\mathbb{W}$. This implies that there must exist a constant $\beta(\mathbb{W})$ such that $\beta(\mathbb{W})\mathcal{H}^d_{\mathrm{eu}}\llcorner \mathbb{W}=\mathcal{C}^h\llcorner \mathbb{W}$. 

 Finally, in order to prove (iii), we prove the following. For every left-invariant homogeneous distance $d$ on $\mathbb G$ and every homogeneous subgroup $\mathbb W\subseteq\mathbb G$ of Hausdorff dimension $h$, we have that 
 \begin{equation}\label{eqn:Ch1}
 \mathcal{C}^h(\mathbb W\cap B(0,1))=1,
 \end{equation}
 where $\mathcal{C}^h$ is the centered Hausdorff measure relative to the distance $d$ and $B(0,1)$ is the closed ball relative to the distance $d$.
 
 Indeed, let us fix an $\varepsilon>0$, let us take $A\subseteq \mathbb W\cap B(0,1)$ such that 
$\mathcal{C}^h_0(A)\geq \mathcal{C}^h(\mathbb W\cap B(0,1))-\varepsilon$, $\delta>0$ and a covering of $A$ with closed balls $B_i:=\{B(x_i,r_i)\}_{i\in\N}$ centred on $A\subseteq \mathbb{W}$ and with radii $r_i\leq \delta$ such that
\begin{equation}
    \sum_{i\in\N} r_i^h\leq \mathcal{C}^h_0(A)+\varepsilon.
    \nonumber
\end{equation}
This implies that
\begin{equation}
\begin{split}
        &\mathcal{C}^h(B(0,1)\cap \mathbb{W})\big(\mathcal{C}^h(B(0,1)\cap\mathbb W)+\varepsilon\big)\geq \mathcal{C}^h(B(0,1)\cap \mathbb{W})\big(\mathcal{C}^h_0(A)+\varepsilon\big) \\
        &\geq \sum_{i\in\N}\mathcal{C}^h(B(0,1)\cap \mathbb{W})r_i^h=\sum_{i\in\N}\mathcal{C}^h(B(x_i,r_i)\cap \mathbb{W})\geq  \mathcal{C}^h(A) \\
        &\geq \mathcal{C}^h_0(A)\geq \mathcal{C}^h(\mathbb W\cap B(0,1))-\varepsilon,
        \nonumber
\end{split}
\end{equation}
where the first inequality is true since $\mathcal{C}^h(B(0,1)\cap\mathbb W)\geq \mathcal{C}^h(A)\geq \mathcal{C}^h_0(A)$, and the third equality is true since $x_i\in\mathbb W$ and $\mathcal{C}^h\llcorner\mathbb W$ is a Haar measure on $\mathbb W$.
Thanks to the arbitrariness of $\varepsilon$ we finally infer that $\mathcal{C}^h(\mathbb{W}\cap B(0,1))\geq 1$. 

On the other hand, thanks to \cite[item (ii) of Theorem 2.13 and Remark 2.14]{FSSCArea}, we have that, calling $B_t:=\{x\in\mathbb W\cap B(0,1): \Theta^{*,h}(\mathcal{C}^h\llcorner\mathbb W,x)>t\}$ for every $t>0$, we infer that $\mathcal{C}^h(B_t)\geq t\mathcal{C}^h(B_t)$ for every $t>0$. Thus, for every $t>1$ we conclude $\mathcal{C}^h(B_t)=0$ and hence for $\mathcal C^h\llcorner\mathbb W$-almost every $x\in\mathbb W\cap B(0,1)$ we have that $\Theta^{*,h}(\mathcal C^h\llcorner\mathbb W,x)\leq 1$. For one of such $x\in\mathbb W\cap B(0,1)$ we can write
$$
\mathcal{C}^h(B(0,1)\cap \mathbb{W})=\limsup_{r\to 0}\frac{\mathcal{C}^h(B(x,r)\cap \mathbb{W})}{r^h}=\Theta^{*,h}(\mathcal{C}^h\llcorner\mathbb W,x)\leq 1,
$$
where the first equality comes from \cref{prop:haar}. Thus $\mathcal{C}^h(\mathbb W\cap B(0,1))=1$ and this concludes the proof of the first part of (iii) thanks to item (ii). The fact that $\beta(\mathbb W)$ depends only on $\mathfrak{s}(\mathbb{W})$ follows from item (i)
\end{proof}

\begin{osservazione}
The above proposition can be proved whenever the distance is a \emph{multiradial distance}, see \cite[Definition 8.5]{MagnaniTowardArea}.
\end{osservazione}

\begin{osservazione}\label{rem:Ch1}
We stress here for future references that in the proof of item (iii) of \cref{prop:TopDimMetricDim} we proved that whenever $\mathbb G$ is endowed with an arbitrary left-invariant homogeneous distance $d$, then for  every homogeneous subgroup $\mathbb W\subseteq\mathbb G$ of Hausdorff dimension $h$, we have that 
 \begin{equation}
 \mathcal{C}^h(\mathbb W\cap B(0,1))=1.
 \end{equation}
\end{osservazione}

We conclude this subsection with two Propositions.

\begin{proposizione}[Corollary 2.15 of \cite{FranchiSerapioni16}]\label{prop:distvsproj} Let $\|\cdot\|$ be a homogeneous norm on $\mathbb G$ and let $\mathbb{V}$ and $\mathbb{L}$ be two complementary subgroups. Then there exists a constant $\newC\label{C:split}(\mathbb{V},\mathbb{L})$ such that for any $g\in\mathbb{G}$ we have
\begin{equation}\label{eqn:SeriuosProj}
\oldC{C:split}(\mathbb{V},\mathbb{L})\lVert P_\mathbb{L}(g)\rVert\leq \dist(g,\mathbb{V})\leq \lVert P_{\mathbb L} (g)\rVert, \qquad\text{for any } g\in\mathbb G.
\end{equation}
In the following, whenever we write $\oldC{C:split}(\mathbb V,\mathbb L)$, we are choosing the supremum of all the constants such that inequality \eqref{eqn:SeriuosProj} is satisfied.
\end{proposizione}

\begin{proposizione}\label{cor:2.2.19}
For any $\mathbb{V}\in \G_c(h)$ with complement $\mathbb L$ there is a constant $ \newC\label{ProjC}(\mathbb{V},\mathbb L)>0$ such that for any $p\in\mathbb{G}$ and any $r>0$ we have
$$\mathcal{S}^{h}\llcorner \mathbb V\big(P_\mathbb{V}(B(p,r))\big)=\oldC{ProjC}(\mathbb{V},\mathbb L)r^{h}.$$
Furthermore, for any Borel set $A\subseteq \mathbb{G}$ for which $\mathcal{S}^{h}(A)<\infty$, we have
\begin{equation}
    \mathcal{S}^{h}\llcorner \mathbb{V}(P_\mathbb{V}(A))\leq 2\oldC{ProjC}(\mathbb{V},\mathbb L)\mathcal{S}^{h}(A).
    \label{eq:n520}
\end{equation}
\end{proposizione}

\begin{proof}
The existence of such $\oldC{ProjC}(\mathbb V,\mathbb L)$ is yielded by \cite[Lemma 2.20]{FranchiSerapioni16}. 
Suppose $\{B(x_i,r_i)\}_{i\in\N}$ is a countable covering of $A$ with closed balls for which $\sum_{i\in\N} r_i^{h}\leq 2\mathcal{S}^{h}(A)$. Then
\begin{equation}
    \mathcal{S}^{h}(P_\mathbb{V}(A))\leq \mathcal{S}^{h}\Big(P_\mathbb{V}\Big(\bigcup_{i\in\N}B(x_i,r_i)\Big)\Big)\leq \oldC{ProjC}(\mathbb{V},\mathbb L)\sum_{i\in\N} r_i^{h}\leq 2 \oldC{ProjC}(\mathbb{V},\mathbb L)\mathcal{S}^{h}(A).
    \nonumber
\end{equation}
\end{proof}

\subsection{Cones over homogeneous subgroups}\label{sub:Cones}
In this subsection, we introduce the intrinsic cone $C_{\mathbb W}(\alpha)$ and the notion of $C_{\mathbb W}(\alpha)$-set, and prove some of their properties. In this subsection $\mathbb G$ will be a fixed Carnot group endowed with an arbitrary homogeneous norm $\|\cdot\|$ that induces a left-invariant homogeneous distance $d$.

\begin{definizione}[Intrinsic cone]\label{def:Cone}
For any  $\alpha>0$ and $\mathbb W\in \G(h)$, we define the cone $C_{\mathbb W}(\alpha)$ as
$$C_\mathbb W(\alpha):=\{w\in\mathbb{G}:\dist(w,\mathbb W)\leq \alpha\|w\|\}.$$
\end{definizione}

\begin{definizione}[$C_{\mathbb W}(\alpha)$-set]\label{def:CvAset}
Given $\mathbb W\in \G(h)$, and $\alpha>0$, we say that a set $E\subseteq \mathbb G$ is a {\em $C_{\mathbb W}(\alpha)$-set} if
$$
E\subseteq p\cdot C_{\mathbb W}(\alpha), \qquad  \text{for any } p\in E. 
$$
\end{definizione}

\begin{lemma}\label{lem:DistanceOfCones}
For any $\mathbb W_1,\mathbb W_2\in \G(h)$, $\varepsilon>0$ and $\alpha>0$ if $d_\mathbb{G}(\mathbb W_1,\mathbb W_2)<\varepsilon/4$, then
$$C_{\mathbb W_1}(\alpha) \subseteq C_{\mathbb W_2}(\alpha+\varepsilon).$$
\end{lemma}
\begin{proof}
We prove that any $z\in C_{\mathbb W_1}(\alpha)$ is contained in the cone $C_{\mathbb W_2}(\alpha+\varepsilon)$. Thanks to the triangle inequality we infer
$$
\text{dist}(z,\mathbb W_2) \leq d(z,b)+\inf_{w\in\mathbb W_2} d (b,w), \qquad \text{for any } b\in\mathbb W_1.
$$
Thus, choosing $b^*\in\mathbb W_1$ in such a way that $d(z,b^*)=\text{dist}(z,\mathbb W_1)$, and evaluating the previous inequality at $b^*$ we get 
\begin{equation}\label{eqn:Estimate1}
\dist(z,\mathbb W_2) \leq \dist(z,\mathbb W_1) +\text{dist}(b^*,\mathbb W_2) \leq \alpha\|z\| + \dist(b^*,\mathbb W_2), 
\end{equation}
where in the second inequality we used $z\in C_{\mathbb W_1}(\alpha)$. 

Let us notice that, given $\mathbb W$ an arbitrary homogeneous subgroup of $\mathbb G$, $p\in \mathbb G$ an arbitrary point such that $p^*\in\mathbb W$ is one of the points at minimum distance from $\mathbb W$ to $p$, then the following inequality holds
\begin{equation}\label{eqn:Estimate2}
\|p^*\|\leq 2\|p\|.
\end{equation}
Indeed,
$$
\|p^*\|-\|p\| \leq \|(p^*)^{-1}\cdot p\| = d(p,\mathbb W)\leq \|p\| \Rightarrow \|p^*\|\leq 2\|p\|.
$$

Now, by homogeneity, since $b^*\in\mathbb W_1$ is the point at minimum distance from $\mathbb W_1$ of $z$, we get that $D_{1/\|z\|}(b^*)$ is the point at minimum distance from $\mathbb W_1$ of $D_{1/\|z\|}(z)$. Thus, since $\|D_{1/\|z\|}(z)\|=1$, from \eqref{eqn:Estimate2} we get that $\|D_{1/\|z\|}(b^*)\|\leq 2$. Finally we obtain
\begin{equation}\label{eqn:Estimate3}
\begin{split}
\text{dist}(b^*,\mathbb W_2)&=\|z\|\text{dist}\big(D_{1/\|z\|}(b^*),\mathbb W_2\big)=\|z\|\text{dist}\big(D_{1/\|z\|}(b^*),\mathbb W_2\cap B(0,4)\big) \leq \\
&\leq \|z\|d_H(\mathbb W_1\cap B(0,4),\mathbb W_2\cap B(0,4)) \\ &=4\|z\|d_H(\mathbb W_1\cap B(0,1),\mathbb W_2\cap B(0,1)) < \varepsilon \|z\|, 
\end{split}
\end{equation}
where the first equality follows from the homogeneity of the distance, and the second is a consequence of the fact that $\|D_{1/\|z\|}(b^*)\|\leq 2$, and thus, from \eqref{eqn:Estimate2}, the point at minimum distance of $D_{1/\|z\|}(b^*)$ from $\mathbb W_2$ has norm bounded above by 4; the third inequality comes from the definition of Hausdorff distance, the fourth equality is true by homogeneity and the last inequality comes from the hypothesis $d_{\mathbb G}(\mathbb W_1,\mathbb W_2)<\varepsilon/4$. Joining \eqref{eqn:Estimate1}, and \eqref{eqn:Estimate3} we get $z\in C_{\mathbb W_2}(\alpha+\varepsilon)$, that was what we wanted. 
\end{proof}

\begin{lemma}\label{lemma:LCapCw=e}
Let $\mathbb V\in \G_c(h)$, and let $\mathbb L$  be a complementary subgroup of $\mathbb V$. There exists $\newep\label{ep:Cool}:=\oldep{ep:Cool}(\mathbb V,\mathbb L)>0$ such that 
$$
\mathbb L\cap C_{\mathbb V}(\oldep{ep:Cool})=\{0\}.
$$
Moreover we can, and will, choose $\oldep{ep:Cool}(\mathbb V,\mathbb L):=\oldC{C:split}(\mathbb V,\mathbb L)/2$.
\end{lemma}

\begin{proof}
We prove that it suffices to take $\oldep{ep:Cool}(\mathbb V,\mathbb L):=\oldC{C:split}(\mathbb V,\mathbb L)/2$. Let us suppose the statement is false. Thus there exists $0\neq v\in\mathbb L\cap C_{\mathbb V}(\oldep{ep:Cool})$. From \cref{prop:distvsproj} and from the very definition of the cone $C_{\mathbb V}(\oldep{ep:Cool})$ we have
$$
\oldC{C:split}(\mathbb V,\mathbb L)\|v\|\leq \dist(v,\mathbb V) \leq \oldep{ep:Cool}\lVert v\rVert= \oldC{C:split}(\mathbb V,\mathbb L)\|v\|/2,
$$
which is a contradiction with the fact that $v\neq 0$.
\end{proof}

We will not use the following proposition in the paper, but it is worth mentioning it.

\begin{proposizione}\label{prop:ComplGrassmannianOpen}
The family of the complemented subgroups $\G_c(h)$ is an open subset of $\G(h)$. 
\end{proposizione}

\begin{proof}
Fix a $\mathbb W\in \G_c(h)$ and let $\mathbb{L}$ be one complementary subgroup of $\mathbb W$ and set $\varepsilon<\min\{\oldep{ep:Cool}(\mathbb{V},\mathbb{L}),\hbar_{\mathbb G}\}$. Then, if $\mathbb{W}^\prime\in\G(h)$ is such that $d_\mathbb{G}(\mathbb{W},\mathbb{W}^\prime)<\varepsilon/4$, \cref{lem:DistanceOfCones} implies that $\mathbb{W}^\prime\subseteq C_{\mathbb{W}}(\varepsilon)$ and in particular $$\mathbb{L}\cap\mathbb{W}^\prime\subseteq \mathbb{L}\cap C_{\mathbb{W}}(\varepsilon)=\{0\}.$$
Moreover, since $\varepsilon<\hbar_\mathbb{G}$ from \cref{prop:htagliato}, we get that $\mathbb W'$ has the same stratification of $\mathbb W$ and thus the same topological dimension. This, jointly with the previous equality and the Grassmann formula, means that $(\mathbb{W}'\cap V_i)+(\mathbb{L}\cap V_i)=V_i$ for every $i=1,\dots,\kappa$. This, jointly with the fact that $\mathbb L\cap\mathbb W'=\{0\}$, implies that $\mathbb L$ and $\mathbb W'$ are complementary subgroups in $\mathbb G$ due to the triangular structure of the product $\cdot$ on $\mathbb G$, see \eqref{opgr}. For an alternative proof of the fact that $\mathbb L$ and $\mathbb W'$ are complementary subgroups,  see also \cite[Lemma 2.7]{JNGV20}.
\end{proof}

The following definition of intrinsically Lipschitz functions is equivalent to the classical one in \cite[Definition 11]{FranchiSerapioni16} because the cones in \cite[Definition 11]{FranchiSerapioni16} and the cones $C_{\mathbb V}(\alpha)$ are equivalent whenever $\mathbb V$ admits a complementary subgroup, see \cite[Proposition 3.1]{FranchiSerapioni16}.
\begin{definizione}[Intrinsically Lipschitz functions]\label{def:iLipfunctions}
Let $\mathbb{W}\in \G_c(h)$ and assume $\mathbb{L}$ is a complement of $\mathbb{W}$ and let $E\subseteq \mathbb{W}$ be a subset of $\mathbb{V}$. Let $\alpha>0$. A function $f:E\to \mathbb{L}$ is said to be an \emph{$\alpha$-intrinsically Lipschitz function} if $\text{graph}(f):=\{v\cdot f(v):v\in E\}$ is a $C_\mathbb{W}(\alpha)$-set. A function $f:E\to \mathbb{L}$ is said to be an \emph{intrinsically Lipschitz function} if there exists $\alpha>0$ such that $f$ is an $\alpha$-intrinsically Lipschitz function.
\end{definizione}

\begin{proposizione}\label{prop:ConeAndGraph}
Let us fix $\mathbb W\in\G_c(h)$ with complement $\mathbb L$. If $\Gamma\subset\mathbb G$ is a $C_{\mathbb W}(\alpha)$-set for some $\alpha\leq \oldep{ep:Cool}(\mathbb W,\mathbb L)$, then the map $P_{\mathbb W}:\Gamma\to\mathbb W$ is injective. As a consequence $\Gamma$ is the intrinsic graph of an intrinsically Lipschitz map defined on $P_{\mathbb W}(\Gamma)$.
\end{proposizione}
\begin{proof}
Suppose by contradiction that $P_{\mathbb W}:\Gamma\to\mathbb W$ is not injective. Then, there exist $p\neq q$ with $p,q\in \Gamma$ such that $P_{\mathbb W}(p)=P_{\mathbb W}(q)$. Thus $p^{-1}\cdot q\in\mathbb L$. Moreover, since $\Gamma$ is a $C_{\mathbb W}(\alpha)$-set, we have that $p^{-1}\cdot q\in C_{\mathbb W}(\alpha)$. Eventually we get
$$
p^{-1}\cdot q\in \mathbb L\cap C_{\mathbb W}(\alpha) \subseteq \mathbb L\cap C_{\mathbb W}(\oldep{ep:Cool}(\mathbb W,\mathbb L)),
$$
where the last inclusion follows since $\alpha\leq\oldep{ep:Cool}(\mathbb W,\mathbb L)$. The above inclusion, jointly with \cref{lemma:LCapCw=e}, gives that $p^{-1}\cdot q=0$ and this is a contradiction.  Concerning the last part of the statement, let us notice that the map $P_{\mathbb L}\circ\big((P_{\mathbb W})_{|_{\Gamma}}\big)^{-1}$ is well-defined from $P_{\mathbb W}(\Gamma)$ to $\mathbb L$ and its intrinsic graph is $\Gamma$ by definition. Moreover, since $\Gamma$ is a $C_{\mathbb W}(\alpha)$-set, the latter map is intrinsically Lipschitz by \cref{def:iLipfunctions}.
\end{proof}

The following two lemmata will play a fundamental role in the proof that $\mathscr{P}^{c}_h$-rectifiable measures have $h$-density.

\begin{lemma}\label{lemma:projections}
Let $\mathbb V\in \G_c(h)$ and $\mathbb L$ be one of its complementary subgroups. For any $0<\alpha<\oldC{C:split}(\mathbb{V},\mathbb{L})/2$, let 
\begin{equation}\label{eqn:calpha}
\mathfrak{c}(\alpha):=\alpha/(\oldC{C:split}(\mathbb{V},\mathbb{L})-\alpha).
\end{equation}
Then we have
\begin{equation}
    B(0,1)\cap \mathbb{V}\subseteq P_\mathbb{V}(B(0,1)\cap C_\mathbb{V}(\alpha))\subseteq B(0,1/(1-\mathfrak{c}(\alpha)))\cap \mathbb{V}.\label{eq:incl}
\end{equation}
\end{lemma}

\begin{proof}
The first inclusion comes directly from the definition of projections and cones. Concerning the second, if $v\in B(0,1)\cap C_\mathbb{V}(\alpha)$, thanks to \cref{prop:distvsproj} we have
\begin{equation}\label{eqn:Estimatecalpha}
\oldC{C:split}(\mathbb{V},\mathbb{L})\lVert P_\mathbb{L}(v)\rVert\leq \dist(v,\mathbb{V})\leq \alpha\lVert v\rVert\leq \alpha(\lVert P_\mathbb{L}(v)\rVert+\lVert P_\mathbb{V}(v)\rVert).
\end{equation}
This implies in particular that
$\lVert P_\mathbb{L}(v)\rVert\leq \mathfrak{c}(\alpha)\lVert P_\mathbb{V}(v)\rVert$ and thus
$$1\geq \lVert P_{\mathbb{V}}(v)P_\mathbb{L}(v)\rVert\geq\lVert P_{\mathbb{V}}(v)\rVert-\lVert P_\mathbb{L}(v)\rVert\geq (1-\mathfrak{c}(\alpha))\lVert P_\mathbb{V}(v)\rVert.$$
This concludes the proof of the lemma.
\end{proof}

\begin{lemma}\label{lemma:lowerbd}
Let $\mathbb V\in \G_c(h)$ and $\mathbb L$ be one of its complementary subgroups.
Suppose $\Gamma$ is a $C_{\mathbb V}(\alpha)$-set with $\alpha<\oldC{C:split}(\mathbb{V},\mathbb{L})/2$, and let
\begin{equation}\label{eqn:Calpha}
\mathfrak{C}(\alpha):=\frac{1-\mathfrak{c}(\alpha)}{1+\mathfrak{c}(\alpha)},
\end{equation} 
where $\mathfrak c(\alpha)$ is defined in \eqref{eqn:calpha}. Then
$$
\mathcal{S}^h(P_\mathbb{V}(B(x,r)\cap \Gamma))\geq \mathcal{S}^h\Big(P_\mathbb{V}\big(B(x,\mathfrak{C}(\alpha)r)\cap xC_\mathbb{V}(\alpha)\big)\cap P_\mathbb{V}(\Gamma)\Big), \qquad \text{for any $x\in\Gamma$}.
$$
The same inequality above holds if we substitute $\mathcal{S}^h$ with any other Haar measure on $\mathbb V$, see \cref{prop:haar}, because all of them are equal up to a constant.
\end{lemma}

\begin{proof}
First of all, let us note that we have
\begin{equation}
    \mathcal{S}^h\big(P_\mathbb{V}(B(x,r)\cap \Gamma)\big)=\mathcal{S}^h\Big(P_\mathbb{V}\big(B(0,r)\cap x^{-1}\Gamma\big)\Big),
    \label{eq:bd3}
\end{equation}
where the last equality is true since $\mathcal{S}^h(P_{\mathbb V}(\mathcal{E}))=\mathcal{S}^h(P_{\mathbb V}(x^{-1}\mathcal{E}))$ for any Borel $\mathcal{E}\subseteq\mathbb G$, see \cref{prop:InvarianceOfProj}.
We wish now to prove the following inclusion
\begin{equation}
     P_\mathbb{V}\big(B(0,\mathfrak{C}(\alpha)r)\cap C_\mathbb{V}(\alpha)\big)\cap P_\mathbb{V}(x^{-1}\Gamma)\subseteq  P_\mathbb{V}(B(0,r)\cap x^{-1}\Gamma).
     \label{eq:bd4}
\end{equation}
Indeed, fix an element $y$ of $P_\mathbb{V}(B(0,
\mathfrak{C}(\alpha)r)\cap C_{\mathbb V}(\alpha))\cap P_\mathbb{V}(x^{-1}\Gamma)$. Thanks to our choice of $y$ there are a $w_1\in x^{-1}\Gamma$ and a $w_2\in B(0,\mathfrak{C}(\alpha)r)\cap C_\mathbb{V}(\alpha)$ such that $$P_\mathbb{V}(w_1)=y=P_\mathbb{V}(w_2).$$
Furthermore, since $\Gamma$ is a $C_\mathbb{V}(\alpha)$-set, we infer that $w_1\in C_\mathbb{V}(\alpha)$ and thus with the same computations as in \eqref{eqn:Estimatecalpha}
we obtain that $\lVert P_\mathbb{L}(w_1)\rVert\leq \mathfrak{c}(\alpha)\lVert P_\mathbb{V}(w_1)\rVert$ and thus
\begin{equation}
  \lVert w_1\rVert\leq (1+\mathfrak{c}(\alpha))\lVert P_{\mathbb V} w_1\rVert\leq (1+\mathfrak{c}(\alpha))\lVert y\rVert.
   \label{eq:bd1}
\end{equation}
 Furthermore, since by assumption $w_2\in B(0,\mathfrak{C}(\alpha)r)\cap C_\mathbb{V}(\alpha)$, \cref{lemma:projections} yields
 \begin{equation}
     \lVert y\rVert=\lVert P_\mathbb{V}(w_2)\rVert \leq \mathfrak{C}(\alpha)r/(1-\mathfrak{c}(\alpha))=r/(1+\mathfrak{c}(\alpha)).
     \label{eq:bd2}
 \end{equation}
The bounds \eqref{eq:bd1} and \eqref{eq:bd2} together imply that
$\lVert w_1\rVert\leq r$, and thus $w_1\in B(0,r)\cap x^{-1}\Gamma$ and this concludes the proof of the inclusion \eqref{eq:bd4}.
Finally \eqref{eq:bd3}, \eqref{eq:bd4} imply
\begin{equation}\label{eqn:IlLemmaèfinito}
\mathcal{S}^h(P_\mathbb{V}(B(x,r)\cap \Gamma))\geq \mathcal{S}^h\Big(P_\mathbb{V}(B(0,\mathfrak{C}(\alpha)r)\cap C_\mathbb{V}(\alpha))\cap P_\mathbb{V}(x^{-1}\Gamma)\Big).
\end{equation}
Furthermore, for any Borel subset $\mathcal{E}$ of $\mathbb{G}$ we have $P_\mathbb{V}(x\mathcal{E})=P_\mathbb{V}(xP_\mathbb{V}(\mathcal{E}))$, since for every $g\in\mathcal{E}$ we have the following simple equality $P_{\mathbb V}(xg)=P_{\mathbb V}(xP_{\mathbb V}g)$. Therefore, by using the latter observation and \cref{prop:InvarianceOfProj}, we get, denoting with $\Psi$ the map $\Psi(v)=P_{\mathbb V}(x^{-1}v)$ for every $v\in\mathbb V$, that
\begin{equation}\label{eqn:Lemmafinito2}
    \begin{split}
        \mathcal{S}^h\Big(P_\mathbb{V}\big(B(0,&\mathfrak{C}(\alpha)r)\cap C_\mathbb{V}(\alpha)\big)\cap P_\mathbb{V}\big(x^{-1}\Gamma\big)\Big)\\
        &=\mathcal{S}^h\Big(P_\mathbb{V}\big(x^{-1}P_\mathbb{V}(B(x,\mathfrak{C}(\alpha) r)\cap xC_\mathbb{V}(\alpha))\big)\cap P_\mathbb{V}\big(x^{-1}P_\mathbb{V}(\Gamma)\big)\Big)\\
        &=\mathcal{S}^h\Big(\Psi\big(P_\mathbb{V}(B(x,\mathfrak{C}(\alpha) r)\cap xC_\mathbb{V}(\alpha))\big)\cap \Psi\big(P_\mathbb{V}(\Gamma)\big)\Big)\\
        &=\mathcal{S}^h\Big(P_\mathbb{V}(B(x,\mathfrak{C}(\alpha) r)\cap xC_\mathbb{V}(\alpha))\cap P_\mathbb{V}(\Gamma)\Big).
    \end{split}
\end{equation}
 Joining together \eqref{eqn:IlLemmaèfinito} and \eqref{eqn:Lemmafinito2} gives the sought conclusion.
\end{proof}

\subsection{Rectifiable measures in Carnot groups}
In what follows we are going to define the class of $h$-flat measures on a Carnot group and then we will give proper definitions of rectifiable measures on Carnot groups. Again we recall that throughout this subsection $\mathbb G$ will be a fixed Carnot group endowed with an arbitrary left-invariant homogeneous distance.

\begin{definizione}[Flat measures]\label{flatmeasures}
For any $h\in\{1,\ldots,Q\}$ we let $\mathfrak{M}(h)$ to be the {\em family of flat $h$-dimensional measures} in $\mathbb{G}$, i.e.
$$\mathfrak{M}(h):=\{\lambda\mathcal{S}^h\llcorner \mathbb W:\text{ for some }\lambda> 0 \text{ and }\mathbb W\in\G(h)\}.$$ 
Furthermore, if $G$ is a subset of the $h$-dimensional Grassmanian $\G(h)$, we let $\mathfrak{M}(h,G)$ to be the set \begin{equation}
    \mathfrak{M}(h,G):=\{\lambda\mathcal{S}^h\llcorner \mathbb{W}:\text{ for some }\lambda> 0\text{ and }\mathbb{W}\in G\}.
    \label{Gotico(h)}
\end{equation}
We stress that in the previous definitions we can use any of the Haar measures on $\mathbb W$, see \cref{prop:haar}, since they are the same up to a constant.
\end{definizione}

\begin{definizione}[$\mathscr{P}_h$ and $\mathscr{P}_h^*$-rectifiable measures]\label{def:PhRectifiableMeasure}
Let $h\in\{1,\ldots,Q\}$. A Radon measure $\phi$ on $\mathbb G$ is said to be a $\mathscr{P}_h$-rectifiable measure if for $\phi$-almost every $x\in \mathbb{G}$ we have
\begin{itemize}
    \item[(i)]$0<\Theta^h_*(\phi,x)\leq\Theta^{h,*}(\phi,x)<+\infty$,
    \item[(\hypertarget{due}{ii})]there exists a $\mathbb{V}(x)\in\G(h)$ such that $\mathrm{Tan}_h(\phi,x) \subseteq \{\lambda\mathcal{S}^h\llcorner \mathbb V(x):\lambda\geq 0\}$.
\end{itemize}
Furthermore, we say that $\phi$ is $\mathscr{P}_h^*$-rectifiable if (\hyperlink{due}{ii})  is replaced with the weaker
\begin{itemize}
    \item[(ii)*] $\mathrm{Tan}_h(\phi,x) \subseteq \{\lambda\mathcal{S}^h\llcorner \mathbb V:\lambda\geq 0\,\,\text{and}\,\,\mathbb V\in \G(h)\}$.
\end{itemize}
\end{definizione}
\begin{osservazione}(About $\lambda=0$ in \cref{def:PhRectifiableMeasure})\label{rem:AboutLambda=0}
It is readily noticed that, since in \cref{def:PhRectifiableMeasure} we are asking $\Theta^h_*(\phi,x)>0$ for $\phi$-almost every $x$, we can not have the zero measure as a tangent measure. As a consequence, a posteriori, we have that in item (ii) and item (ii)* above we can restrict to $\lambda>0$. We will tacitly work in this restriction from now on.

On the contrary, if we only know that for $\phi$-almost every $x\in\mathbb G$ we have 
\begin{equation}\label{eqn:YAYA}
\Theta^{h,*}(\phi,x)<+\infty, \qquad \text{and} \qquad  \mathrm{Tan}_h(\phi,x)\subseteq \{\lambda\mathcal{S}^h\llcorner\mathbb V(x):\lambda>0\},
\end{equation}
for some $\mathbb V(x)\in \G(h)$, hence $\Theta^h_*(\phi,x)>0$ for $\phi$-almost every $x\in\mathbb G$, and the same property holds with the item (ii)* above. Indeed, if at some $x$ for which \eqref{eqn:YAYA} holds we have $\Theta^h_*(\phi,x)=0$, then there exists $r_i\to 0$ such that $r_i^{-h}\phi(B(x,r_i))=0$. Since $\Theta^{h,*}(\phi,x)<+\infty$, up to subsequences (see \cite[Theorem 1.60]{AFP00}), we have $r_i^{-h}T_{x,r_i}\phi\to \lambda\mathcal{S}^h\llcorner\mathbb V(x)$, for some $\lambda>0$. Hence, by applying \cite[Proposition 1.62(b)]{AFP00} we conclude that $r_i^{-h}T_{x,r_i}\phi(B(0,1))\to \lambda\mathcal{S}^h\llcorner\mathbb V(x)(B(0,1))>0$, that is a contradiction.
\end{osservazione}

Throughout the paper it will be often convenient to restrict our attention to the subclass of $\mathscr{P}_h$- and $\mathscr{P}_h^*$-rectifiable measures, given by the measures that have complemented tangents. More precisely we give the following definition.

\begin{definizione}[$\mathscr{P}_h^c$-rectifiable measures]\label{def:SubclassesPh}
Let $h\in\{1,\ldots,Q\}$. In the following we denote by $\mathscr{P}^{c}_h$ the family of those $\mathscr{P}_h$-rectifiable measures such that for $\phi$-almost every $x\in \mathbb{G}$ we have
$$\Tan_h(\phi,x)\subseteq \mathfrak{M}(h,\G_c(h)).$$
\end{definizione}

\begin{osservazione}
As explained in the introduction, the bridge between the notion of $\mathscr{P}_h^c$-rectifiable measures and the other notions of rectifiability in the Heisenberg groups $\mathbb H^n$ is nowadays very well understood after the results in \cite[Theorem 3.14 and Theorem 3.15]{MatSerSC} and \cite{Vittone20}. Let us now discuss some examples of flat rectifiable measures in a different setting, i.e., in the Engel group, which we denote by $\mathbb E$.

The Engel group $\mathbb E$ is the Carnot group whose Lie algebra $\mathfrak e$ admits a basis $\{X_1,X_2,X_3,X_4\}$ such that $[X_1,X_2]=X_3$, and $[X_1,X,3]=X_4$. Hence it is a step-3 Carnot group of topological dimension 4 and homogeneous dimension 7 where $V_1=\mathrm{span}\{X_1,X_2\}$, $V_2=\mathrm{span}\{X_3\}$, and $V_3=\mathrm{span}\{X_4)$. In exponential coordinates associated to the basis $(X_1,X_2,X_3,X_4)$ the law product can be explicitly written as in \cite[Example 2.6]{FranchiSerapioni16}. As explicitly computed in \cite[Example 2.6]{FranchiSerapioni16}, in $\mathbb E$ we have two families of homogeneous complementary subgroups. The first family is given by $\mathbb M_{\alpha,\beta}\cdot\mathbb N_{\gamma,\delta}$ where $\alpha,\beta,\gamma,\delta\in\mathbb R$ satisfy $\alpha\delta-\beta\gamma\neq 0$, and
\[
\mathbb M_{\alpha,\beta}:=\{(\alpha t,\beta t,0,0):t\in\mathbb R\}, \qquad \mathbb N_{\gamma,\delta}:=\{(\gamma t,\delta t,x_3,x_4): t,x_3,x_4\in\mathbb R\}.
\]
In this case notice that, by homogeneity, $\mathcal{S}^1\llcorner \mathbb M_{\alpha,\beta}$ and $\mathcal{S}^6\llcorner \mathbb N_{\gamma,\delta}$ are both $\mathscr{P}_h^c$-rectifiable measures. Notice, moreover, that $\mathbb N_{\gamma,\delta}$ is normal and it is also a $C^1_{\mathrm{H}}$-hypersurface, thus being rectifiable in the sense of \cref{def:GG'Rect}. The same does not hold for $\mathbb M_{\alpha,\beta}$ since it is not normal, compare with \cref{oss:LackOfGenerality}.

The second family of homogeneous subgroups is given by $\mathbb K\cdot \mathbb H_{\alpha,\beta}$, where $\alpha,\beta\in\mathbb R$ satisfy $\alpha+\beta\neq 0$, and 
\[
\mathbb K:=\{(x_1,-x_1,x_3,0):x_1,x_3\in\mathbb R\}, \qquad \mathbb H_{\alpha,\beta}:=\{(\alpha t,\beta t, 0,x_4):t,x_4\in\mathbb R\}.
\]
Notice that $\mathcal{S}^3\llcorner \mathbb K$ and $\mathcal{S}^4\llcorner \mathbb H_{\alpha,\beta}$ are both $\mathscr{P}_h^c$-rectifiable measures. Nevertheless, since both $\mathbb K$ and $\mathbb H_{\alpha,\beta}$ are not normal, none of them can be a $C^1_{\mathrm H}(\mathbb E;\mathbb G')$ for any Carnot group $\mathbb G'$ because otherwise the tangent, which coincides everywhere with the same subgroup by homogeneity, would be normal, compare with \cref{oss:LackOfGenerality}.

\end{osservazione}

\begin{proposizione}\label{prop:density}
Let $h\in\{1,\ldots,Q\}$ and assume $\phi$ is a Radon measure on $\mathbb G$.
If $\{r_i\}_{i\in\N}$ is an infinitesimal sequence such that $r_i^{-h}T_{x,r_i}\phi\rightharpoonup \lambda \mathcal{C}^h\llcorner \mathbb{V}$ for some $\lambda>0$ and $\mathbb{V}\in\G(h)$ then
$$\lim_{i\to\infty} \phi(B(x,r_i))/r_i^h=\lambda.$$
\end{proposizione}

\begin{proof}
Since $\mathcal{C}^h\llcorner \mathbb{V}(x)(\partial B(0,1))=0$, see e.g., \cite[Lemma 3.5]{JNGV20}, thanks to \cref{rem:Ch1} and to \cite[Proposition 1.62(b)]{AFP00} we have
$$\lambda=\lambda \mathcal{C}^h\llcorner\mathbb{V}(x)(B(0,1))=\lim_{i\to\infty}\frac{T_{x,r_i}\phi(B(0,1))}{r_i^h}=\lim_{i\to\infty}\frac{\phi(B(x,r_i))}{r_i^h},$$
and this concludes the proof.
\end{proof}

The above proposition has the following immediate consequence.

\begin{corollario}\label{cordenstang}
Let $h\in\{1,\ldots,Q\}$ and assume $\phi$ is a $\mathscr{P}^*_h$-rectifiable. Then for $\phi$-almost every $x\in\mathbb{G}$ we have
$$\Tan_h(\phi,x)\subseteq \{\lambda\mathcal{C}^h\llcorner\mathbb{W}:\lambda\in[\Theta^h_*(\phi,x),\Theta^{h,*}(\phi,x)]\text{ and }\mathbb{W}\in\G(h)\}.$$
\end{corollario}

We introduce now a way to estimate how far two measures are.

\begin{definizione}\label{def:Fk}
Given $\phi$ and $\psi$ two Radon measures on $\mathbb G$, and given $K\subseteq \mathbb G$ a compact set, we define 
\begin{equation}
    F_K(\phi,\psi):= \sup\left\{\left|\int fd\phi - \int fd\psi\right|:f\in \mathrm{Lip}_1^+(K)\right\}.
    \label{eq:F}
\end{equation}
We also write $F_{x,r}$ for $F_{B(x,r)}$.
\end{definizione}

\begin{osservazione}\label{rem:ScalinfFxr}
With few computations that we omit, it is easy to see that $F_{x,r}(\phi,\psi)=rF_{0,1}(T_{x,r}\phi,T_{x,r}\psi)$. Furthermore, $F_K$ enjoys the triangle inequality, indeed if $\phi_1,\phi_2,\phi_3$ are Radon measures and $f\in\lip(K)$, then
\begin{equation*}
\begin{split}
\Big\lvert\int fd\phi_1-\int fd\phi_2\Big\rvert&\leq \Big\lvert\int fd\phi_1-\int fd\phi_3\Big\rvert+\Big\lvert\int fd\phi_3-\int fd\phi_2\Big\rvert\\ &\leq F_K(\phi_1,\phi_2)+F_K(\phi_2,\phi_3).
\end{split}
\end{equation*}
The arbitrariness of $f$ concludes that $F_K(\phi_1,\phi_2)\leq F_K(\phi_1,\phi_3)+F_K(\phi_3,\phi_2)$.
\end{osservazione}

The proof of the following criterion is contained in  \cite[Proposition 1.10]{MarstrandMattila20} and we omit the proof.
\begin{proposizione}\label{prop:WeakConvergenceAndFk}
Let $\{\mu_i\}$ be a sequence of Radon measures on $\mathbb G$. Let $\mu$ be a Radon measure on $\mathbb G$. The following are equivalent
\begin{enumerate}
    \item $\mu_i\rightharpoonup \mu$;
    \item $F_K(\mu_i,\mu)\to 0$, for every $K\subseteq \mathbb G$ compact. 
\end{enumerate}
\end{proposizione}

Now we are going to define a functional that in some sense tells how far is a measure from being flat around a point $x\in\mathbb G$ and at a certain scale $r>0$.

\begin{definizione}\label{def:metr}
For any $x\in\mathbb{G}$, any $h\in\{1,\ldots,Q\}$ and any $r>0$ we define the functional:
\begin{equation}\label{eqn:dxr}
    d_{x,r}(\phi,\mathfrak{M}(h)):=\inf_{\substack{\Theta>0,\\ \mathbb V\in \G(h)}} \frac{F_{x,r}(\phi,\Theta \mathcal{S}^{h}\llcorner x\mathbb V)}{r^{h+1}}. 
\end{equation}
Furthermore, if $G$ is a subset of the $h$-dimensional Grassmanian $\G(h)$, we also define
$$ d_{x,r}(\phi,\mathfrak{M}(h,G)):=\inf_{\substack{\Theta>0,\\ \mathbb V\in G }} \frac{F_{x,r}(\phi,\Theta \mathcal{S}^{h}\llcorner x\mathbb V)}{r^{h+1}}.$$

\end{definizione}
\begin{osservazione}\label{rem:dxrContinuous}
It is a routine computation to prove that, whenever $h\in \mathbb N$ and $r>0$ are fixed, the function $x\mapsto d_{x,r}(\phi,\mathfrak M(h,G))$ is a continuous function. The proof can be reached as in \cite[Item (ii) of Proposition 2.2]{MarstrandMattila20}. Moreover, from the invariance property in \cref{rem:ScalinfFxr} and \cref{prop:haar}, if in \eqref{eqn:dxr} we use the measure $\mathcal{C}^h\llcorner x\mathbb V$ we obtain the same functional.
\end{osservazione}

\begin{proposizione}\label{prop:TanAndDxr}
Let $\phi$ be a Radon measure on $\mathbb G$ satisfying item (i) in \cref{def:PhRectifiableMeasure}. Further, let $G$ be a subfamily of $\G(h)$ and let $\mathfrak{M}(h,G)$ be the set defined in \eqref{Gotico(h)}. If for $\phi$-almost every $x\in \mathbb G$ we have $\mathrm{Tan}_h(\phi,x)\subseteq \mathfrak{M}(h,G)$, then for $\phi$-almost every $x\in\mathbb G$ we have
$$
\lim_{r\to 0} d_{x,r}(\phi,\mathfrak{M}(h,G))=0.$$
\end{proposizione}

\begin{proof}
Let us fix $x\in\mathbb G$ a point for which $\mathrm{Tan}_h(\phi,x)\subseteq \mathfrak{M}(h,G)$ and let us assume by contradiction that there exist $r_i\to 0$ such that, for some $\varepsilon>0$ we have
\begin{equation}\label{eqn:CONT}
d_{x,r_i}(\phi,\mathfrak{M}(h,G)) > \varepsilon.
\end{equation}
Since $\phi$ satisfies item (i) in \cref{def:PhRectifiableMeasure}, we can use \cite[Proposition 1.62(b)]{AFP00} and then, up to subsequences, there are $\Theta^*>0$ and $\mathbb V^*\in G$ such that 
\begin{equation}\label{eqn:CONV}
r_i^{-h}T_{x,r_i}\phi \rightharpoonup \Theta^*\mathcal{S}^h\llcorner \mathbb V^*.
\end{equation}
Thus,
$$
d_{x,r_i}(\phi,\mathfrak{M}(h,G)) =d_{0,1}(r_i^{-h}T_{x,r_i}\phi,\mathfrak{M}(h,G)) \leq F_{0,1}(r_i^{-h}T_{x,r_i}\phi,\Theta^*\mathcal{S}^h\llcorner \mathbb V^*)\to 0,
$$
where the first equality follows from the first part of \cref{rem:ScalinfFxr}, and the last convergence follows from \eqref{eqn:CONV}, and \cref{prop:WeakConvergenceAndFk}. This is in contradiction with \eqref{eqn:CONT}.
\end{proof}

\section{Structure of $\mathscr{P}_h$-rectifiable measures}\label{sec:Structure}

\textbf{In what follows we let $\mathbb G$ be a Carnot group of homogeneous dimension $Q$ and we fix $1\leq h\leq Q$. We endow $\mathbb G$ with a fixed homogeneous left-invariant distance. We also assume that $\phi$ is a fixed Radon measure on $\mathbb G$ and we suppose that it is supported on a compact set $K$. Moreover we fix $\vartheta,\gamma\in\mathbb N$ and we freely use the notation $E(\vartheta,\gamma)$ introduced in \cref{def:EThetaGamma}.}
\medskip

In this section we prove \cref{thm:MainTheorem1Intro} and an important step toward the proof of \cref{coroll:IdiffMeasureIntroNEW}, see the statements in \cref{thm:MainTheorem1} and \cref{thm:MainTheorem2}, respectively. 

The first step in order to prove \cref{thm:MainTheorem1Intro} is to observe the following general property, that can be made quantitative at arbitrary points $x\in E(\vartheta,\gamma)$: if the measure $\mathcal{S}^h\llcorner x\mathbb V$, with $\mathbb V\in \G(h)$, is sufficiently near to $\phi$ in a precise Measure Theoretic sense at the scale $r$ around $x$, then in some ball of center $x$ and with radius comparable with $r$, the points in the set $E(\vartheta,\gamma)$ are not too distant from $x\mathbb V$. Roughly speaking, if we denote with $F_{x,r}$ the functional that measures the distance between measures on the ball $B(x,r)$, see \cref{def:Fk}, we prove that the following implication holds 
\begin{equation}\label{eqn:Implication1Intro}
\begin{split}
     &\text{if there exist a $\Theta,\delta>0$ such that $F_{x,r}(\phi,\Theta\mathcal{S}^h\llcorner x\mathbb V)\leq \delta r^{h+1}$,}\\
 &\text{then $E(\vartheta,\gamma)\cap B(x,r)\subseteq B(x\mathbb V, \omega(\delta)r)$ where $\omega$ is continuous and $\omega(0)=0$.}
\end{split}
\end{equation}
For the precise statement of \eqref{eqn:Implication1Intro}, see \cref{prop1vsinfty}. Let us remark that when $\phi$ is a $\mathscr{P}_h$-rectifiable measure, then for $\phi$-almost every $x\in\mathbb G$ the bound on $F_{x,r}$ in the premise of \eqref{eqn:Implication1Intro} is satisfied with $\mathbb V(x)\in\G(h)$, for arbitrarily small $\delta>0$ whenever $r<r_0(x,\delta)$. Thus for $\mathscr{P}_h$-rectifiable measures we deduce that the estimate in the conclusion of \eqref{eqn:Implication1Intro} holds for arbitrarily small $\delta$, and with $r<r_0(x,\delta)$. This latter estimate easily implies, by a very general geometric argument, that $E(\vartheta,\gamma)\cap B(x,r) \subseteq xC_{\mathbb V(x)}(\alpha)$ for arbitrarily small $\alpha$ and for all $r<r_0(x,\alpha)$. For the latter assertion we refer the reader to \cref{prop:cono}. The proof of \cref{thm:MainTheorem1Intro} is thus concluded by joining together the previous observations and by the general cone-rectifiability criterion in \cref{prop:RectifiabilityCriterion}.

\vspace{0.3cm}
There is a difference between the Euclidean case and the Carnot case that we discuss now. In the Euclidean case it is easy to see that whenever we are given a vector subspace $V$, an arbitrary $C_{V}(\alpha)$-set, with $\alpha$ sufficiently small, is actually the graph of a (Lipschitz) map $f:A\subseteq V\to V^\perp$. The main reason behind this latter statement is the following: we have a canonical choice of a complementary subgroup $V^\perp$ of $V$, and moreover $V^\perp\cap C_{V}(\alpha)=\{0\}$ for $\alpha$ small enough. Already in the first Heisenberg group $\mathbb H^1$ if we take the vertical line $\mathbb V_{\mathbb H^1}$, we notice that there is no choice of a complementary subgroup of $\mathbb V_{\mathbb H^1}$ in $\mathbb H^1$. One could try to bypass this problem by defining properly some {\em coset projections} that would play the role of the projection over a splitting, see \cref{def:Projections}. This will be the topic of further investigations.

Nevertheless, if we work in an arbitrary Carnot group $\mathbb G$ and one of its homogeneous subgroups $\mathbb V$ admits a complementary subgroup $\mathbb L$ we already proved that there exists a constant $\oldep{ep:Cool}:=\oldep{ep:Cool}(\mathbb V,\mathbb L)$ such that every $C_{\mathbb V}(\oldep{ep:Cool})$-set is the intrinsic graph of a function $f:A\subseteq\mathbb V\to\mathbb L$. This last statement is precisely the analogous of the Euclidean property that we discussed above, see \cref{prop:ConeAndGraph}. As a consequence, in order to start to prove \cref{coroll:IdiffMeasureIntroNEW} we follow the path of the proof of \cref{thm:MainTheorem1Intro}, which we discussed above, but we have to pay attention to one technical detail. We have to split the subset of the Grassmannian $\G(h)$ made by the homogeneous subgroups $\mathbb V$ that admit at least one complementary subgroup $\mathbb L$ into countable subsets according to the value of $\oldep{ep:Cool}(\mathbb V,\mathbb L)$. Then we have to write the proof of \cref{thm:MainTheorem1Intro} by paying attention to the fact that we want to control the opening of the final $C_{\mathbb V_i}(\alpha_i)$-sets with $\alpha_i<\oldep{ep:Cool}(\mathbb V_i,\mathbb L_i)$. This is what we do in \cref{thm:MainTheorem1}: we prove a refinement of \cref{thm:MainTheorem2} in which we further ask that the opening of the cones is controlled above also by some a priori defined function $\bold F(\mathbb V,\mathbb L)$.

\begin{definizione}\label{def:Pixr}
Let us fix $x\in \mathbb G$, $r>0$ and $\phi$ a Radon measure on $\mathbb G$.
We define $\Pi_{\delta}(x,r)$ to be the subset of planes $\mathbb V\in\G(h)$ for which there exists $\Theta>0$ such that
\begin{equation}\label{eqn:Previous}
F_{x,r}(\phi, \Theta \mathcal{S}^{h}\llcorner x\mathbb V)\leq 2\delta r^{h+1}.
\end{equation}
\end{definizione}

\begin{definizione}\label{eqn:DefinitionOfDeltaG}
For any $\vartheta\in\mathbb N$ we define $\delta_\mathbb{G}=\delta_{\mathbb G}(h,\vartheta):=\vartheta^{-1}2^{-(4h+5)}$.
\end{definizione}

 In the following proposition we prove that if $\phi$ is sufficiently $d_{x,r}$-near to $\mathfrak{M}(h)$, see \cref{def:metr} for the definition of $d_{x,r}$, then $E(\vartheta,\gamma)$ is at a controlled distance from a plane $\mathbb V$.

\begin{proposizione}\label{prop1vsinfty}
Let $x\in E(\vartheta,\gamma)$,
fix $\delta<\delta_\mathbb{G}$, where $\delta_{\mathbb G}$ is defined in \cref{eqn:DefinitionOfDeltaG}, and set $0<r<1/\gamma$. Then for every $\mathbb V\in \Pi_{\delta}(x,r)$, see \cref{def:Pixr}, we have
\begin{equation}
\sup_{w\in E(\vartheta,\gamma)\cap B(x,r/4)} \frac{\dist\big(w,x\mathbb V\big)}{r}\leq2^{1+1/(h+1)}\vartheta^{1/(h+1)}\delta^{1/(h+1)}=: \newC\label{C:b1}(\vartheta,h)\delta^{1/(h+1)}.
\end{equation}
\end{proposizione}

\begin{proof}
Let $\mathbb V$ be any element of $\Pi_{\delta}(x,r)$ and suppose $\Theta>0$ is such that
\begin{equation*}
\bigg\lvert \int f d\phi-\Theta\int fd\mathcal{S}^{h}\llcorner x\mathbb V\bigg\rvert\leq 2\delta r^{h+1},\text{ for any }f\in\lip(B(x,r)).
\end{equation*}
Since the function $g(w):=\min\{\dist(w,U(x,r)^c),\dist(w,x\mathbb V)\}$ belongs to \\ $\lip(B(x,r))$, we deduce that
\begin{equation*}
\begin{split}
2\delta r^{h+1}&\geq\int g(w)d\phi(w)-\Theta\int g(w)d\mathcal{S}^{h}\llcorner x\mathbb V\\&=\int g(w)d\phi(w)\geq \int_{B(x,r/2)}\min\{r/2,\dist(w,x\mathbb V)\}d\phi(w).
\end{split}
\end{equation*}
Suppose that $y$ is a point in $B(x,r/4)\cap E(\vartheta,\gamma)$ furthest from $x\mathbb V$ and let $D:=\dist(y,x\mathbb V)$.
If $D\geq r/8$, this would imply that
\begin{equation}
\begin{split}
 2\delta r^{h+1}&\geq\int_{B(x,r/2)}\min\{r/2,\dist(w,x\mathbb V)\}d\phi(w)\\&\geq \int_{B(y,r/16)}\min\{r/2,\dist(w,x\mathbb V)\}d\phi(w)
 \geq\frac{r}{16}\phi(B(y,r/16))\geq \frac{r^{h+1}}{\vartheta 16^{h+1}},
    \nonumber
\end{split}
\end{equation}
where the last inequality follows from the definition of $E(\vartheta,\gamma)$ and the fact that $0<r<1/\gamma$. The previous inequality would imply $\delta\geq \vartheta^{-1}2^{-(4h+5)}$, which is not possible since $\delta<\delta_{\mathbb G}=\vartheta^{-1}2^{-(4h+5)}$, see \cref{eqn:DefinitionOfDeltaG}.
This implies that $D\leq r/8$ and as a consequence, we have
\begin{equation}
\begin{split}
     2\delta r^{h+1}&\geq \int_{B(x,r/2)}\min\{r/2,\dist(w,x\mathbb V)\}d\phi(w) \\&\geq \int_{B(y,D/2)}\min\{r/2,\dist(w,x\mathbb V)\}d\phi(w)
     \geq \frac{D\phi(B(y,D/2))}{2}\geq \vartheta^{-1}\left(\frac{D}{2}\right)^{h+1},
\end{split}
    \label{eq:2023}
\end{equation}
where the second inequality comes from the fact that $B(y,D/2)\subseteq B(x,r/2)$. This implies thanks to \eqref{eq:2023}, that
$$\sup_{w\in E(\vartheta,\gamma)\cap B(x,r/4)}\frac{\dist(w,x\mathbb V)}{r}\leq\frac{D}{r}\leq 2^{1+1/(h+1)}\vartheta^{1/(h+1)}\delta^{1/(h+1)}=\oldC{C:b1}(\vartheta,h)\delta^{1/(h+1)}.$$
\end{proof}
\begin{osservazione}\label{rem:prop1vsinftyANCHEC}
Notice that a priori $\Pi_{\delta}(x,r)$ in the statement of \cref{prop1vsinfty} may be empty. Nevertheless it is easy to notice, by using the definitions, that if $d_{x,r}(\phi,\mathfrak{M})\leq \delta$ then $\Pi_{\delta}(x,r)$ is nonempty. 
\end{osservazione}

In the following proposition we show that if we are at a point $x\in E(\vartheta,\gamma)$ for which the $h$-tangents are flat, then locally around $x$ the set $E(\vartheta,\gamma)$ enjoys an appropriate cone property with arbitrarily small opening.
\begin{proposizione}\label{prop:cono}
For any $\alpha>0$ and any $x\in E(\vartheta,\gamma)$ for which $\mathrm{Tan}_h(\phi,x)\subseteq \{\lambda \mathcal{S}^h\llcorner \mathbb V(x):\lambda>0\}$ for some $\mathbb V(x)\in\G(h)$, there exists a $\rho(\alpha,x)>0$ such that whenever $0<r<\rho$ we have
$$
E(\vartheta,\gamma)\cap B(x,r)\subseteq xC_{\mathbb V(x)}(\alpha).
$$

\end{proposizione}

\begin{proof}

By using \cref{prop:TanAndDxr}, we conclude that
$$
\lim_{r\to 0} \inf_{\Theta>0}\frac{ F_{x,r}(\phi, \Theta\mathcal{S}^h\llcorner x\mathbb V(x))}{r^{h+1}}=0.
$$
From the previous equality it follows that for every $\varepsilon>0$ there exists $1/\gamma>r_0(\varepsilon)>0$ such that 
\begin{equation}\label{eqn:VeryImportant}
\inf_{\Theta>0} F_{x,r}(\phi,\Theta\mathcal{S}^h\llcorner x\mathbb V(x)) \leq \varepsilon r^{h+1}, \qquad \text{whenever}\,\, 0<r\leq r_0(\varepsilon).
\end{equation}
Now we aim at proving that, for $\varepsilon>0$ small enough, $E(\vartheta,\gamma)\cap B(x,r_0(\varepsilon)/4)\subseteq xC_{\mathbb V(x)}(\alpha)$. In order to prove this we notice that \eqref{eqn:VeryImportant} and \cref{prop1vsinfty} imply that, for $\varepsilon>0$ sufficiently small, the following inequality holds
\begin{equation}\label{eqn:Veryimportant2}
\sup_{p\in E(\vartheta,\gamma)\cap B(x,r/4)} \dist(p,x\mathbb V(x)) \leq \oldC{C:b1}(h,\vartheta)\varepsilon^{1/(h+1)}r, \qquad \text{whenever}\,\, 0<r\leq r_0(\varepsilon).
\end{equation}
Indeed, from \eqref{eqn:VeryImportant} it follows that $\mathbb V(x)\in \Pi_{\varepsilon}(x,r)$ for every $0<r\leq r_0$, see \cref{def:Pixr}; so that it suffices to choose $\varepsilon<\delta_{\mathbb G}=\vartheta^{-1}2^{-(4h+5)}$, see \cref{eqn:DefinitionOfDeltaG}, in order to apply \cref{prop1vsinfty} and conclude \eqref{eqn:Veryimportant2}.

Now let us take $\varepsilon<\delta_{\mathbb G}$ so small  that the inequality $8\oldC{C:b1}(h,\vartheta)\varepsilon^{1/(h+1)}<\alpha$ holds. We finally prove $E(\vartheta,\gamma)\cap B(x,r_0(\varepsilon)/4)\subseteq xC_{\mathbb V(x)}(\alpha)$. Indeed, let $p\in E(\vartheta,\gamma)\cap B(x,r_0(\varepsilon)/4)$, and $k\geq 3$ be such that $r_02^{-k}< \|x^{-1}\cdot p\|\leq r_02^{-k+1}$. Since $p\in E(\vartheta,\gamma)\cap B(x,(r_02^{-k+3})/4)$, from \eqref{eqn:Veryimportant2} we get
$$
d(p,x\mathbb V(x))\leq \oldC{C:b1}(h,\vartheta)\varepsilon^{1/(h+1)}r_02^{-k+3} \leq 8\oldC{C:b1}(h,\vartheta)\varepsilon^{1/(h+1)}\|x^{-1}\cdot p\|\leq \alpha \|x^{-1}\cdot p\|,
$$
thus showing the claim.
\end{proof}


We now prove a cone-type rectifiability criterion that will be useful in combination with the previous results in order to split the support of a $\mathscr{P}_h$ or a $\mathscr{P}_h^c$-rectifiable measures with sets that have the cone property. 
\begin{proposizione}[Cone-rectifiability criterion]\label{prop:RectifiabilityCriterion}
Suppose that $E$ is a closed subset of $\mathbb G$ for which there exists a countable family $\mathscr{F}\subseteq \G(h)$ and a function $\alpha:\mathscr{F}\to (0,1)$ such that for every $x\in E$ there exist $\rho(x)>0$, and $\mathbb{V}(x)\in\mathscr{F}$ for which
\begin{equation}\label{eqn:HypothesisCvAlpha}
B(x,r)\cap E \subseteq xC_{\mathbb{V}(x)}(\alpha(\mathbb{V}(x))),
\end{equation}
whenever $0<r<\rho(x)$. Then, there are countably many compact $C_{\mathbb{V}_i}(3\beta_i)$-sets $\Gamma_i$ such that $\mathbb V_i\in\mathscr{F}$, and $\alpha(\mathbb V_i)<\beta_i<2\alpha(\mathbb{V}_i)$ for which
\begin{equation}
E= \bigcup_{i\in\N} \Gamma_i.
\label{eqn:FiniteUNion}
\end{equation}
\end{proposizione}

\begin{proof}
Let us split $E$ in the following way. Let $G(i,j,k)$ be the subset of those $x\in E\cap B(0,k)$ for which
\begin{equation}
B(x,r)\cap E \subseteq xC_{\mathbb V_i}(\alpha(\mathbb{V}_i)),
\nonumber
\end{equation}
for any $0<r<1/j$. Then, from the hypothesis, it follows $E=\cup_{i,j,k \in\mathbb N} G(i,j,k)$. Since $E$ is closed, it is not difficult to see that $G(i,j,k)$ is closed too. Let us fix $i,j,k\in\N$, some $\beta_i<1$ with $\alpha(\mathbb{V}_i)<\beta_i<2\alpha(\mathbb V_i)$, and let us prove that $G(i,j,k)$ can be covered with countably many compact $C_{\mathbb{V}_i}(3\beta_i)$-sets. Since $i,j,k\in\N$ are fixed from now on we assume without loss of generality that $G(i,j,k)=E$ so that we can drop the indeces.

Let us take $\{q_\ell\}$ a dense subset of $E$, and let us define the closed tubular neighbourhood of $q_\ell\mathbb{V}$
\begin{equation}\label{eqn:DefinitionSlijk}
S(\ell):=B(q_\ell \mathbb{V},2^{-\kappa}j^{-\kappa}\oldC{c:1}(14k)^{-\kappa}\beta^\kappa),
\end{equation}
where we recall that $\kappa$ is the step of the group, and where $\oldC{c:1}$ is defined in \eqref{lem:EstimateOnConjugate}. 
We will now prove that $S(\ell)\cap E$ is a  $C_{\mathbb V}(3\beta)$-set, or equivalently that for any $p\in S(\ell)\cap E$ we have
\begin{equation}\label{eqn:CONE}
S(\ell)\cap E\subseteq p\cdot C_{\mathbb V}(3\beta).
\end{equation}
If $q\in S(\ell)\cap E\cap B(p,1/(2j))$, the inclusion \eqref{eqn:CONE} holds thanks to our assumptions on $E$. 
If on the other hand $q\in S(\ell)\cap E\setminus B(p,1/(3j))$, let $p^*,q^*\in \mathbb V$ be such that $d(p,q_\ell\mathbb V)=\|(p^*)^{-1} q_\ell^{-1} p\|$, and $d(q,q_\ell \mathbb V)=\|(q^*)^{-1} q_{\ell}^{-1} q\|$.  Let us prove that $\lVert q^*\rVert\leq 4k$ and $\lVert p^*\rVert\leq 4k$. This is due to the fact that
$$
\lVert q^*\rVert-\lVert q_\ell\rVert-\lVert q\rVert\leq\lVert (q^*)^{-1} q_{\ell}^{-1}q\rVert=d(q,q_\ell\mathbb{V})\leq 1,
$$
where the last inequality follows from the definition of $S(\ell)$, see \eqref{eqn:DefinitionSlijk}. From the previous inequality it follows that  $\lVert q^*\rVert\leq 2k+1$, since $q,q_\ell\in B(0,k)$. A similar computation proves the bound for $\lVert p^*\rVert$ and this implies that
$$
\|p^{-1}\cdot q_\ell\cdot p^*\|+\|(p^*)^{-1}\cdot q^*\| \leq \|p^{-1}\|+\|q_\ell\|+2\|p^*\|+\|q^*\|\leq 14k.
$$
The application of \cref{lem:EstimateOnConjugate} and the fact that $(q^*)^{-1}p^*$ and $p^{-1} q_\ell  p^*$ are in $B(0,14k)$, due to the previous inequality, imply that
\begin{equation}\label{eqn:LongComputation}
\begin{split}
d(p^{-1} q, \mathbb V) &\leq \|(q^*)^{-1} p^*\cdot p^{-1}q\| = \|(q^*)^{-1}p^*\cdot p^{-1}\cdot q_\ell  p^* (p^*)^{-1} q^*  (q^*)^{-1} q_\ell^{-1}\cdot q\| \\
&\leq \|(q^*)^{-1}p^*\cdot p^{-1} q_\ell  p^* \cdot(p^*)^{-1} q^*\| + \|(q^*)^{-1} q_\ell^{-1} q\| \\
&\leq \oldC{c:1}(14k)\|p^{-1} q_\ell p^*\|^{1/\kappa}+d(q,q_\ell\mathbb V) \\&= \oldC{c:1}(14k)d(p,q_\ell\mathbb V)^{1/\kappa}+d(q,q_\ell\mathbb V).
\end{split}
\end{equation}
Finally, thanks to \eqref{eqn:DefinitionSlijk} and \eqref{eqn:LongComputation} we infer
$$d(p^{-1}q,\mathbb{V})\leq \frac{\oldC{c:1}(14k)+1}{2j\oldC{c:1}(14k)} \beta\leq \beta j^{-1}\leq 3\beta \lVert p^{-1}q\rVert,$$
thus showing \eqref{eqn:CONE} in the remaining case.
In conclusion we have proved that for any $i,j,k,\ell\in\N$, the sets $G(i,j,k)\cap S(\ell)$ are $C_{\mathbb{V}_i}(3\beta_i)$-sets. This concludes the proof since
$$
E\subseteq \bigcup_{i,j,k,\ell\in\N} G(i,j,k)\cap S(\ell),
$$
and on the other hand every $G(i,j,k)\cap S(\ell)$ is a bounded and closed, thus compact, $C_{\mathbb V_i}(3\beta_i)$-set. The fact that the sets $G(i,j,k)\cap S(\ell)$ are contained in $E$ follows by definition, thus concluding the proof of the equality. 
\end{proof}

In the following, with the symbol $\mathrm{Sub}(h)$, we denote the subset of $\G_c(h)\times\G_c(Q-h)$ defined by
\begin{equation}\label{eqn:defBoldG}
\{(\mathbb V,\mathbb L): \mathbb V \in \G_c(h)\,\,\text{and}\,\,\mathbb L\,\,\text{is a homogeneous subgroup complementary to $\mathbb V$}\},
\end{equation}
 we fix a function $\bold F:\mathrm{Sub}(h)\to (0,1)$, and for every 
$\ell\in\N$ with $\ell\geq 2$ let us define
$$
\G_c^{\bold F}(h,\ell):=\{\mathbb V\in \G_c(h):\exists\, \mathbb L\,\, \text{complement of}\,\, \mathbb V
 \,\,\text{s.t.}\,\, 1/\ell <\bold F(\mathbb V,\mathbb L)\leq 1/(\ell-1)\}.
$$ 
Observe that \cref{prop:CompGrassmannian} implies that $\G_c^{\bold F}(h,\ell)$ is separable for any $\ell\in\N$, since $\G_c^{\bold F}(h,\ell)\subseteq \G(h)$ and $(\G(h),d_{\mathbb G})$ is a compact metric space, see \cref{prop:CompGrassmannian}. Let 
 \begin{equation}\label{eqn:Dl}
 \mathscr{D}_\ell:=\{\mathbb V_{i,\ell}\}_{i\in\N},
 \end{equation} 
 be a countable dense subset of $\G_c^{\bold F}(h,\ell)$ and
 \begin{equation}\label{def:Ll}
     \text{for all $i\in\N$, choose a complement $\mathbb L_{i,\ell}$ of $\mathbb V_{i,\ell}$ s.t. $1/\ell <\bold F(\mathbb V_{i,\ell},\mathbb L_{i,\ell})\leq 1/(\ell-1)$}.
 \end{equation}
 
 

 Let us now prove the following theorem which will be of fundamental importance to prove \cref{coroll:IdiffMeasureIntroNEW}.
\begin{teorema}\label{thm:MainTheorem1}
Let $\bold F:\mathrm{Sub}(h)\to (0,1)$ be a function, where $\mathrm{Sub}(h)$ is defined in \eqref{eqn:defBoldG}, and for every $\ell\in \mathbb N$ define $\mathscr{D}_\ell$ as in \eqref{eqn:Dl}, set $\mathscr F:=\{\mathbb V_{i,\ell}\}_{i,\ell\in\mathbb N}$, 
and choose $\mathbb L_{i,\ell}$ as in \eqref{def:Ll}. Furthermore, let $\beta:\mathbb N\to (0,1)$ and define $\beta(\mathbb V_{i,\ell}):=\beta (\ell)$ for every $i,\ell\in \mathbb N$. For the ease of notation we rename $\mathscr{F}:=\{\mathbb V_k\}_{k\in\mathbb N}$. Then the following holds. 

Let $\phi$ be a $\mathscr{P}_h^c$-rectifiable measure. There are countably many compact sets $\Gamma_k$ that are $C_{\mathbb V_k}(\min\{\bold F(\mathbb{V}_k,\mathbb L_k),\beta(\mathbb V_k)\})$-sets for some $\mathbb V_k\in \mathscr{F}$, and such that
$$\phi\Big(\mathbb{G}\setminus \bigcup_{k=1}^{+\infty}\Gamma_k\Big)=0.$$
\end{teorema}

\begin{proof}
Let us notice that without loss of generality, by restricting the measure on balls with integer radius, we can suppose that $\phi$ has a compact support. Fix $\vartheta,\gamma\in\N$ and let $E(\vartheta,\gamma)$ be the set introduced in \cref{def:EThetaGamma} with respect to $\phi$. Furthermore, for any $\ell,i,j\in\N$, we let 
\begin{equation}\label{eqn:DefinitionFlij}
\begin{split}
F_{\ell}(i,j):=\{x\in E(\vartheta,\gamma): B(x,r)\cap E(\vartheta,\gamma)\subseteq xC_{\mathbb{V}_{i,\ell}}(6^{-1}\min\{\bold F(\mathbb{V}_{i,\ell},\mathbb L_{i,\ell}),\beta(\mathbb V_{i,\ell}) \text{for any}\,\, 0<r<1/j
\}.
\end{split}
\end{equation}
It is not hard to prove, since $E(\vartheta,\gamma)$ is compact, see \cref{prop:cpt}, that for every $\ell,i,j$ the sets $F_{\ell}(i,j)$ are compact. We claim that
\begin{equation}
    \phi\Big(E(\vartheta,\gamma)\setminus\bigcup_{\ell,i,j\in\N} F_\ell(i,j)\Big)=0.
    \label{eq:covering1}
\end{equation}
Indeed, let $w\in E(\vartheta,\gamma)$ be such that
$\mathrm{Tan}_h(\phi,w)\subseteq \{\lambda \mathcal{S}^{h}\llcorner \mathbb V(w):\lambda>0\}$
for some $\mathbb V(w)\in \G_c(h)$; this can be done for $\phi$-almost every point $w$ in $E(\vartheta,\gamma)$ since $\phi$ is $\mathscr{P}_h^c$-rectifiable. Let $\ell(w)\in\N$ be the smallest natural number for which there exists $\mathbb L$ complementary to $\mathbb V(w)$ with  $1/\ell(w)<\bold F(\mathbb{V}(w),\mathbb L)\leq 1/(\ell(w)-1)$. Then by definition we have $\mathbb V(w)\in \G_c^{\bold F}(h,\ell(w))$. By density of the family $\mathscr{D}_{\ell(w)}$ in $\G_c^{\bold F}(h,\ell(w))$ there exists a plane $\mathbb V_{i,\ell(w)}\in\mathscr{D}_{\ell(w)}$ such that
$$
d_{\mathbb G}(\mathbb{V}_{i,\ell(w)},\mathbb{V}(w))<30^{-1}\min\{1/\ell(w),\beta(\mathbb V_{i,\ell(w)})\};
$$
for this last observation to hold it is important that $\beta$ only depends on $\ell(w)$, as it is by construction.
The previous inequality, jointly with \cref{lem:DistanceOfCones}, imply that
\begin{equation}\label{eq:split1}
\begin{split}
    C_{\mathbb{V}(w)}(30^{-1}\min\{1/\ell(w),\beta(\mathbb V_{i,\ell(w)})\}) &\subseteq C_{\mathbb V_{i,\ell(w)}}(6^{-1}\min\{1/\ell(w),\beta(\mathbb V_{i,\ell(w)})\})  \\
    &\subseteq C_{\mathbb V_{i,\ell(w)}}(6^{-1}\min\{\bold F(\mathbb V_{i,\ell(w)},\mathbb L_{i,\ell(w)}),\beta(\mathbb V_{i,\ell(w)})\}), \\
    \end{split}
\end{equation}
where the last inclusion follows from the fact that by definition of the family $\mathscr{D}_{\ell(w)}$ it holds $\bold F(\mathbb V_{i,\ell(w)},\mathbb L_{i,\ell(w)})>1/\ell(w)$. Thanks to \cref{prop:cono} we can find a $\rho(w)>0$ such that for any $0<r<\rho(w)$ we have
\begin{equation}
     B(w,r)\cap E(\vartheta,\gamma)\subseteq  wC_{\mathbb{V}(w)}(30^{-1}\min\{1/\ell(w),\beta(\mathbb V_{i,\ell(w)})\}).
     \label{eq:split2}
\end{equation}
In particular, putting together \eqref{eq:split1} and \eqref{eq:split2} we infer that for $\phi$-almost every $w\in E(\vartheta,\gamma)$ there are an $i=i(w)>0$, an $\ell(w)\in\N$ and a $\rho(w)>0$ such that whenever $0<r<\rho(w)$ we have
$$
B(w,r)\cap E(\vartheta,\gamma)\subseteq wC_{\mathbb{V}_{i,\ell(w)}}(6^{-1}\min\{\bold F(\mathbb{V}_{i,\ell(w)},\mathbb L_{i,\ell(w)}),\beta(\mathbb V_{i,\ell(w)})\}).
$$
This concludes the proof of \eqref{eq:covering1}. 

Now, if we fix $\ell,i,j\in \mathbb N$, we can apply \cref{prop:RectifiabilityCriterion} to the set $F_{\ell}(i,j)$. It suffices to take the family $\mathscr{F}$ in the statement of \cref{prop:RectifiabilityCriterion} to be the singleton $\{\mathbb V_{i,\ell}\}$ and the function $\alpha$ in the statement of \cref{prop:RectifiabilityCriterion} to be $\alpha(\mathbb V_{i,\ell}):=6^{-1}\min\{\bold F(\mathbb V_{i,\ell},\mathbb L_{i,\ell}),\beta(\mathbb V_{i,\ell})\}$. As a consequence we can write each $F_{\ell}(i,j)$ as the union of countably many compact $C_{\mathbb V_{i,\ell}}(\min\{\bold F(\mathbb V_{i,\ell},\mathbb L_{i,\ell}),\beta(\mathbb V_{i,\ell})\})$-sets. Thus the same holds $\phi$-almost everywhere for $E(\vartheta,\gamma)$, allowing $i,\ell$ to vary in $\mathbb N$, since \eqref{eq:covering1} holds. Finally, we have $$
\phi(\mathbb G\setminus \cup_{\vartheta,\gamma\in\mathbb N}E(\vartheta,\gamma))=0,
$$
due to \cref{prop::E}. 

Thus we can cover $\phi$-almost all of $\mathbb G$ with compact $C_{\mathbb V_{i,\ell}}(\min\{\bold F(\mathbb V_{i,\ell},\mathbb L_{i,\ell}),\beta(\mathbb V_{i,\ell})\})$-sets for $i,\ell$ that vary in $\mathbb N$, concluding the proof of the proposition.
\end{proof}
The following theorem is a more detailed version of \cref{thm:MainTheorem1Intro}.
\begin{teorema}\label{thm:MainTheorem2}
There exists a countable subfamily $\mathscr{F}:=\{\mathbb V_k\}_{k\in\mathbb N}$ of $\G(h)$ such that the following holds. Let $\phi$ be a $\mathscr{P}_h$-rectifiable measure. For any $0<\beta<1$ there are countably many compact sets $\Gamma_k$ that are $C_{\mathbb V_k}(\beta)$-sets for some $\mathbb V_k\in \mathscr{F}$, and such that
$$\phi\Big(\mathbb{G}\setminus \bigcup_{k=1}^{+\infty}\Gamma_k\Big)=0.$$
\end{teorema}

\begin{proof}
The proof is similar to the one of \cref{thm:MainTheorem1}. It suffices to choose, as a family $\mathscr{F}$, an arbitrary countable dense subset of $\G(h)$ and then one can argue as in the proof of \cref{thm:MainTheorem1} without the technical effort of introducing the parameter $\ell$. We skip the deatils.
\end{proof}


\section{Bounds for the densities of $\mathcal{S}^h$ on $C_\mathbb{V}(\alpha)$-sets}\label{sec:Density}
\textbf{Throughout this subsection we assume that $\mathbb{V}\in\G_c(h)$ and that $\mathbb{V}\cdot \mathbb{L}=\mathbb{G}$. In this chapter whenever we deal with $C_{\mathbb V}(\alpha)$-sets we are always assuming that $\alpha<\oldep{ep:Cool}(\mathbb V,\mathbb L)$, where $\oldep{ep:Cool}$ is defined in \cref{lemma:LCapCw=e}}.
\medskip

This section is devoted to the proof of \cref{coroll:DENSITYIntro}, that is obtained through three different steps. Let $\Gamma$ be a compact $C_{\mathbb V}(\oldep{ep:Cool}(\mathbb V,\mathbb L))$ set, and recall that by \cref{prop:ConeAndGraph} we can write $\Gamma=\mathrm{graph}(\varphi)$ with $\varphi:P_{\mathbb V}(\Gamma)\to \mathbb L$. Let us denote $\Phi(v):=v\cdot\varphi(v)$ for every $v\in P_{\mathbb V}(\Gamma)$. 

We first show that \textbf{if we assume that $\Theta_*^h(\mathcal{S}^h\llcorner \Gamma,x)>0$ at $\mathcal{S}^h\llcorner \Gamma$-almost every point $x$}, then the push-forward measure $(\Phi)_*(\mathcal S^h\llcorner\mathbb V)$ is mutually absolutely continuous with respect to $\mathcal{S}^h\llcorner\Gamma$, see \cref{prop:mutuallyabs}.  In other words we are proving that whenever an intrinsically Lipschitz graph over a subset of an $h$-dimensional subgroup has strictly positive lower density almost everywhere, then the push-forward of the measure $\mathcal{S}^h$ on the subgroup by means of the graph map is mutually absolutely continuous with respect to the measure $\mathcal{S}^h$ on the graph. We stress that we do not address the issue of removing the hypothesis on the strict positivity of the lower density in \cref{prop:mutuallyabs} as it is out of the aims of this paper. We remark that in the Euclidean case the analogous statement holds true without this assumption: this is true because in the Euclidean case every Lipschitz graph over a subset of a vector subspace of dimension $h$ has stricitly positive lower $h$-density almost everywhere. We also stress that every intrinsically Lipschitz graph over a open subset of a $h$-dimensional homogeneous subgroups has strictly positive lower $h$-density almost everywhere, see \cite[Theorem 3.9]{FranchiSerapioni16}. 

As a second step in order to obtain the proof of \cref{coroll:DENSITYIntro} we prove the following statement that can be made quantitative: if $\mathbb V\in \G_c(h)$, $\Gamma$ is a compact $C_{\mathbb V}(\alpha)$-set with $\alpha$ sufficiently small, and $\mathcal{S}^h\llcorner \Gamma$ is a $\mathscr P_h$-rectifiable measure with complemented tangents, which we called $\mathscr{P}_h^c$-rectifiable, then we can give an explicit lower bound of the ratio of the lower and upper $h$-densities of $\mathcal{S}^h\llcorner\Gamma$. We refer the reader to \cref{prop:dens1} for a more precise statement and the proof of the following proposition.
\begin{proposizione}[Bounds on the ratio of the densities]\label{prop:dens1INTRO}
Let $\mathbb V$ be in $\G_c(h)$. There exists $C:=C(\mathbb V)$ such that the following holds. Suppose $\Gamma$ is a compact $C_\mathbb{V}(\alpha)$-set with $\alpha<C(\mathbb V)$ and such that $\mathcal{S}^h\llcorner \Gamma$ is a $\mathscr{P}_h^c$-rectifiable measure. Then there exists a continuous function $\omega:=\omega(\mathbb V)$ of $\alpha$, with $\omega(0)=0$, such that for $\mathcal{S}^h$-almost every $x\in \Gamma$ we have
\begin{equation}
    1-\omega(\alpha)\leq \frac{\Theta^h_*(\mathcal{S}^h\llcorner \Gamma,x)}{\Theta^{h,*}(\mathcal{S}^h\llcorner \Gamma,x)}\leq 1.
    \label{eq:dens1INTRO}
\end{equation}
\end{proposizione}
The previous result is obtained through a blow-up analysis and a careful use of the mutually absolute continuity property that we discussed above, and which is contained in \cref{prop:mutuallyabs}. We stress that in order to differentiate in the proof of \cref{prop:dens1INTRO}, we need to use proper $\mathcal{S}^h\llcorner P_{\mathbb V}(\Gamma)$ and $\mathcal{S}^h\llcorner \mathbb V$-Vitali relations, see \cref{prop:vitalirel}, and \cref{prop:vitali2PD}, respectively. 

As a last step of the proof of \cref{coroll:DENSITYIntro} we first use the result in \cref{prop:dens1INTRO} in order to prove that \cref{coroll:DENSITYIntro} holds true for measures of the type $\mathcal{S}^h\llcorner\Gamma$, see \cref{thm:ExistenceOfDensitySets}. Then we conclude the proof for arbitrary measures by reducing ouserlves to the sets $E(\vartheta,\gamma)$, see \cref{coroll:DENSITY}. The last part about the convergence in \cref{coroll:DENSITYIntro} readily comes from the first part and \cref{prop:density}.

We start this chapter with some lemmata.
\begin{lemma}\label{lemma:tec.cono}
There exists an $A:= A(\mathbb V,\mathbb L)>1$ such that for any $w\in B(0,1/5A)$, any $y\in \partial B(0,1)\cap C_{\mathbb{V}}(\oldep{ep:Cool}(\mathbb{V},\mathbb{L}))$ and any $z\in B(y,1/5A)$, we have
$w^{-1}z\not\in \mathbb{L}$.
\end{lemma}

\begin{proof}
By contradiction let us assume that we can find sequences $\{w_n\}$, $\{y_n\}\subseteq \partial B(0,1)\cap C_{\mathbb{V}}(\oldep{ep:Cool})$ and $z_n\in B(y_n,1/n)$ such that $w_n$ converges to $0$ and $w_n^{-1}z_n\in\mathbb{L}$. By compactness without loss of generality we can assume that the sequence $y_n$ converges to some $y\in\partial B(0,1) \cap C_\mathbb{V}(\oldep{ep:Cool})$. Furthermore, by construction we also have that $z_n$ must converge to $y$. This implies that $w_n^{-1}z_n$ converges to $y$ and since by hypothesis $w_n^{-1}z_n\in\mathbb L$, thanks to the fact that $\mathbb{L}$ is closed we infer that $y\in\mathbb{L}$. This however is a contradiction since $y$ has unit norm and at the same time we should have $y\in  C_{\mathbb{V}}(\oldep{ep:Cool})\cap \mathbb{L}=\{0\}$ by \cref{lemma:LCapCw=e}.
\end{proof}

\begin{proposizione}\label{prop:palleseparate}
Let $\alpha<\oldep{ep:Cool}(\mathbb V,\mathbb L)$ and suppose $\Gamma$ is a compact $C_\mathbb{V}(\alpha)$-set. For any $x\in \Gamma$ let $\rho(x)$ to be the biggest number satisfying the following condition. For any $y\in B(x,\rho(x))\cap \Gamma$ we have
    $$P_\mathbb{V}(B(x,r)) \cap P_\mathbb{V}(B(y,s))=\emptyset\,\, \text{ for any }r,s<d(x,y)/5A,$$
where $A=A(\mathbb V,\mathbb L)$ is the constant yielded by \cref{lemma:tec.cono}. Then, the function $x\mapsto \rho(x)$ is positive everywhere on $\Gamma$ and upper semicontinuous.  
\end{proposizione}

\begin{proof}
Let $x\in \Gamma$ and suppose by contradiction that there is a sequence of points $\{y_i\}_{i\in\N}\subseteq \Gamma$ converging to $x$ and
\begin{equation}
    P_{\mathbb{V}}(B(x,r_i))\cap P_{\mathbb{V}}(B(y_i,s_i))\neq \emptyset,
    \label{eq:intersectionconi}
\end{equation}
for some $r_i,s_i<d(x,y_i)/5A$. We note that \eqref{eq:intersectionconi} is equivalent to assuming that there are $z_i\in B(x,r_i)$ and $w_i\in B(y_i,s_i)$ such that
\begin{equation}
    P_{\mathbb{V}}(w_i)=P_{\mathbb{V}}(z_i).
    \label{eq:1}
\end{equation}
Identity \eqref{eq:1} implies in particular that for any $i\in\N$ we have $w_i^{-1}z_i\in \mathbb{L}$ and let us denote $\rho_i:=d(x,y_i)$.
Thanks to the assumptions on $y_i,z_i$ and $w_i$ we have that
\begin{itemize}
\item[(1)] $d(0,\delta_{1/\rho_i}(x^{-1}y_i))=1$ and thus we can assume without loss of generality that there exists a $y\in \partial B(0,1)$ such that
$$\lim_{i\to \infty}\delta_{1/\rho_i}(x^{-1}y_i)=y,$$
 \item[(2)]$d(0,\delta_{1/\rho_i}(x^{-1}z_i))\leq 1/5A$ and thus up to passing to a non-relabelled subsequence we can assume that there exists a $z\in B(0,1/5A)$ such that
 $$\lim_{i\to\infty} \delta_{1/\rho_i}(x^{-1}z_i)=z,$$
    \item[(3)] $d(\delta_{1/\rho_i}(x^{-1}y_i),\delta_{1/\rho_i}(x^{-1}w_i))\leq 1/5A$ and thus, up to passing to a non re-labelled subsequence, we can suppose that there exists a $w\in B(y,1/5A)$ such that
    $$\lim_{i\to\infty} \delta_{1/\rho_i}(x^{-1}w_i)=w.$$
\end{itemize}

Since $\Gamma$ is supposed to be a $C_\mathbb{V}(\alpha)$-set, we have that for any $i\in\N$ the point $x^{-1}y_i$ is contained in the cone $C_\mathbb{V}(\alpha)$ and, since $\mathbb{C}_\mathbb{V}(\alpha)$ is closed, we infer that $y\in C_\mathbb{V}(\alpha)$. Since we assumed $\alpha<\oldep{ep:Cool}(\mathbb{V},\mathbb{L})$, we have $y\in\partial B(0,1)\cap C_\mathbb{V}(\oldep{ep:Cool}(\mathbb{V},\mathbb{L}))$.
Since $\delta_{1/\rho_i}(x^{-1}z_i)$ and $\delta_{1/\rho_i}(x^{-1}w_i)$ converge to $z$ and $w$, respectively, we have
$$
\lim_{i\to\infty}\delta_{1/\rho_i}(w_i^{-1}z_i)=\lim_{i\to\infty}\delta_{1/\rho_i}(w_i^{-1}x)\delta_{1/\rho_i}(x^{-1}z_i)=w^{-1}z.
$$
Furthermore since $w_i^{-1}z_i\in \mathbb{L}$ for any $i\in\N$, we infer that $w^{-1}z\in\mathbb{L}$ since $\mathbb L$ is closed. Applying \cref{lemma:tec.cono} to $y,z,w$ we see that the fact that $w^{-1}z\in\mathbb{L}$, $z\in B(0,1/5A)$ and $w\in B(y,1/5A)$ results in a contradiction. This concludes the proof of the first part of the proposition.

In order to show that $\rho$ is upper semicontinuous we fix an $x\in \Gamma$ and we assume by contradiction that there exists a sequence $\{x_i\}_{i\in\N}\subseteq \Gamma$ converging to $x$ such that
\begin{equation}
    \limsup_{i\to\infty} \rho(x_i)>(1+\tau)\rho(x),
    \label{eq:absupsc}
\end{equation}
for some $\tau>0$. Fix an $y\in B(x,(1+\tau/2)\rho(x))\cap\Gamma$ and assume $s,r<d(x,y)/5A$. Thus, thanks to \eqref{eq:absupsc} and the fact that the $x_i$ converge to $x$, we infer that there exists a $i_0\in\mathbb{N}$ such that, up to non re-labelled subsequences, for any $i\geq i_0$ we have $\rho(x_i)>(1+\tau)\rho(x)$, $d(x_i,x)<\tau\rho(x)/4$ and $s,r+d(x_i,x)<d(x_i,y)/5A$. Therefore, for any $i\geq i_0$ we have
\begin{equation}
    y\in B(x_i,(1+3\tau/4)\rho(x))\subseteq B(x_i,\rho(x_i)),\qquad\text{and}\qquad s,r+d(x_i,x)<d(x_i,y)/5A.\nonumber
\end{equation}
This however, thanks to the definition of $\rho(x_i)$, implies that
$$P_\mathbb{V}(B(x,r))\cap P_\mathbb{V}(B(y,s))\subseteq P_\mathbb{V}(B(x_i,r+d(x_i,x)))\cap P_\mathbb{V}(B(y,s))=\emptyset. $$
Summing up, we have proved that for any $y\in B(x,(1+\tau/2)\rho(x))\cap\Gamma$ whenever $r,s<d(x,y)/5A$ we have
$$P_\mathbb{V}(B(x,r))\cap P_\mathbb{V}(B(y,s))=\emptyset,$$
and this contradicts the maximality of $\rho(x)$. This concludes the proof.

\end{proof}

\begin{corollario}\label{cor:pervitali}
Let us fix $\alpha<\oldep{ep:Cool}(\mathbb V,\mathbb L)$ and suppose that $\Gamma$ is a compact $C_\mathbb{V}(\alpha)$-set. Let us fix $x\in \Gamma$ and choose $\rho(x)>0$ as in the statement of \cref{prop:palleseparate}. Then there is a $0<r(x)<1/2$ such that the following holds
\begin{equation}
\begin{split}
    &\text{if $0<r<r(x)$ and $y\in \Gamma$ are such that}\,\, P_\mathbb{V}(B(x,2r)) \cap P_\mathbb{V}(B(y,10r))\neq\emptyset,\,\,\\ &\text{then $y\in B(x,\rho(x))$ and $d(x,y)\leq 50Ar$},
\end{split}
\end{equation}
where $A=A(\mathbb V,\mathbb L)$ is the constant yielded by \cref{lemma:tec.cono}.
\end{corollario}

\begin{proof}
Let us first prove that there exists $\widetilde\alpha:=\widetilde\alpha(\alpha,x)$ such that whenever $y\in\Gamma$ is such that $d(x,y)\geq \rho(x)$ then $d(P_{\mathbb V}(x),P_{\mathbb V}(y))\geq \widetilde\alpha$. Indeed if it is not the case, we have a sequence $\{y_i\}_{i\in\mathbb N}\subseteq \Gamma$ such that $d(x,y_i)\geq \rho (x)$ for every $i\in\N$ and $d(P_{\mathbb V}(x),P_{\mathbb V}(y_i))\to 0$ as $i\to +\infty$. Since $\Gamma$ is compact we can suppose, up to passing to a non re-labelled subsequence, that $y_i\to y\in \Gamma$. Moreover since $d(x,y_i)\geq \rho(x)$ and $d(P_{\mathbb V}(x),P_{\mathbb V}(y_i))\to 0$ we conclude that $d(x,y)\geq \rho(x)$, and hence $x\neq y$, and moreover $P_{\mathbb V}(x)=P_{\mathbb V}(y)$. Then $y^{-1}\cdot x\in \mathbb L\cap C_{\mathbb V}(\alpha)$ that is a contradiction with \cref{lemma:LCapCw=e} because $y\neq x$ and $\alpha<\oldep{ep:Cool}$.

Since $P_\mathbb{V}$ is uniformly continuous on the closed tubular neighborhood $B(\Gamma,1)$, there exists a $r(x)>0$ depending on $\widetilde\alpha=\widetilde\alpha(\alpha,x)$ such that for any $y\in \Gamma$ and any $r<r(x)$, we have
\begin{equation}\label{eqn:INFER}
        P_\mathbb{V}(B(y,10r))\subseteq B(P_\mathbb{V}(y),\widetilde\alpha/10).
\end{equation}
Let us show the first part of the statement. It is sufficient to prove that if $r<r(x)$ and $y\in\Gamma$ is such that $d(x,y)\geq \rho(x)$, then $P_\mathbb{V}(B(x,2r)) \cap P_\mathbb{V}(B(y,10r))=\emptyset$. Indeed if $d(x,y)\geq \rho(x)$ then $d(P_{\mathbb V}(x),P_{\mathbb V}(y))\geq \widetilde\alpha$. Moreover, from \eqref{eqn:INFER} we deduce that $P_{\mathbb V}(B(x,10r))\subseteq B(P_{\mathbb V}(x),\widetilde\alpha/10)$ and $P_{\mathbb V}(B(y,10r))\subseteq B(P_{\mathbb V}(y),\widetilde\alpha/10)$. Since $d(P_{\mathbb V}(x),P_{\mathbb V}(y))\geq \widetilde\alpha$ we conclude that $B(P_{\mathbb V}(x),\widetilde\alpha/10)\cap B(P_{\mathbb V}(y),\widetilde\alpha/10)=\emptyset$ and then also $P_{\mathbb V}(B(x,10r))\cap P_{\mathbb V}(B(y,10r))=\emptyset$, from which the sought conclusion follows. 
In order to prove $d(x,y)\leq 50Ar$, once we have $y\in B(x,\rho(x))$, the conclusion follows thanks to \cref{prop:palleseparate}.
\end{proof}

\begin{lemma}\label{lemmaVitali}
Fix some $N\in\N$ and assume that $\mathscr{F}$ is a family of closed balls of $\mathbb{G}$ with uniformly bounded radii. Then we can find a countable disjoint subfamily $\mathscr{G}$ of $\mathscr{F}$ such that
\begin{itemize}
    \item[(i)] if $B,B^\prime \in\mathscr{G}$ then $5^NB$ and $5^N B^\prime$ are disjoint,
    \item[(ii)] $\bigcup_{B\in\mathscr{F}}B\subseteq \bigcup_{B\in\mathscr{G}} 5^{N+1}B$.
\end{itemize}
\end{lemma}

\begin{proof}
If $N=0$, there is nothing to prove, since it is the classical $5$-Vitali's covering Lemma. 

Let us assume by inductive hypothesis that the claim holds for $N=k$ and let us prove that it holds for $k+1$. 
Let $\mathscr{G}_k$ be the family of balls satisfying (i) and (ii) for $N=k$, and apply the $5$-Vitali's covering Lemma to the family of balls $\widetilde{\mathscr{F}}:=\{ 5^{k+1}B: B\in \mathscr{G}_k\}$. We obtain a countable subfamily $\widetilde{\mathscr{G}}$ of $\widetilde{\mathscr{F}}$ such that if $5^{k+1}B,5^{k+1}B^\prime\in\widetilde{\mathscr{G}}$ then $5^{k+1}B$ and $5^{k+1}B^\prime$ are disjoint and that satisfies $\bigcup_{B\in\widetilde{\mathscr{F}}} B\subseteq \bigcup_{B\in\widetilde{\mathscr{G}}} 5B$. Therefore, if we define
$$\mathscr{G}_{k+1}:=\{B\in\mathscr{G}_{k}:5^{k+1}B\in \widetilde{\mathscr{G}}\},$$
point (i) directly follows and thanks to the inductive hypothesis we have 
$$\bigcup_{B\in \mathscr{F}} B\subseteq \bigcup_{B\in \mathscr{G}_k} 5^{k+1} B\subseteq \bigcup_{B\in \mathscr{G}_{k+1}}5^{k+2} B,$$
proving the second point of the statement.
\end{proof}

\begin{proposizione}\label{prop:proj}
Let $\alpha<\oldep{ep:Cool}(\mathbb V,\mathbb L)$ and suppose $\Gamma$ is a compact $C_\mathbb{V}(\alpha)$-set of finite $\mathcal{S}^h$-measure such that
$$\Theta^h_*(\mathcal{S}^h\llcorner \Gamma,x)>0,$$
for $\mathcal{S}^h$-almost every $x\in \Gamma$. Then, there exists a constant $\newC\label{C:projji}>0$ depending on $\mathbb{V}$, $\mathbb{L}$, and the left-invariant homogeneous distance on $\mathbb G$, such that for $\mathcal{S}^h$-almost every $x\in \Gamma$ there exists an $R:=R(x)>0$ such that for any $0<\ell\leq R$ we have
\begin{equation}
    \mathcal{S}^h(P_{\mathbb{V}}(\Gamma\cap B(x,\ell)))\geq\oldC{C:projji}\Theta^{h}_*(\mathcal{S}^h\llcorner \Gamma,x)^2\ell^h.
    \label{eq:bigproj}
\end{equation}
\end{proposizione}

\begin{proof}
First of all, let us recall that two homogeneous left-invariant distances are always bi-Lipschitz equivalent on $\mathbb G$. Therefore if $d_c$ is a Carnot-Carathéodory distance on $\mathbb G$, which is in particular geodesic, see \cite[Section 3.3]{LD17} there exists a constant $\mathfrak{L}(d,d_c)\geq 1$ such that
$$\mathfrak{L}(d,d_c)^{-1}d_c(x,y)\leq d(x,y)\leq \mathfrak{L}(d,d_c)d_c(x,y)\text{ for any }x,y\in\mathbb G.$$

We claim that if for any $\vartheta,\gamma\in \N$ for which $\mathcal{S}^h(E(\vartheta,\gamma))>0$ we have that for $\mathcal{S}^h$-almost any $w\in E(\vartheta,\gamma)$ there exists a $R(w)>0$ such that
\begin{equation}
     \mathcal{S}^h(P_{\mathbb{V}}(\Gamma\cap B(w,\ell)))\geq\frac{\oldC{ProjC}(\mathbb{V},\mathbb L) \ell^h}{8\cdot5^{3h} A^{h} \mathfrak{L}(d,d_c)^{2h}\vartheta^2},
     \label{eq:claimdens}
\end{equation}
whenever $0<\ell<R(w)$, then the proposition is proved. This is due to the following reasoning. First of all, thanks to \cite[Proposition 2.10.19(5)]{Federer1996GeometricTheory}, we know that $\Theta^{h,*}(\mathcal{S}^h\llcorner \Gamma,x)\leq 1$. Secondly, if we set, for any $k\in\N$,
$\Gamma_k:=\{w\in \Gamma: 1/(k+1)<\Theta^{h}_*(\mathcal{S}^h\llcorner \Gamma,x)\leq 1/k\}$, we have  that
\begin{equation}
\mathcal{S}^h(\Gamma\setminus \bigcup_{k\in\N} \Gamma_k)=0.
    \label{eq:inclproj1}
\end{equation}
We observe now that if $\mathcal{S}^h(\Gamma_k)>0$, then $\mathcal{S}^h$-almost every $w\in \Gamma_k$ belongs to some $E(k+1,\gamma)$ provided $\gamma$ is big enough, or in other words
\begin{equation}
    \mathcal{S}^h\big(\Gamma_k\setminus \bigcup_{\gamma\in\N} E(k+1,\gamma)\big)=0.
    \label{eq:inclproj2}
\end{equation}
If our claim \eqref{eq:claimdens} holds true, whenever $\mathcal{S}^h(E(k+1,\gamma))>0$, we have that for $\mathcal{S}^h\llcorner E(k+1,\gamma)$-almost every $w$ there exists $R(w)$ such that whenever $0<\ell <R(w)$ the following chain of inequalities holds
\begin{equation}
\begin{split}
       \mathcal{S}^h(P_{\mathbb{V}}(\Gamma\cap B(w,\ell)))&\geq\frac{\oldC{ProjC}(\mathbb{V},\mathbb L) \ell^h}{8\cdot5^{3h} A^{h} \mathfrak{L}(d,d_c)^{2h}(k+1)^2}\\
      &\geq \frac{\oldC{ProjC}(\mathbb{V},\mathbb L) \ell^h}{2^5\cdot5^{3h} A^{h} \mathfrak{L}(d,d_c)^{2h} k^2}\geq \frac{\oldC{ProjC}(\mathbb{V},\mathbb L)\Theta^{h}_*(\mathcal{S}^h\llcorner \Gamma,x)^2 \ell^h}{2^5\cdot5^{3h} A^{h} \mathfrak{L}(d,d_c)^{2h}}\\ &=\oldC{C:projji}\Theta^{h}_*(\mathcal{S}^h\llcorner \Gamma,x)^2\ell^h.
\end{split}
    \label{eq:bdprojji}
\end{equation}
Identities \eqref{eq:inclproj1} and \eqref{eq:inclproj2} together with \eqref{eq:bdprojji} imply that our claim suffices to prove the proposition. Therefore, in the following we will assume that $\vartheta,\gamma\in\N$ are fixed and such that $\mathcal{S}^h(E(\vartheta,\gamma))>0$, and we want to prove \eqref{eq:claimdens}.

Let $N\in \N$ be the unique natural number for which $5^{N-2}\leq A\mathfrak{L}(d,d_c)^2<5^{N-1}$ and for any $k\in\N$ and $0<\delta<1/2$ we define the following sets, where $\rho(x)$ is defined in \cref{prop:palleseparate},
\begin{equation}
    \begin{split}
        A_{\vartheta,\gamma}(k):=&\{x\in E(\vartheta,\gamma):\rho (x)>1/k\},\\
        \mathscr{D}_{\vartheta,\gamma}(k):=&\bigg\{x\in A_{\vartheta,\gamma}(k):\lim_{r\to 0}\frac{\mathcal{S}^{h}(B(x,r)\cap A_{\vartheta,\gamma}(k))}{\mathcal{S}^{h}(B(x,r)\cap E(\vartheta,\gamma))}=1\bigg\},\\
        \mathscr{F}_{\delta}(k):=&\bigg\{B(x,r):x\in \mathscr{D}_{\vartheta,\gamma}(k) \text{ and }r\leq \frac{\min\{\vartheta^{-1},\gamma^{-1},k^{-1},\delta\}}{1000 A\mathfrak{L}(d,d_c)^2}\bigg\}.
        \nonumber
    \end{split}
\end{equation}
For any $\vartheta,\gamma\in\N$ the sets $A_{\vartheta,\gamma}(k)$ are Borel since thanks to \cref{prop:palleseparate}, the function $\rho$ is upper semicontinuous. 
Before going on we observe that $\mathcal{S}^h\llcorner E(\vartheta,\gamma)(A_{\vartheta,\gamma}(k)\setminus\mathscr{D}_{\vartheta,\gamma}(k))=0 $. This comes from the fact that the points of $\mathscr{D}_{\vartheta,\gamma}(k)$ are exactly the points of density one of $A_{\vartheta,\gamma}(k)$ with respect to the measure $\mathcal{S}^h\llcorner E(\vartheta,\gamma)$, that is asymptotically doubling at $\mathcal{S}^h\llcorner E(\vartheta,\gamma)$-almost every point because it has positive lower density and finite upper density at $\mathcal{S}^h\llcorner E(\vartheta,\gamma)$-almost every point, see \cref{prop:Lebesuge}. Moreover observe that from \cref{prop:palleseparate} $\mathcal{S}^h(E(\vartheta,\gamma)\setminus\cup_{k=1}^{+\infty} A_{\vartheta,\gamma}(k))=0$. 
Let us apply \cref{lemmaVitali} to $N$ and $\mathscr{F}_\delta(k)$, and thus we infer that there exists a subfamily $\mathscr{G}_\delta(k)$ such that
\begin{itemize}
    \item[($\alpha$)] for any $B,B^\prime\in\mathscr{G}_\delta(k)$ we have that $5^NB\cap 5^NB^\prime=\emptyset$,
    \item[($\beta$)] $\bigcup_{B\in\mathscr{F}_\delta(k)} B\subseteq \bigcup_{B\in\mathscr{G}_\delta(k)} 5^{N+1}B$.
\end{itemize}
The point ($\alpha$) above implies in particular that whenever $B(x,r),B(y,s)\in\mathscr{G}_\delta(k)$ we have $d(x,y)>\mathfrak{L}(d,d_c)^{-2}5^{N}(r+s)$, since $d$ is $\mathfrak{L}(d,d_c)$-Lipschitz equivalent to the geodesic distance $d_c$, and thanks to the choice of $N$ we deduce that
$$r+s<\frac{d(x,y)}{5A}.$$
Throughout the rest of the proof we fix a $w\in \mathscr{D}_{\vartheta,\gamma}(k)$ and a $$0<R(w)<\min\{\vartheta^{-1},\gamma^{-1},k^{-1}\}/8,$$ 
such that
\begin{equation}\label{eqn:RequirementsOnR}
\frac{\mathcal{S}^{h}\llcorner \Gamma(B(w,\ell))}{\ell^h}\geq \frac{1}{2\vartheta}, ~ \text{and}~\frac{\mathcal{S}^{h}\llcorner \mathscr{D}_{\vartheta,\gamma}(k)(B(w,\ell))}{\mathcal{S}^{h}\llcorner \Gamma(B(w,\ell))}\geq \frac{1}{2},\quad\text{for any }0<\ell\leq R(w).
\end{equation}
For the ease of notation we continue the proof fixing the radius $\ell=R(w)=R$. We stress that the forthcoming estimates are verified, mutatis mutandis, also for any $0<\ell<R$. The first inequality above comes from the definition of $E(\vartheta,\gamma)$, see \cref{def:EThetaGamma}, while the second is true, up to choose a sufficiently small $R(w)$, because  $\mathcal{S}^h\llcorner \Gamma$-almost every point of $\mathscr{D}_{\vartheta,\gamma}(k)$ has density one with respect to the asymptotically doubling measure $\mathcal{S}^h\llcorner \Gamma$. Let us stress that if we prove our initial claim for such $w$ and $R(w)$ we are done since $\mathcal{S}^h\llcorner\Gamma$-every point of $\mathscr D_{\vartheta,\gamma}(k)$ satisfies \eqref{eqn:RequirementsOnR}, $\mathcal{S}^h\llcorner E(\vartheta,\gamma)(A_{\vartheta,\gamma}(k)\setminus\mathscr{D}_{\vartheta,\gamma}(k))=0$, and $\mathcal{S}^h(E(\vartheta,\gamma)\setminus \cup_{k=1}^{+\infty}A_{\vartheta,\gamma}(k))=0$.

Let us notice that the definition of $\mathscr{F}_\delta(k)$ implies that there must exist a ball $B\in\mathscr{G}_{\delta}(k)$ such that $w\in 5^{N+1}B $. We now prove that for any couple of closed balls $B(x,r),B(y,s)\in \mathscr{G}_\delta(k)$ such that $B(w,R)$ intersects both $B(x,5^{N+1}r)$ and $B(y,5^{N+1}s)$, we have
\begin{equation}
    P_\mathbb{V}(B(x,r))\cap P_{\mathbb{V}}(B(y,s))=\emptyset.
    \label{eq:palle}
\end{equation}
Indeed, Suppose that $p\in B(x,5^{N+1}r)\cap B(w,R)$ and note that
\begin{equation}\label{eq:stimad}
\begin{split}
        d(x,w)&\leq d(x,p)+d(p,w)\leq R+5^{N+1}r
          \\ &\leq \Big(\frac{1}{8}+\frac{5^{N+1}}{1000A\mathfrak{L}(d,d_c)^2}\Big)\min\{\vartheta^{-1},\gamma^{-1},k^{-1}\}\leq\frac{\min\{\vartheta^{-1},\gamma^{-1},k^{-1}\}}{4},
        \end{split}
\end{equation}
where the last inequality comes from the choice of $N$. The bound \eqref{eq:stimad} shows in particular that
$$d(x,y)\leq d(x,w)+d(w,y)\leq \frac{\min\{\vartheta^{-1},\gamma^{-1},k^{-1}\}}{2}<\rho(x),$$
where the last inequality comes from the fact that by construction $x$ is supposed to be in $\mathscr{D}_{\vartheta,\gamma}(k)$. Thanks to the fact that $r+s<d(x,y)/5A$ and $y\in B(x,\rho(x))\cap E(\vartheta,\gamma)$ we have that \cref{prop:palleseparate} implies that \eqref{eq:palle} holds.

In order to proceed with the conclusion of the proof, let us define
\begin{equation}
    \begin{split}
    \mathscr{F}_\delta(w,R):=&\{B\in\mathscr{F}_\delta(k): 5^{N+1}B\cap B(w, R)\cap\mathscr{D}_{\vartheta,\gamma}(k)\neq \emptyset\},\\
    \mathscr{G}_\delta(w,R):=&\{B\in\mathscr{G}_\delta(k): 5^{N+1}B\cap B(w, R)\cap\mathscr{D}_{\vartheta,\gamma}(k)\neq \emptyset\},
        \nonumber
    \end{split}
\end{equation}
Thanks to our choice of $R$, see \eqref{eqn:RequirementsOnR}, and the definition of $\mathscr{G}_{\delta}(w,R)$ we have
$$\frac{R^h}{2\vartheta}\leq \mathcal{S}^h\llcorner \Gamma(B(w,R))\leq 2\mathcal{S}^h\llcorner \mathscr{D}_{\vartheta,\gamma}(k)(B(w,R))\leq2\mathcal{S}^h\llcorner \mathscr{D}_{\vartheta,\gamma}(k)\bigg(\bigcup_{B\in \mathscr{G}_{\delta}(w,R)} 5^{N+1}B\bigg).$$
Let $\mathscr{G}_{\delta}(w,R)=\{B(x_i,r_i)\}_{i\in\N}$ and recall that $ x_i\in \mathscr{D}_{\vartheta,\gamma}(k)$ and that $5^{N+1}r_i\leq 1/\gamma$. This implies, thanks to \cref{cor:2.2.19}, that
\begin{equation}
\begin{split}
    \mathcal{S}^h\llcorner \mathscr{D}_{\vartheta,\gamma}(k)\bigg(\bigcup_{B\in \mathscr{G}_{\delta}(w,R)} 5^{N+1}B\bigg)&\leq 2\vartheta 5^{h(N+1)}\sum_{i\in\N}r_i^h \\
    &=  2\vartheta 5^{h(N+1)}\oldC{ProjC}(\mathbb{V},\mathbb L)^{-1}\sum_{i\in\N} \mathcal{S}^h(P_\mathbb{V}(B(x_i,r_i))) \\
    &= 2\vartheta 5^{h(N+1)}\oldC{ProjC}(\mathbb{V},\mathbb L)^{-1}\mathcal{S}^h\bigg(P_\mathbb{V}\bigg(\bigcup_{i\in\N} B(x_i,r_i)\bigg)\bigg) \\ &\leq 2\vartheta 5^{h(N+1)}\oldC{ProjC}(\mathbb{V},\mathbb L)^{-1}\mathcal{S}^h\bigg(P_\mathbb{V}\bigg(\bigcup_{B\in \mathscr{F}_\delta(w,R)} B\bigg)\bigg),
    \nonumber
    \end{split}
\end{equation}
where the first inequality comes from the subadditivity and the upper estimate that we have in the definition of $E(\vartheta,\gamma)$, see \cref{def:EThetaGamma}; while identity in the third line above comes from \eqref{eq:palle}. Summing up, for any $\delta>0$ we have
$$\frac{\oldC{ProjC}(\mathbb{V},\mathbb L) R^h}{8\cdot 5^{h(N+1)}\vartheta^2}\leq \mathcal{S}^h\bigg(P_\mathbb{V}\bigg(\bigcup_{B\in \mathscr{F}_\delta(w,R)} B\bigg)\bigg).$$
We now prove that the projection under $P_\mathbb{V}$ of the closure of $\bigcup_{B\in \mathscr{F}_\delta(w,R)} B$ converges in the Hausdorff sense to $P_\mathbb{V}(\overline{\mathscr{D}_{\vartheta,\gamma}(k)\cap B(w,R)})$ as $\delta$ goes to $0$. 
Since the set $\bigcup_{B\in \mathscr{F}_\delta(w,R)} B$ is a covering of $\mathscr{D}_{\vartheta,\gamma}(k)\cap B(w,R)$ we have that
\begin{equation}
    \mathscr{D}_{\vartheta,\gamma}(k)\cap B(w,R)\Subset \bigcup_{B\in \mathscr{F}_\delta(w,R)} B.
    \label{eq:hc1}
\end{equation}
On the other hand, since by definition the balls of $\mathscr{F}_\delta(w,R)$ have radii smaller than $\delta/4$ and centre in $\mathscr{D}_{\vartheta,\gamma}(k)$, we also have
\begin{equation}
    \bigcup_{B\in \mathscr{F}_\delta(w,R)} B\Subset B(\mathscr{D}_{\vartheta,\gamma}(k)\cap B(w,R),5^{N+2}\delta).
    \label{eq:hc2}
\end{equation}
Putting together \eqref{eq:hc1} and \eqref{eq:hc2}, we infer that the closure of $\bigcup_{B\in \mathscr{F}_\delta(w,R)} B$ converges in the Hausdorff metric to the closure of $B(w,R)\cap \mathscr{D}_{\vartheta,\gamma}(k)$. 
Furthermore, since $P_\mathbb{V}$ restricted to the ball $B(w,R+1)$ is uniformly continuous, we infer that
$$P_{\mathbb{V}}\bigg(\overline{\bigcup_{B\in \mathscr{F}_\delta(w,R)} B}\bigg)\underset{H}{\longrightarrow} P_\mathbb{V}\bigg(\overline{\mathscr{D}_{\vartheta,\gamma}(k)\cap B(w,R)}\bigg).$$
Thanks to the upper semicontinuity of the Lebesgue measure with respect to the Hausdorff convergence we eventually infer that
\begin{equation*}
\begin{split}
\frac{\oldC{ProjC}(\mathbb{V},\mathbb L) R^h}{8\cdot 5^{h(N+1)}\vartheta^2}&\leq\limsup_{\delta\to0}\mathcal{S}^h\bigg(P_\mathbb{V}\bigg(\overline{\bigcup_{B\in \mathscr{F}_\delta(w,R)} B}\bigg)\bigg) \\&\leq \mathcal{S}^h(P_\mathbb{V}(\overline{\mathscr{D}_{\vartheta,\gamma}(k)\cap B(w,R)}))\leq \mathcal{S}^h(P_\mathbb{V}(E(\vartheta,\gamma)\cap B(w,R))),
\end{split}
\end{equation*}
where the last inequality above comes from the fact that by construction $\mathscr{D}_{\vartheta,\gamma}(k)\subseteq E(\vartheta,\gamma)$ and the compactness of $E(\vartheta,\gamma)$. Finally, since $\oldC{C:projji}=2^{-5} 5^{-3h} A^{-h} \mathfrak{L}(d,d_c)^{-2h}\oldC{ProjC}(\mathbb{V},\mathbb L)$, we  infer
$$\mathcal{S}^h(P_\mathbb{V}(E(\vartheta,\gamma)\cap B(w,R)))\geq \frac{\oldC{ProjC}(\mathbb{V},\mathbb L) R^h}{8\cdot 5^{h(N+1)}\vartheta^2}\geq \frac{4\oldC{C:projji}R^h}{\vartheta^2},$$
thus showing the claim \eqref{eq:claimdens} and then the proof.
\end{proof}

\begin{proposizione}\label{prop:mutuallyabs}
Let us fix $\alpha<\oldep{ep:Cool}(\mathbb V,\mathbb L)$ and suppose $\Gamma$ is a compact $C_\mathbb{V}(\alpha)$-set of finite $\mathcal{S}^h$-measure such that
$$
\Theta^h_*(\mathcal{S}^h\llcorner \Gamma,x)>0,
$$
for $\mathcal{S}^h$-almost every $x\in \Gamma$. Let us set $\varphi:P_{\mathbb V}(\Gamma)\to \mathbb{L}$ the map whose graph is $\Gamma$, see \cref{prop:ConeAndGraph}, and set $\Phi:P_{\mathbb V}(\Gamma)\to\mathbb G$ to be the graph map of $\varphi$. Let us define $\Phi_*\mathcal S^h\llcorner\mathbb V$ to be the measure on $\Gamma$ such that for every measurable $A\subseteq \Gamma$ we have $\Phi_*\mathcal S^h\llcorner\mathbb V(A):=\mathcal{S}^h\llcorner\mathbb V(\Phi^{-1}(A))=\mathcal{S}^h\llcorner \mathbb V(P_{\mathbb V}(A))$. Then $\Phi_*\mathcal{S}^h\llcorner \mathbb{V}$ is mutually absolutely continuous with respect to $\mathcal{S}^h\llcorner \Gamma$.
\end{proposizione}

\begin{proof}

The fact that $\Phi_*\mathcal{S}^h\llcorner \mathbb{V}$ is absolutely continuous with respect to $\mathcal{S}^h\llcorner \Gamma$ is an immediate consequence of \cref{cor:2.2.19}. 
Viceversa, suppose by contradiction that there exists a compact subset $C$ of $\Gamma$ of positive $\mathcal{S}^h$-measure such that
\begin{equation}
    0=\Phi_*\mathcal{S}^h\llcorner \mathbb{V}(C)=\mathcal{S}^h(P_\mathbb{V}(C)).
    \label{eq:contradiction}
\end{equation}
Since by assumption $\Theta^h_*(\mathcal{S}^h\llcorner C,x)>0$ for $\mathcal{S}^h$-almost every $x\in C$, by \cref{prop:Lebesuge} and the fact that $C$ has positive and finite $\mathcal{S}^h$-measure, we infer thanks to Proposition \ref{prop:proj} that it must have a projection of positive $\mathcal{S}^h$-measure. This however comes in contradiction with \eqref{eq:contradiction}.
\end{proof}

In the following propositions we are going to introduce two fine coverings of $P_\mathbb{V}(\Gamma)$ and $\mathbb{V}$, respectively, that will be used in the proof of \cref{prop:dens1INTRO} to differentiate with respect to the measure $\mathcal{S}^h\llcorner P_\mathbb{V}(\Gamma)$. 

\begin{definizione}[$\phi$-Vitali relation]\label{def:covering}
Let $(X,d)$ be a metric space with a Borel measure $\phi$ on it and let $\mathcal{B}(X)$ be the family of Borel sets of $X$. We say that $S\subseteq X\times \mathcal{B}(X)$ is a {\em covering relation} if
$$S=\{(x,B):x\in B\subseteq X\}.$$
Furthermore for any $Z\subseteq X$ we let
\begin{equation}\label{SZ}
S(Z):=\{B:(x,B)\in S\text{ for some }x\in Z \}.
\end{equation}
Finally a covering $S$ is said to be {\em fine} at $x\in X$ if and only if
$$\inf\{\diam(B):(x,B)\in S\}=0.$$
By a $\phi$-Vitali relation we mean a covering relation that is fine at every point of $X$ and the following condition holds
\begin{itemize}
    \item[]If $C$ is a subset of $S$ and $Z$ is a subset of $X$ such that $C$ is fine at each point of $Z$, then $C(Z)$ has a countable disjoint subfamily covering $\phi$-almost all of $Z$.
\end{itemize}
If $\delta$ is a nonnegative function on $S(X)$, for any $B\in S(X)$ we define its {\em $\delta$-enlargement} as
\begin{equation}\label{eqn:deltaenlarg}
\hat{B}:=\bigcup\{B^\prime:B^\prime\in S(X),\text{ }B^\prime\cap B\neq \emptyset\text{ and }\delta(B^\prime)\leq 5\delta(B)\}.
\end{equation}
\end{definizione}
In the remaining part of this section we use the following general result due to Federer: it contains a cryterion to show that a fine covering relation is a $\phi$-Vitali relation, and a Lebesgue theorem for $\phi$-Vitali relations.
\begin{proposizione}[{\cite[Theorem 2.8.17, Corollary 2.9.9 and Theorem 2.9.11]{Federer1996GeometricTheory}}]\label{prop:FedererVitali}
Let $X$ be a metric space, and let $\phi$ be a Borel regular measure on $X$ that is finite on bounded sets.  Let $S$ be a covering relation such that $S(X)$ is a family of bounded closed sets, $S$ is fine at each point of $X$, and let $\delta$ be a nonnegative function on $S(X)$ such that 
$$
\lim_{\varepsilon\to 0^+}\sup\left\{\delta(B)+\frac{\phi(\hat B)}{\phi (B)}: (x,B)\in S,\,\,\diam B<\varepsilon \right\}<+\infty,
$$
for $\phi$-almost every $x\in X$. Then $S$ is a $\phi$-Vitali relation. 

Moreover, if $S$ is a $\phi$-Vitali relation on $X$, and $f$ is a $\phi$-measurable real-valued function with $\int_K |f|\mathrm{d}\phi<+\infty$ on every bounded $\phi$-measurable $K$, we have 
$$
\lim_{\varepsilon\to 0^+}\sup\left\{\frac{\int_B|f(z)-f(x)|\mathrm{d}\phi(z)}{\phi(B)}: (x,B)\in S,\,\,\diam B<\varepsilon \right\}=0,
$$
for $\phi$-almost every $x\in X$.
In addition, given $A\subseteq X$, if we define
$$
P:=\left\{x\in X: \lim_{\varepsilon\to 0^+}\inf\left\{\frac{\phi(B\cap A)}{\phi(B)}: (x,B)\in S,\,\,\diam B<\varepsilon \right\}=1\right\},
$$
then $P$ is $\phi$-measurable and $\phi(A\setminus P)=0$.

\end{proposizione}

\begin{proposizione}\label{prop:vitalirel}
Let $\alpha<\oldep{ep:Cool}(\mathbb V,\mathbb L)$ and suppose that $\Gamma$ is a compact $C_\mathbb{V}(\alpha)$-set of finite $\mathcal{S}^h$-measure such that
$$\Theta^h_*(\mathcal{S}^h\llcorner \Gamma,x)>0,$$
for $\mathcal{S}^h$-almost every $x\in \Gamma$. As in the statement of \cref{prop:mutuallyabs}, let us denote with $\Phi:P_{\mathbb V}(\Gamma)\to\mathbb G$ the graph map of $\varphi:P_{\mathbb V}(\Gamma)\to \mathbb L$ whose intrinsic graph is $\Gamma$. Then the covering relation
$$
S_1:=\Big\{\big(z,P_\mathbb{V}(B(\Phi(z),r)\cap \Gamma)\big):z\in P_\mathbb{V}(\Gamma)\text{ and }0<r<\min\{1,R(\Phi(z))\}\Big\},
$$
is a $\mathcal{S}^h\llcorner P_\mathbb{V}(\Gamma)$-Vitali relation, where $R(\Phi(z))$ is defined as in \cref{prop:proj} for $\mathcal{S}^h\llcorner P_\mathbb{V}(\Gamma)$-almost every $z\in \mathbb V$
and it is $+\infty$ on the remaining null set where \cref{prop:proj} eventually does not hold.
\end{proposizione}

\begin{proof}
First of all, it is readily noticed that $S_1$ is a fine covering of $P_\mathbb{V}(\Gamma)$ sine $P_{\mathbb V}$ is continuous. Let us prove that $S_1$ is a $\mathcal{S}^h\llcorner P_\mathbb{V}(\Gamma)$-Vitali relation in $(P_\mathbb{V}(\Gamma),d)$ with the distance $d$ induced form $\mathbb G$. For $x\in\ P_{\mathbb V}(\Gamma)$ and $r>0$, define $G(x,r):=P_\mathbb{V}(B(\Phi(x),r)\cap \Gamma)$. Notice that an arbitrary element of $S_1(P_{\mathbb V}(\Gamma))$, see \eqref{SZ}, is of the form $G(x,r)$ for some $x\in P_{\mathbb V}(\Gamma)$ and some $0<r<\min\{1,R(\Phi(x))\}$. Let $\delta\big(G(x,r)\big):=r$ and note that the $\delta$-enlargement, see \eqref{eqn:deltaenlarg}, of $G(x,r)$ is
\begin{equation}
    \begin{split}
        \hat{G}(x,r)&:=\bigcup\{G(y,s):y\in P_{\mathbb V}(\Gamma),~ 0<s<\min\{1,R(\Phi(y))\},~  G(y,s)\cap G(x,r)\neq \emptyset\text{ and }\delta(G(y,s))\leq 5\delta(G(x,r))\}\\
        &=\bigcup\{G(y,s):y\in P_{\mathbb V}(\Gamma),~ 0<s<\min\{1,R(\Phi(y))\},~ G(y,s)\cap G(x,r)\neq \emptyset\text{ and }s\leq 5r\}.
    \end{split}
\end{equation}
Whenever $G(x,r)\cap G(y,s)\neq \emptyset$ we have that $d(\Phi(x),\Phi(y))\leq r+s$: indeed, since $P_\mathbb{V}$ is injective on $\Gamma$, see \cref{prop:ConeAndGraph}, we have $P_\mathbb{V}(B(\Phi(x),r)\cap \Gamma)\cap P_\mathbb{V}(B(\Phi(y),s)\cap\Gamma)\neq \emptyset$ if and only if $B(\Phi(x),r)\cap B(\Phi(y),s)\cap \Gamma\neq \emptyset$. In particular, since $s\leq 5r$ we have $B(\Phi(y),s)\Subset B(\Phi(x),12r)$, and thus $\hat{G}(x,r)\subseteq G(x,12r)$ for every $x\in P_{\mathbb V}(\Gamma)$ and $0<r<\min\{1,R(\Phi(x))\}$.

Finally, thanks to \cref{prop:proj} and \cref{prop:mutuallyabs}, for $\mathcal{S}^h$-almost every $x\in P_\mathbb{V}(\Gamma)$ we have
\begin{equation}
    \begin{split}
       \lim_{\xi\to 0}&\sup\left\{\delta(G(x,r))+\frac{\mathcal{S}^h(\hat{G}(x,r))}{\mathcal{S}^h(G(x,r))}: 0<r<\min\{1,R(\Phi(x))\},\, \diam(G(x,r))\leq \xi\right\} \\
       &\leq 1+  \lim_{\xi\to 0}\sup\frac{\mathcal{S}^h(G(x,12r))}{\mathcal{S}^h(G(x,r))}\leq 1+ \lim_{\xi\to 0}\sup \frac{\mathcal{S}^h(P_{\mathbb{V}}(B(\Phi(x),12r)))}{\mathcal{S}^h(P_\mathbb{V}(B(\Phi(x),r)\cap \Gamma))}\\
        &\leq  1+\frac{(12r)^h\mathcal{S}^h(P_{\mathbb{V}}(B(0,1)))}{\oldC{C:projji}\Theta^h_*(\mathcal{S}^h\llcorner \Gamma,\Phi(x))^2r^h} =1+\frac{12^{h}\mathcal{S}^h(P_{\mathbb{V}}(B(0,1)))}{\oldC{C:projji}\Theta^h_*(\mathcal{S}^h\llcorner \Gamma,\Phi(x))^2},
        \label{eq:limitVit}
    \end{split}
\end{equation}
where we explicitly mentioned the set over which we take the supremum only in the first line for the ease of notation, and where the first inequality in the third line follows from the fact that $\mathcal{S}^h(P_\mathbb{V}(\mathcal{E}))=\mathcal{S}^h(P_\mathbb{V}(x\mathcal{E}))$ for any $x\in \mathbb{G}$ and any Borel set $\mathcal{E}\subseteq \mathbb{G}$, see \cref{prop:InvarianceOfProj}.
Thanks to \eqref{eq:limitVit} we can apply the first part of \cref{prop:FedererVitali} and thus we infer that $S_1$ is a $\mathcal{S}^h\llcorner P_\mathbb{V}(\Gamma)$-Vitali relation.
\end{proof}

\begin{proposizione}
\label{prop:vitali2PD}
Let $\alpha<\oldep{ep:Cool}(\mathbb V,\mathbb L)$ and let $\Gamma$ be a compact $C_\mathbb{V}(\alpha)$-set of finite $\mathcal{S}^h$-measure.
As in the statement of \cref{prop:mutuallyabs}, let us denote with $\Phi:P_{\mathbb V}(\Gamma)\to\mathbb G$ the graph map of $\varphi:P_{\mathbb V}(\Gamma)\to \mathbb L$ whose intrinsic graph is $\Gamma$. Then for $\mathcal{S}^h$-almost every $w\in P_\mathbb{V}(\Gamma)$ we have
\begin{equation}
    \lim_{r\to 0}\frac{\mathcal{S}^h\big(P_\mathbb{V}\big(B(\Phi(w),r)\cap \Phi(w)C_\mathbb{V}(\alpha)\big)\cap P_\mathbb{V}(\Gamma)\big)}{\mathcal{S}^h\big(P_\mathbb{V}\big(B(\Phi(w),r)\cap \Phi(w)C_\mathbb{V}(\alpha)\big)\big)}=1.
    \label{eq:limi}
\end{equation}
\end{proposizione}

\begin{proof}
For any $w\in\mathbb{V}\setminus P_\mathbb{V}(\Gamma)$ we let
$$\rho(w):=\inf\{r\geq 0:B(w,r)\cap P_\mathbb{V}(B(\Gamma,r^{1/\kappa}))\neq \emptyset\}.$$
It is immediate to see that $\rho(w)\leq \dist(w,P_{\mathbb V}(\Gamma))$ and that $\rho(w)=0$ if and only if $w\in P_\mathbb{V}(\Gamma)$. Throughout the rest of the proof we let $S$ be the fine covering of $\mathbb{V}$ given by the couples $(w,G(w,r))$ for which
\begin{itemize}
    \item[($\alpha$)] if $w\in \mathbb{V}\setminus P_\mathbb{V}(\Gamma)$ then $r\in(0,\min\{\rho(w)/2,1\})$ and     $G(w,r):=B(w,r)\cap \mathbb{V}$,
    \item[($\beta$)] if $w\in P_\mathbb{V}(\Gamma)$ then $r\in(0,1)$ and     $G(w,r):=P_\mathbb{V}(B(\Phi(w),r)\cap \Phi(w)C_\mathbb{V}(\alpha))$.
\end{itemize}
Furthermore, for any $w\in\mathbb{V}$ we define the function $\delta$ on $S(\mathbb V)$, see \eqref{SZ}, as
\begin{equation}
    \delta\big(G(w,r)\big):=r.
\end{equation}
If we prove that $S$ is a $\mathcal{S}^h\llcorner\mathbb{V}$-Vitali relation, the second part of \cref{prop:FedererVitali} directly implies that \eqref{eq:limi} holds. If for $\mathcal{S}^h$-almost every $w\in \mathbb{V}$  we prove that
\begin{equation}
\begin{split}
\lim_{\xi\to 0}\underset{(w,G(w,r))\in S,\,\diam(G(w,r))\leq\xi}{\sup}\left\{\delta\big(G(w,r)\big)+\frac{\mathcal{S}^h(\hat{G}(w,r))}{\mathcal{S}^h(G(w,r))}\right\} \leq 1+\lim_{\xi\to 0}\sup\frac{\mathcal{S}^h(\hat{G}(w,r))}{\mathcal{S}^h(G(w,r))}<\infty,    
\label{eq:vitalilimit}
\end{split}
\end{equation}
where we explicitly mentioned the set over which we take the supremum only the first time for the ease of notation, and where $\hat{G}(w,r)$ is the $\delta$-enlargement of $G(w,r)$, see \eqref{eqn:deltaenlarg}; thus,
thanks to the first part of \cref{prop:FedererVitali} we would immediately infer that $S$ is a $\mathcal{S}^h\llcorner \mathbb{V}$-Vitali relation. In order to prove that \eqref{eq:vitalilimit} holds, we need to get a better understanding of the geometric structure of the $\delta$-enlargement of $G(w,r)$. 

If $w\in \mathbb{V}\setminus P_\mathbb{V}(\Gamma)$, we note that there must exist an $0<r(w)<\min\{\rho(w)/2,1\}$ such that for any $0<r<r(w)$ we have
$$B(w,r)\cap P_\mathbb{V}(B(\Gamma,5r))= \emptyset.$$
Indeed, if this is not the case there would exist a sequence $r_i\downarrow 0$ and a sequence $\{z_i\}_{i\in\N}$ such that
$$z_i\in B(w,r_i)\cap P_\mathbb{V}(B(\Gamma,5r_i)).$$
Since $P_\mathbb{V}(\Gamma)$ is compact and $P_\mathbb{V}$ is continuous on the closed tubular neighborhood $B(\Gamma,1)$, up to passing to a non re-labelled subsequence we have that the $z_i$'s converge to some $z\in P_\mathbb{V}(\Gamma)$ and on the other hand by construction the $z_i$'s converge to $w$ which is not contained in $P_\mathbb{V}(\Gamma)$, and this is a contradiction. This implies that if $0<r<r(w)$, we have
\begin{equation}
    \begin{split}
\hat{G}(w,r)&= \bigcup \{G(y,s):\, y\in\mathbb V,\, s>0,\,  {}(y,G(y,s))\in S,\text{ } G(y,s)\cap G(w,r)\neq \emptyset,\text{ and }s\leq 5r\}\\
&\subseteq \bigcup\{B(y,s)\cap\mathbb V: B(y,s)\cap B(w,r)\cap\mathbb V\neq \emptyset\text{ and }s\leq 5r\}\subseteq B(w,11r )\cap\mathbb V,
\label{eq:vit1}
    \end{split}
\end{equation}
where in the inclusion we are using the fact that if $y$ were in $\mathbb P_{\mathbb V}(\Gamma)$, and $s\leq 5r$, then $G(y,s)\subseteq P_{\mathbb V}(B(\Gamma,s))\subseteq P_{\mathbb V}(B(\Gamma,5r))$ which would be in contradiction with $G(y,s)\cap G(w,r)\neq \emptyset$, since we chose $0<r<r(w)$. Summing up, if $w\in \mathbb{V}\setminus P_\mathbb{V}(\Gamma)$ the bound \eqref{eq:vitalilimit} immediately follows thanks to \eqref{eq:vit1} and the homogeneity of $\mathcal{S}^h$.

If on the other hand $w\in P_\mathbb{V}(\Gamma)$ the situation is more complicated. If $y\in\mathbb{V}\setminus P_\mathbb{V}(\Gamma)$ and $s\leq 5r$ are such that
\begin{equation}
    G(y,s)\cap P_\mathbb{V}(B(\Phi(w),r)) = B(y,s)\cap P_\mathbb{V}(B(\Phi(w),r))\neq \emptyset,
\label{eq:intervit}
\end{equation}
since by construction of the covering $S$ we also assumed that $0<s<\rho(y)/2$, we infer that we must have $r\geq s^{1/\kappa}$ for \eqref{eq:intervit} to be satisfied. This allows us to infer that, for every $w\in P_{\mathbb V}(\Gamma)$ and $0<r<1$, we have
\begin{equation}
\begin{split}
    \hat{G}(w,r)&= \bigcup \{G(y,s): y\in\mathbb V,\, s>0,\, (y,G(y,s))\in S,\text{ }  G(y,s)\cap G(w,r)\neq \emptyset,\text{ and }s\leq 5r\} \\
    &\subseteq \bigcup \{P_\mathbb{V}(B(\Phi(y),s)): y\in P_\mathbb{V}(\Gamma),\text{ }  P_\mathbb{V}(B(\Phi(y),s))\cap P_\mathbb{V}(B(\Phi(w),r))\neq \emptyset,\text{ and }s\leq 5r\}
    \text{}\cup \\ &\text{}\cup\bigcup \{B(y,s)\cap \mathbb{V}: \text{ } y\in \mathbb{V}\setminus P_\mathbb{V}(\Gamma),\text{ } B(y,s)\cap P_\mathbb{V}(B(\Phi(w),r))\neq \emptyset, \text{ and }s\leq \min\{5r,\rho(y)/2\}\} \\
    &\subseteq \bigcup \{P_\mathbb{V}(B(\Phi(y),s)): y\in P_\mathbb{V}(\Gamma),\text{ }  P_\mathbb{V}(B(\Phi(y),s))\cap P_\mathbb{V}(B(\Phi(w),r))\neq \emptyset,\text{ and }s\leq 5r\}\cup \\&\text{} \cup\left(B(P_\mathbb{V}(B(\Phi(w),r)),3r^\kappa)\cap \mathbb V\right),
    \label{eqn:MonsterInc}
\end{split}
\end{equation}
where in the last inclusion we are using the observation right after \eqref{eq:intervit} according to which $s\leq r^\kappa$. 
We now study independently each of the two terms of the union of the last two lines above. Let us first note that if $w,y\in P_\mathbb{V}(\Gamma)$, $s\leq 5r$ and
$$P_\mathbb{V}(B(\Phi(y),s))\cap P_\mathbb{V}(B(\Phi(w),r))\neq \emptyset,$$
then $P_\mathbb{V}(B(\Phi(y),10r))\cap P_\mathbb{V}(B(\Phi(w),2r))\neq \emptyset$. This observation and \cref{cor:pervitali} imply that if $0<r<r(w)$ is sufficiently small we have $d(\Phi(w),\Phi(y))\leq 50Ar$, where the constant $A=A(\mathbb V,\mathbb L)$ is yielded by \cref{lemma:tec.cono}. In particular we deduce that for every $0<r<r(w)$ sufficiently small
\begin{equation*}
\begin{split}
\bigcup \{P_\mathbb{V}(B(\Phi(y),s)): y\in P_\mathbb{V}(\Gamma),\text{ } P_\mathbb{V}(B(\Phi(y),s))\cap P_\mathbb{V}(B(\Phi(w),r))\neq \emptyset,\text{ and }s\leq 5r\}\subseteq P_\mathbb{V}(B(\Phi(w),50(A+1)r)).
\end{split}
\end{equation*}
In order to study the term in the last line of \eqref{eqn:MonsterInc}, we prove the following claim: for every $0<r<1$, every $z\in P_\mathbb{V}(B(\Phi(w),r))$, and every $\Delta\in B(0,3r^\kappa)\cap \mathbb V$ we have $z\Delta\in P_\mathbb{V}(B(\Phi(w), C(\Gamma)r))$, where $C(\Gamma)$ is a constant depending only on $\Gamma$. Indeed, since $\Gamma$ is compact and $P_{\mathbb L}$ is continuous, there exists a constant $K':=K'(\Gamma)$ such that whenever $0<r<1$, and $z\in P_\mathbb{V}(B(\Phi(w),r))$, there exsits an $\ell\in\mathbb{L}$ such that $z\ell\in B(\Phi(w),r)$ and $\|\ell\|\leq K'$. Thus there exists a constant $K:=K(\Gamma)>0$ such that whenever $0<r<1$, $z\in P_\mathbb{V}(B(\Phi(w),r))$, and $\Delta\in B(0,3r^\kappa)\cap \mathbb V$, there exists $\ell\in\mathbb L$ with $z\ell\in B(\Phi(w),r)$ and $\|\Delta\|+\|\ell\|\leq K$. Thus we can estimate
$$d(\Phi(w),z\Delta \ell)\leq d(\Phi(w),z\ell)+d(z\ell,z\Delta \ell)\leq r+\oldC{c:1}(K)\|\Delta\|^{1/\kappa}\leq C(\Gamma)r,
$$
where the second inequality in the last equation comes from \cref{lem:EstimateOnConjugate}. Thus $z\Delta\in P_{\mathbb V}(B(\Phi(w),C(\Gamma)r))$, and the claim is proved.  Summing up, we have proved that whenever $w\in P_{\mathbb V}(\Gamma)$ and $0<r<r(w)$ is sufficiently small we have
$$\hat{G}(w,r)\subseteq P_\mathbb{V}(B(\Phi(w),50(A+1)r)) \cup P_\mathbb{V}(B(\Phi(w), C(\Gamma)r)),$$
and thus \eqref{eq:vitalilimit} immediately follows by the homogeneity of $\mathcal{S}^h\llcorner \mathbb V$ and the fact that $\mathcal{S}^h(P_{\mathbb V}(x\mathcal{E}))=\mathcal{S}^h(P_{\mathbb V}(\mathcal{E}))$ for every $x\in\mathbb G$ and $\mathcal{E}$ a Borel subset of $\mathbb G$, see \cref{prop:InvarianceOfProj}. This concludes the proof of the proposition.
\end{proof}

We prove below a more precise version of \cref{prop:dens1INTRO}.
\begin{proposizione}\label{prop:dens1}
Let us fix $\alpha< \oldep{ep:Cool}(\mathbb V,\mathbb L)$. Suppose $\Gamma$ is a compact $C_\mathbb{V}(\alpha)$-set such that $\mathcal{S}^h\llcorner \Gamma$ is $\mathscr{P}^c_h$-rectifiable. For $\mathcal{S}^h$-almost every $x\in \Gamma$ we have
\begin{equation}
    (1-\mathfrak{c}(\alpha))^{2h}(1+\mathfrak{c}(\alpha))^{-h}\leq \frac{\Theta^h_*(\mathcal{S}^h\llcorner \Gamma,x)}{\Theta^{h,*}(\mathcal{S}^h\llcorner \Gamma,x)}\leq 1,
    \label{eq:dens1}
\end{equation}
where $\mathfrak c(\alpha)$ is defined in \cref{lemma:projections}.
\end{proposizione}

\begin{proof}
Let us preliminarly observe that since $\mathcal{S}^h\llcorner\mathbb V$ and $\mathcal{C}^h\llcorner\mathbb V$ are both Haar measures on $\mathbb V$, they coincide up to a constant. Since for $\mathcal{S}^h$-almost every $x\in \Gamma$ we have $\Theta^{h,*}(\mathcal{S}^h\llcorner \Gamma,x)>0$, the upper bound is trivial.
Let us proceed with the lower bound. Thanks to \cref{prop:mutuallyabs} and the Radon-Nikodym Theorem, see \cite[page 82]{HeinonenKoskelaShanmugalingam}, there exists $\rho\in L^1(\Phi_*\mathcal{C}^h\llcorner \mathbb{V})$ such that
\begin{itemize}
    \item[(i)]$\rho(x)>0$ for  $\Phi_*\mathcal{C}^h\llcorner \mathbb{V}$-almost every $x\in\Gamma$,
    \item[(ii)] $\mathcal{S}^h\llcorner \Gamma=\rho\Phi_*\mathcal{C}^h\llcorner \mathbb{V}$.
\end{itemize}
We stress that the following reasoning holds for $\mathcal{S}^h\llcorner\Gamma$-almost every $x\in\Gamma$. Let $\{r_i\}_{i\in\N}$ be an infinitesimal sequence such that $r_i^{-h}T_{x,r_i}\mathcal{S}^h\llcorner \Gamma\rightharpoonup \lambda \mathcal{C}^h\llcorner \mathbb{V}(x)$ for some $\lambda>0$. First of all, we immediately see that \cref{cordenstang} implies that $\lambda\in [\Theta_*^h(\mathcal{S}^h\llcorner \Gamma,x),\Theta^{h,*}(\mathcal{S}^h\llcorner \Gamma,x)]$ and that
\begin{equation}
\begin{split}
    1&=\lim_{i\to\infty}\frac{\mathcal{S}^h\llcorner \Gamma(B(x,r_i))}{\mathcal{S}^h\llcorner \Gamma(B(x,r_i))}=\lim_{i\to\infty}\frac{\int_{P_\mathbb{V}(B(x,r_i)\cap \Gamma)}\rho(\Phi(y))d\mathcal{C}^h\llcorner \mathbb{V}(y)}{\mathcal{S}^h\llcorner \Gamma(B(x,r_i))}\\ &=\frac{\rho(x)}{\lambda}\lim_{i\to\infty}\frac{\mathcal{C}^h\llcorner \mathbb{V}(P_\mathbb{V}(B(x,r_i)\cap \Gamma))}{r_i^h},
    \nonumber
    \end{split}
\end{equation}
where the last identity comes from \cref{prop:vitalirel}, that allows us to differentiate by using the second part of \cref{prop:FedererVitali}, and \cref{prop:density}.
Thanks to \cref{lemma:projections}, \cref{rem:Ch1}, and the fact that $\Gamma$ is a $C_\mathbb{V}(\alpha)$-set, we have
\begin{equation}
\begin{split}
    \frac{\lambda}{\rho(x)}&\leq\lim_{i\to\infty}\frac{\mathcal{C}^h\llcorner \mathbb{V}(P_\mathbb{V}(B(x,r_i)\cap xC_\mathbb{V}(\alpha)))}{r_i^h}= \mathcal{C}^h(P_\mathbb{V}(B(0,1)\cap C_\mathbb{V}(\alpha))) \\&\leq \frac{\mathcal{C}^h\llcorner \mathbb{V}(B(0,1))}{(1-\mathfrak{c}(\alpha))^{h}}=(1-\mathfrak{c}(\alpha))^{-h},
    \label{eq:bdabove}
    \end{split}
\end{equation}
where in the second equality we are using the homogeneity of $\mathcal{C}^h$ and the fact that $\mathcal{C}^h(P_{\mathbb V}(x\mathcal{E}))=\mathcal{C}^h(P_{\mathbb V}(\mathcal{E}))$ for every $x\in\mathbb G$ and $\mathcal{E}$ a Borel subset of $\mathbb G$, see \cref{prop:InvarianceOfProj}.
On the other hand, thanks to \cref{lemma:lowerbd} we have
\begin{equation}
\begin{split}
    \frac{\lambda}{\rho(x)}&=\lim_{i\to\infty}\frac{\mathcal{C}^h\llcorner \mathbb{V}(P_\mathbb{V}(B(x,r_i)\cap \Gamma))}{r_i^h}\\ &\geq\lim_{i\to\infty}\frac{\mathcal{C}^h\big(P_\mathbb{V}\big(B(x,\mathfrak{C}(\alpha)r_i)\cap xC_\mathbb{V}(\alpha)\big)\cap P_\mathbb{V}(\Gamma)\big)}{\mathcal{C}^h\big(P_\mathbb{V}\big(B(x,\mathfrak{C}(\alpha)r_i)\cap xC_\mathbb{V}(\alpha)\big)\big)}\frac{\mathcal{C}^h\big(P_\mathbb{V}\big(B(x,\mathfrak{C}(\alpha)r_i)\cap xC_\mathbb{V}(\alpha)\big)\big)}{r_i^h}\\
     &=\mathfrak{C}(\alpha)^h\mathcal{C}^h\big(P_\mathbb{V}(B(0,1)\cap C_\mathbb{V}(\alpha))\big)\geq\mathfrak{C}(\alpha)^h,
\end{split}
    \label{eq:bdbelow}
\end{equation}
where the first identity in the last line comes from \cref{prop:vitali2PD} and the last inequality from \cref{lemma:projections}, \cref{rem:Ch1}, and $\mathfrak C(\alpha)$ is defined in \eqref{eqn:Calpha}. Putting together \eqref{eq:bdabove} and \eqref{eq:bdbelow}, we have
\begin{equation}
    \frac{(1-\mathfrak{c}(\alpha))^h}{(1+\mathfrak{c}(\alpha))^h}\leq \frac{\lambda}{\rho(x)}\leq\frac{1}{ (1-\mathfrak{c}(\alpha))^{h}}.
    \label{eq:bdbilateral}
\end{equation}
Thanks  to the definition of $\Theta^{h}_*(\mathcal{S}^h\llcorner \Gamma,x)$ and $\Theta^{h,*}(\mathcal{S}^h\llcorner \Gamma,x)$ we can find two sequences $\{r_i\}_{i\in\N}$ and $\{s_i\}_{i\in\N}$ such that
$$\Theta^{h}_*(\mathcal{S}^h\llcorner \Gamma,x)=\lim_{i\to\infty}\frac{\mathcal{S}^h\llcorner\Gamma(B(x,r_i))}{r_i^h},\qquad\text{and}\qquad\Theta^{h,*}(\mathcal{S}^h\llcorner \Gamma,x)=\lim_{i\to\infty}\frac{\mathcal{S}^h\llcorner\Gamma(B(x,s_i))}{s_i^h},$$
and without loss of generality, taking \cref{prop:density} into account, we can assume that $$r_i^{-h}T_{x,r_i}\mathcal{S}^h\llcorner \Gamma\rightharpoonup \Theta^{h}_*(\mathcal{S}^h\llcorner \Gamma,x) \mathcal{C}^h\llcorner \mathbb V(x),\quad , s_i^{-h}T_{x,s_i}\mathcal{S}^h\llcorner \Gamma\rightharpoonup \Theta^{h,*}(\mathcal{S}^h\llcorner \Gamma,x)\mathcal{C}^h\llcorner \mathbb V(x).$$
The bounds \eqref{eq:bdbilateral} imply therefore that
\begin{equation}
\begin{split}
    \frac{(1-\mathfrak{c}(\alpha))^h}{(1+\mathfrak{c}(\alpha))^h}\leq \frac{\Theta^h_*(\mathcal{S}^h\llcorner\Gamma,x)}{\rho(x)}\leq \frac{1}{(1-\mathfrak{c}(\alpha))^h},\\
\frac{(1-\mathfrak{c}(\alpha))^h}{(1+\mathfrak{c}(\alpha))^h}\leq \frac{\Theta^{h,*}(\mathcal{S}^h\llcorner\Gamma,x)}{\rho(x)}\leq \frac{1}{(1-\mathfrak{c}(\alpha))^h}.
\end{split}
    \label{eq:bdfinale}
\end{equation}
Finally the bounds in \eqref{eq:bdfinale} yield
$$(1-\mathfrak{c}(\alpha))^{2h}(1+\mathfrak{c}(\alpha))^{-h}\leq \frac{\Theta^h_*(\mathcal{S}^h\llcorner\Gamma,x)}{\Theta^{h,*}(\mathcal{S}^h\llcorner\Gamma,x)}\leq 1,$$
and this concludes the proof.
\end{proof}

We prove now the existence of density of $\mathscr{P}_h^c$-rectifiable measures, see \cref{coroll:DENSITYIntro}. We first prove an algebraic lemma, then we prove the existence of the density for measures of the type $\mathcal{S}^h\llcorner \Gamma$, and then we conclude with the proof of the existence of the density for arbitrary $\mathscr{P}_h^c$-rectifiable measures.
\begin{lemma}\label{lem:LemmaContoso}
Let us fix $0<\varepsilon<1$ a real number, $\ell,h\in \mathbb N$, and let $f$ be the function defined as follows
$$
f\colon \{(\alpha,C)\in(0,+\infty)^2:\alpha<C\}\to (0,+\infty), \qquad f(\alpha,C):=\frac{\alpha}{C-\alpha}.
$$
Then, there exists $\widetilde\alpha:=\widetilde\alpha(\varepsilon,\ell,h)>0$ such that the following implication holds
$$
\text{if $0<\alpha\leq \widetilde\alpha$ and $C>1/\ell$, then $\alpha<C$ and  $(1-f(\alpha,C))^{2h}(1+f(\alpha,C))^{-h}\geq 1-\varepsilon$}.
$$
\end{lemma}
\begin{proof}
Let us choose $0<\widetilde\varepsilon:=\widetilde\varepsilon(\varepsilon,h)<1$ such that
$$
(1-x)^{2h}(1+x)^{-h} \geq 1-\varepsilon, \qquad \text{for all $0\leq x\leq \widetilde\varepsilon$}.
$$
Let us show that the sought constant $\widetilde\alpha(\varepsilon,\ell,h)$ is $\widetilde\alpha:=\widetilde\varepsilon/(\ell(1+\widetilde\varepsilon))$. Indeed, if $\alpha\leq\widetilde\alpha$ and $C>1/\ell$ we infer that $\alpha<C$ and
$$
\alpha\leq \frac{\widetilde\varepsilon}{\ell(1+\widetilde\varepsilon)} \leq \frac{C\widetilde\varepsilon}{1+\widetilde\varepsilon}, \quad \text{and then}\quad f(\alpha,C)=\frac{\alpha}{C-\alpha}\leq \widetilde\varepsilon.
$$
This implies that if $\alpha\leq \widetilde\alpha$ and $C>1/\ell$, then
$$
(1-f(\alpha,C))^{2h}(1+f(\alpha,C))^{-h} \geq 1-\varepsilon,
$$
where the last inequality above comes from the choice of $\widetilde\varepsilon$. This concludes the proof.
\end{proof}
\begin{teorema}\label{thm:ExistenceOfDensitySets} 
Let $\Gamma$ be a compact subset of $\mathbb G$ such that 
$\mathcal{S}^h\llcorner\Gamma$ is a $\mathscr{P}^c_h$-rectifiable measure. Then 
$$
0<\Theta^h_*(\mathcal{S}^h\llcorner\Gamma,x)=\Theta^{*,h}(\mathcal{S}^h\llcorner\Gamma,x)<+\infty, \quad \text{for $\mathcal{S}^h\llcorner\Gamma$-almost every $x\in\mathbb G$.}
$$
\end{teorema}
\begin{proof}
In the following, for any $\varepsilon>0$, we will construct a measurable set $A_{\varepsilon}\subseteq\Gamma$ such that $\mathcal{S}^h(\Gamma\setminus A_\varepsilon)=0$ and 
\begin{equation}\label{eqn:WhatWeNeed}
1-\varepsilon \leq \frac{\Theta^{*,h}(\mathcal{S}^h\llcorner\Gamma,x)}{\Theta^h_*(\mathcal{S}^h\llcorner\Gamma,x)}\leq 1, \quad \text{for every  $x\in A_{\varepsilon}$}.
\end{equation}
If \eqref{eqn:WhatWeNeed} holds then we are free to choose $\varepsilon=1/n$ for every $n\in\mathbb N$ and then the density of $\mathcal{S}^h\llcorner\Gamma$ exists on the set $\cap_{n=1}^{+\infty} A_{1/n}$, that has full $\mathcal{S}^h\llcorner\Gamma$-measure. So we are left to construct $A_\varepsilon$ as in \eqref{eqn:WhatWeNeed}.
Let us define the function 
$$
\bold F(\mathbb V,\mathbb L):=\oldep{ep:Cool}(\mathbb V,\mathbb L), \quad \text{for all $\mathbb V\in\G_c(h)$ with complement $\mathbb L$}.
$$
Let us take the family $\mathscr{F}:=\{\mathbb V_i\}_{i=1}^{+\infty} \subseteq \G_c(h)$ and let us choose $\mathbb L_i$ complementary subgroups to $\mathbb V_i$ as in the statement of \cref{thm:MainTheorem1}. We remark that the choices of the family $\mathscr{F}$ and of the complementary subgroups depend on the function $\bold F$ previously defined, see the discussion before \cref{thm:MainTheorem1}. 
Let us define 
$$
\beta\colon\mathbb N\to (0,1), \qquad \beta(\ell):=\widetilde \alpha(\varepsilon,\ell,h),
$$
where $\widetilde \alpha(\varepsilon,\ell,h)$ is the constant in \cref{lem:LemmaContoso}, and with an abuse of notation let us lift $\beta$ to a function on $\mathscr{F}$ as we did in the statement of \cref{thm:MainTheorem1}. From \cref{thm:MainTheorem1} we conclude that there exist countably many $\Gamma_i$'s that are compact $C_{\mathbb V_i}(\min\{\oldep{ep:Cool}(\mathbb V_i,\mathbb L_i),\beta(\mathbb V_i)\})$-sets contained in $\Gamma$ such that
\begin{equation}\label{eqn:UsefulToConclude}
\mathcal{S}^h\left(\Gamma\setminus\cup_{i=1}^{+\infty} \Gamma_i\right)=0. 
\end{equation}
Let us write, for the ease of notation, $\alpha_i:= \min\{\oldep{ep:Cool}(\mathbb V_i,\mathbb L_i),\beta(\mathbb V_i)\}$ for every $i\in\mathbb N$. Since $\Gamma_i\subseteq \Gamma$ and $\mathcal{S}^h\llcorner\Gamma$ is $\mathscr{P}_h^c$-rectifiable, we conclude, by exploiting the locality of tangents, see \cref{prop:LocalityOfTangent}, and the Lebesgue differentiation theorem in \cref{prop:Lebesuge}, that the measures $\mathcal{S}^h\llcorner\Gamma_i$ are $\mathscr{P}_h^c$-rectifiable as well for every $i\in\mathbb N$. Thus, since $\alpha_i\leq \oldep{ep:Cool}(\mathbb V_i,\mathbb L_i)$, we can apply \cref{prop:dens1} and conclude that, for every $i\in\mathbb N$, we have
$$
(1-\mathfrak{c}(\alpha_i))^{2h}(1+\mathfrak{c}(\alpha_i))^{-h} \leq \frac{\Theta^{*,h}(\mathcal{S}^h\llcorner\Gamma_i,x)}{\Theta^h_*(\mathcal{S}^h\llcorner\Gamma_i,x)}\leq 1, \quad \text{for $\mathcal{S}^h\llcorner\Gamma_i$-almost every $x\in \mathbb G$},
$$
where $\mathfrak{c}(\alpha_i):=\alpha_i/(\oldC{C:split}(\mathbb V_i,\mathbb L_i)-\alpha_i)$. 
Since $\Theta^{*,h}(\mathcal{S}^h\llcorner\Gamma_i,x)=\Theta^{*,h}(\mathcal{S}^h\llcorner\Gamma,x)$ and $\Theta^h_*(\mathcal{S}^h\llcorner\Gamma_i,x)=\Theta^h_*(\mathcal{S}^h\llcorner\Gamma,x)$ for $\mathcal{S}^h\llcorner\Gamma_i$-almost every $x\in \mathbb G$, see \cref{prop:Lebesuge}, for every $i\in\mathbb N$ we conclude that
\begin{equation}\label{eqn:USEFUL}
(1-\mathfrak{c}(\alpha_i))^{2h}(1+\mathfrak{c}(\alpha_i))^{-h} \leq \frac{\Theta^{*,h}(\mathcal{S}^h\llcorner\Gamma,x)}{\Theta^h_*(\mathcal{S}^h\llcorner\Gamma,x)}\leq 1, \quad \text{for $\mathcal{S}^h\llcorner\Gamma_i$-almost every $x\in \mathbb G$}.
\end{equation}
Let us now fix $i\in \mathbb N$ and note there exists a unique $\ell(i)\in\mathbb N$ such that 
$$
1/\ell(i) < \oldep{ep:Cool}(\mathbb V_i,\mathbb L_i) \leq 1/(\ell(i)-1).
$$
Moreover, from the definition of $\beta$ and $\mathscr F$ we see that $\beta(\mathbb V_i)=\beta(\varepsilon,\ell(i),h)$. This allows us to infer that
\begin{enumerate}
    \item $\alpha_i\leq \beta(\mathbb V_i)=\beta(\varepsilon,\ell(i),h)$, since  $\alpha_i:= \min\{\oldep{ep:Cool}(\mathbb V_i,\mathbb L_i),\beta(\mathbb V_i)\}$,
    \item $\oldC{C:split}(\mathbb V_i,\mathbb L_i)>1/\ell(i)$, since $1/\ell(i) < \oldep{ep:Cool}(\mathbb V_i,\mathbb L_i)= \oldC{C:split}(\mathbb V_i,\mathbb L_i)/2$, see \cref{lemma:LCapCw=e}.
\end{enumerate}
Thus we can apply \cref{lem:LemmaContoso} and conclude that 
$$
(1-\mathfrak{c}(\alpha_i))^{2h}(1+\mathfrak{c}(\alpha_i))^{-h} \geq 1-\varepsilon.
$$
This shows, thanks to \eqref{eqn:USEFUL}, that for any $i\in \mathbb N$, we have
$$
1-\varepsilon\leq \frac{\Theta^{*,h}(\mathcal{S}^h\llcorner\Gamma,x)}{\Theta^h_*(\mathcal{S}^h\llcorner\Gamma,x)}\leq 1, \quad \text{for $\mathcal{S}^h\llcorner\Gamma_i$-almost every $x\in \mathbb G$}.
$$
Thus by taking into account \eqref{eqn:UsefulToConclude} and the previous equation we conclude \eqref{eqn:WhatWeNeed}, that is the sought claim.
\end{proof}
\begin{osservazione}\label{rem:DensityOne}
It is a classical result that if $E\subseteq \mathbb R^n$ is a $h$-rectifiable set, with $1\leq h\leq n$, then $\Theta^h(\mathcal{S}^h\llcorner E,x)=1$ for $\mathcal{S}^h$-almost every point $x\in E$, see \cite[Theorem 3.2.19]{Federer1996GeometricTheory}. This is true also in the setting of Heisenberg groups for arbitrary $\mathscr{P}_h^c$-rectifiable measures, and it is a direct consequence of \cite[(iv)$\Rightarrow$(ii) of Theorem 3.14 \& Theorem 3.15]{MatSerSC}.

We point out that as a consequence of the non-trivial results developed in the subsequent paper \cite{antonelli2020rectifiable2}, see \cite[Theorem 1.1]{antonelli2020rectifiable2}, we have that whenever $\Gamma\subseteq \mathbb G$ is a Borel set such that $0<\mathcal{S}^h(\Gamma)<+\infty$, and $\mathcal{C}^h\llcorner\Gamma$ is $\mathscr{P}_h^c$-rectifiable, then  $\Theta^h(\mathcal{C}^h\llcorner\Gamma,x)=1$ for $\mathcal{C}^h$-almost every $x\in\Gamma$.

\end{osservazione}

\begin{corollario}\label{coroll:DENSITY}
Let $\phi$ be a $\mathscr{P}^c_h$-rectifiable measure on a Carnot group $\mathbb G$. Then 
$$
0<\Theta^h_*(\phi,x)=\Theta^{*,h}(\phi,x)<+\infty, \quad \text{for $\phi$-almost every $x\in\mathbb G$.}
$$
\end{corollario}
\begin{proof}
We stress that by restricting ouserlves on balls of integer radii, by using \cref{prop:Lebesuge} and \cref{prop:LocalityOfTangent}, we can assume that $\phi$ has compact support. Let us first recall that, by \cref{prop::E}, we have 
\begin{equation}\label{eqn:FINEDELLEFINI}
\phi\left(\mathbb G\setminus \bigcup_{\vartheta,\gamma\in\mathbb N}E(\vartheta,\gamma)\right) = 0.
\end{equation}
Let us fix $\vartheta,\gamma\in\mathbb N$. From Lebesgue's differentiation theorem, see \cref{prop:Lebesuge}, and the locality of tangents, see \cref{prop:LocalityOfTangent}, we deduce that $\phi$ being $\mathscr{P}_h^c$-rectifiable implies that $\phi\llcorner E(\vartheta,\gamma)$ is $\mathscr{P}_h^c$-rectifiable. From \cref{prop:MutuallyEthetaGamma} we deduce that $\phi\llcorner E(\vartheta,\gamma)$ is mutually absolutely continuous with respect to $\mathcal{S}^h\llcorner E(\vartheta,\gamma)$, and thus, by Radon-Nikodym theorem, see \cite[page 82]{HeinonenKoskelaShanmugalingam}, there exists a positive function $\rho\in L^1(\mathcal{S}^h\llcorner E(\vartheta,\gamma))$ such that $\phi\llcorner E(\vartheta,\gamma)=\rho\mathcal{S}^h\llcorner E(\vartheta,\gamma)$. We stress that we can apply Lebesgue-Radon-Nikodym theorem since $\phi\llcorner E(\vartheta,\gamma)$ is asymptotically doubling because it has positive $h$-lower density and finite $h$-upper density almost everywhere. By Lebesgue-Radon-Nikodym theorem, see \cite[page 82]{HeinonenKoskelaShanmugalingam}, and the locality of tangents again, we deduce that $\mathcal{S}^h\llcorner E(\vartheta,\gamma)$ is a $\mathscr{P}_h^c$-rectifiable measure, since $\phi\llcorner E(\vartheta,\gamma)$ is a $\mathscr{P}_h^c$-rectifiable measure. 
Thus we can apply \cref{thm:ExistenceOfDensitySets} to $\mathcal{S}^h\llcorner E(\vartheta,\gamma)$ and obtain that for every $\vartheta,\gamma\in\mathbb N$ we have that 
$$
0<\Theta^h_*(\mathcal{S}^h\llcorner E(\vartheta,\gamma),x)=\Theta^{*,h}(\mathcal{S}^h\llcorner E(\vartheta,\gamma),x)<+\infty, \quad \text{for $\mathcal{S}^h\llcorner E(\vartheta,\gamma)$-a.e. $x\in\mathbb G$.}
$$
Since $\phi\llcorner E(\vartheta,\gamma)=\rho\mathcal{S}^h\llcorner E(\vartheta,\gamma)$ we thus conclude from the previous equality and by Lebesgue-Radon-Nikdoym theorem that 
for every $\vartheta,\gamma\in\mathbb N$ we have that 
$$
0<\Theta^h_*(\phi\llcorner E(\vartheta,\gamma),x)=\Theta^{*,h}(\phi\llcorner E(\vartheta,\gamma),x)<+\infty, \quad \text{for $\phi\llcorner E(\vartheta,\gamma)$-a.e. $x\in\mathbb G$.}
$$
The previous equality, jointly with \cref{prop:Lebesuge} and together with \eqref{eqn:FINEDELLEFINI}
allows us to conclude the proof.
\end{proof}


\section{Comparison with other notions of rectifiability}\label{sec:Comparison}

In this section we provide the proof of \cref{prop:GG'Intro} and \cref{coroll:IdiffMeasureIntroNEW}. 
The key step for proving the rectifiability with intrinsically differentiable graphs is the following proposition.
\begin{proposizione}[Hausdorff convergence to tangents]\label{prop:TANGENTSIntro}
Let $\phi$ be a $\mathscr{P}_h$-rectifiable measure. 
Let $K$ be a compact 
set
such that $\phi(K)>0$. Then for $\phi$-almost every point $x\in K$ there exists $\mathbb V(x)\in\G(h)$ such that
$$
\delta_{1/r}(x^{-1}\cdot K) \to \mathbb V(x), \qquad \text{as r goes to $0$},
$$
in the sense of Hausdorff convergence on closed balls $\{B(0,k)\}_{k>0}$.
\end{proposizione}
First of all, by reducing the measure $\phi$ to have compact support, e.g., considering the restriction on the balls with integer radii, and then by using \cref{prop::E}, we can assume without loss of generality that $K\subseteq E(\vartheta,\gamma)$ for some $\vartheta,\gamma \in\mathbb N$. In order to prove the Hausdorff convergence to the plane $\mathbb V(x)$ we need to prove two different things: first, around almost every point $x$ of $K$, the points of the set $K$ at decreasingly small scales lies ever closer to the points of $x\mathbb V(x)$, and this is exactly what comes from the implication \eqref{eqn:Implication1Intro}, see \cref{prop1vsinfty}. Secondly, we have to prove the converse assertion with respect to the previous one, i.e., that the points of $x\mathbb V(x)$ around $x$ at decreasingly small scales are ever closer to the points of $K$. For this latter assumption to hold we also need to add to the condition in \eqref{eqn:Implication1Intro} the additional control $F_{x,r}(\phi\llcorner K,\Theta\mathcal{S}^h\llcorner x\mathbb V)\leq \delta r^{h+1}$, see \cref{prop:bil2}. As a consequence of \cref{prop:TANGENTSIntro}, we can prove \cref{coroll:IdiffMeasureIntroNEW} for measures of the form $\mathcal{S}^h\llcorner\Gamma$.
Finally by the usual reduction to $E(\vartheta,\gamma)$, we can give the proof of \cref{coroll:IdiffMeasureIntroNEW} for arbitrary measures.
\subsection{$C^1_{\mathrm H}(\mathbb G,\mathbb G')$-rectifiability}
This subsection is devoted to the proof of \cref{prop:GG'Intro}, i.e., the fact that the spherical Hausdorff measure restricted to a $(\mathbb G,\mathbb G')$-rectifiable set is $\mathscr P$-rectifiable. In \cite{JNGV20} the authors give the following definitions of $C^1_{\mathrm H}$-submanifold of a Carnot group and rectifiable sets. We first recall the definition of $C^1_{\mathrm H}$-function.
\begin{definizione}[$C^1_{\mathrm H}$-function]\label{def:C1h}
Let $\mathbb G$ and $\mathbb G'$ be two Carnot groups endowed with left-invariant homogeneous distances $d$ and $d'$, respectively. Let $\Omega\subseteq \mathbb G$ be open and let $f:\Omega\to\mathbb G'$ be a function. We say that $f$ is {\em Pansu differentiable at $x\in\Omega$} if there exists a homogeneous homomorphism $df_x:\mathbb G\to\mathbb G'$ such that
$$
\lim_{y\to x}\frac{d'(f(x)^{-1}\cdot f(y),df_x(x^{-1}\cdot y))}{d(x,y)}=0.
$$
Moreover we say that $f$ is {\em of class $C^1_{\mathrm H}$ in $\Omega$} if the map $x\mapsto df_x$ is continuous from $\Omega$ to the space of homogeneous homomorphisms from $\mathbb G$ to $\mathbb G'$.
\end{definizione}
\begin{definizione}[$C^1_{\mathrm H}$-submanifold]\label{def:C1Hmanifold}
Given an arbitrary Carnot group $\mathbb G$, we say that $\Sigma\subseteq \mathbb G$ is a {\em $C^1_{\mathrm H}$-submanifold} of $\mathbb G$ if there exists a Carnot group $\mathbb G'$ such that for every $p\in \Sigma$ there exists an open neighborhood $\Omega$ of $p$ and a function $f\in C^1_{\mathrm H}(\Omega;\mathbb G')$ such that 
\begin{equation}\label{eqn:RepresentationOfSigma}
\Sigma\cap \Omega =\{g\in\Omega:f(g)=0\},
\end{equation}
and $df_p:\mathbb G\to\mathbb G'$ is surjective with ${\mathrm Ker}(df_p)$ complemented. In this case we say that $\Sigma$ is a {\em $C^1_{\mathrm H}(\mathbb G,\mathbb G')$-submanifold}.
\end{definizione}
\begin{definizione}[($\mathbb G,\mathbb G')$-rectifiable set]\label{def:GG'Rect}
Given two arbitrary Carnot groups $\mathbb G$ and $\mathbb G'$ of homogeneous dimension $Q$ and $Q'$, respectively, we say that $\Sigma\subseteq \mathbb G$ is a {\em $(\mathbb G,\mathbb G')$-rectifiable set} if there exist countably many subsets $\Sigma_i$ of $\mathbb G$ that are $C^1_{\mathrm H}(\mathbb G,\mathbb G')$-submanifolds, such that 
$$
\mathcal{H}^{Q-Q'}\left(\Sigma\setminus\bigcup_{i=1}^{+\infty}\Sigma_i\right)=0.
$$
\end{definizione}
Using the results of \cite{JNGV20}, we prove the following.
\begin{proposizione}\label{prop:GG'IsPh}
Let us fix $\mathbb G$ and $\mathbb G'$ two arbitrary Carnot groups of homogeneous dimensions $Q$ and $Q'$ respectively and suppose $\Sigma\subseteq \mathbb G$ is a $(\mathbb G,\mathbb G')$-rectifiable set. Then the measure $\mathcal{S}^{Q-Q'}\llcorner \Sigma$ is $\mathscr{P}_{Q-Q'}^c$-rectifiable.
\end{proposizione}
\begin{proof}
By \cite[Corollary 3.6]{JNGV20} a $(\mathbb G,\mathbb G')$-rectifiable set $\Sigma$ has $\mathcal{S}^{Q-Q'}\llcorner\Sigma$-almost everyhwere positive and finite density. Thus, by the locality of tangents, see \cref{prop:LocalityOfTangent}, by Lebesgue differentiation theorem in \cref{prop:Lebesuge}, and by the very definitions of $(\mathbb G,\mathbb G')$-rectifiable set and $C^1_{\mathrm H}(\mathbb G,\mathbb G')$-submanifold, it suffices to prove the statement when $\Sigma$ is the zero-level set of a function $f\in C^1_{\mathrm H}(\Omega,\mathbb G')$, with $\Omega\subseteq \mathbb G$ open, and such that for every $p\in \{g\in\Omega:f(g)=0\}=:\Sigma$ the differential $df_p:\mathbb G\to\mathbb G'$ is surjective with $\mathrm{Ker}(df_p)$ complemented.

Fix $p\in\Sigma$ and note that the homogeneous subgroup  $\mathrm{Ker}(df_p)$, where $f$ is a representation as in \eqref{eqn:RepresentationOfSigma}, is independent of the choice of $f$. This follows for instance from \cite[Lemma 2.14, (iii)]{JNGV20}. We denote this homogeneous subgroup with $\mathbb W(p)$ and we call it the {\em tangent subgroup} at $p$ to $\Sigma$. We first prove that
\begin{equation}\label{eqn:WantToObtain}
\mathrm{Tan}_{Q-Q'}\big(\mathcal S^{Q-Q'}\llcorner\Sigma,p\big)\subseteq \{\lambda\mathcal{S}^{Q-Q'}\llcorner \mathbb W(p): \lambda >0\}, \qquad \mbox{for every $p\in\Sigma$}.
\end{equation} 
Indeed, from \cite[Lemma 3.4]{JNGV20}, denoting by $\Sigma_{p,r}$ the set $ \delta_{1/r}(p^{-1}\cdot\Sigma)$, we have 
\begin{equation}\label{eqn:Claim}
\mathcal{S}^{Q-Q'}\llcorner \Sigma_{p,r} \rightharpoonup \mathcal{S}^{Q-Q'}\llcorner \mathbb W(p), \qquad \mbox{for every $p\in\Sigma$ and for $r\to 0$}.
\end{equation}
We claim that this last equality implies that
$$
r_i^{-(Q-Q')}T_{p,r_i}\big(\mathcal{S}^{Q-Q'}\llcorner\Sigma\big) \rightharpoonup \mathcal{S}^{Q-Q'}\llcorner \mathbb W(p), \mbox{ for every infinitesimal sequence $r_i$},
$$
thus showing \eqref{eqn:WantToObtain}. Indeed, for every measurable set $A\subseteq \mathbb G$, we have
\begin{equation}
\begin{split}
     T_{p,r_i}\big(\mathcal{S}^{Q-Q'}\llcorner\Sigma\big)(A)=\mathcal{S}^{Q-Q'}\llcorner \Sigma(p\cdot\delta_{r_i}(A))=\mathcal{S}^{Q-Q'}\llcorner (p^{-1}\cdot \Sigma) (\delta_{r_i}(A))
     =r_i^{Q-Q'}\mathcal{S}^{Q-Q'}\llcorner \Sigma_{p,r_i}(A),
\end{split}
\end{equation}
and thus the claim follows from \eqref{eqn:Claim}. In order to conclude the proof, we have to prove that item (i) of \cref{def:PhRectifiableMeasure} holds. This follows from \cite[Corollary 3.6]{JNGV20}. Indeed it is there proved that every $(\mathbb G,\mathbb G')$-rectifiable set has density $\mathcal{S}^{Q-Q'}$-almost everywhere, that is stronger than item (i) of \cref{def:PhRectifiableMeasure}. 
\end{proof}

\begin{osservazione}
We remark that the proof above is heavily based on \cite[Lemma 3.4 \& Corollary 3.6]{JNGV20}. The two latter results in the reference are consequences of the area formula \cite[Theorem 1.1]{JNGV20}. As a consequence the approach in \cite{JNGV20} is, in some sense, reversed with respect to our approach. The authors in \cite{JNGV20} deal with the category of $C^1_{\mathrm H}(\mathbb G,\mathbb G')$-regular submanifolds and prove the area formula relying upon \cite[Proposition 2.2]{JNGV20}, that ultimately tells that a Borel regular measure $\mu$ with positive and finite Federer's density $\theta$ with respect to the spherical Hausdorff measure $\mathcal{S}^h$ admits a representation $\mu=\theta\mathcal{S}^h$. Then with this area formula they are able to prove the results that led to the proof of the above \cref{prop:GG'IsPh}. 

We stress that in the subsequent paper \cite{antonelli2020rectifiable2} we push forward the study of $\mathscr{P}$-rectifiable measures started here, and we prove an area formula for intrinsically differentiable graphs, see \cite[Theorem 1.3]{antonelli2020rectifiable2}, that extends the result of \cite[Theorem 1.1]{JNGV20}. 
\end{osservazione}
\begin{osservazione}[$\mathscr{P}$-rectifiability and $(\mathbb G,\mathbb G')$-rectifiable sets]\label{oss:LackOfGenerality}
 From \cref{def:C1Hmanifold} and \cref{def:GG'Rect} it follows that the tangent subgroup $\mathbb W$ at a point of a $(\mathbb G,\mathbb G')$-rectifiable set is always normal and complemented. Moreover, from \cite[Lemma 2.14, (iv)]{JNGV20}, every complementary subgroup to $\mathbb W$ must be a \textbf{Carnot subgroup} of $\mathbb G$ that in addition is isomorphic to $\mathbb G'$. This results in a lack of generality of this approach to rectifiability. Let us give here an example where the previous phenomenon becomes clear. If we take $L$ a horizontal subgroup in the first Heisenberg group $\mathbb H^1$, on the one hand $\mathcal{S}^1\llcorner L$ is $\mathscr{P}_1$-rectifiable, on the other hand $L$ is not $(\mathbb H^1,\mathbb G')$-rectifiable for any Carnot group $\mathbb G'$ since $L$ is not normal. 
\end{osservazione}

\subsection{Rectifiability with intrinsically differentiable graphs}
This subsection is devoted to the proof of \cref{prop:TANGENTSIntro} and \cref{coroll:IdiffMeasureIntroNEW}. Throughout this subsection we let $\mathbb G$ to be a Carnot group of homogeneous dimension $Q$ and $h$ an arbitrary natural number with $1\leq h\leq Q$. Whenever $\phi$ is a Radon measure supported on a compact set we freely use the notation $E(\vartheta,\gamma)$ introduced in \cref{def:EThetaGamma}, for $\vartheta,\gamma\in\mathbb N$. We start with some useful definitions and facts. 
\begin{definizione}\label{def:C4}
For $1\leq h\leq Q$ and $\vartheta\in\mathbb N$, let us set 
$$
\eta(h):=1/(h+1),
$$
and then let us define the constant
$$
\newC\label{G1}=\oldC{G1}(h,\vartheta):= 
\left(\frac{\eta(1-\eta)^h}{32\vartheta}\right)^{h+2}.
$$
\end{definizione}

\begin{proposizione}\label{prop:bil2}
Let $\phi$ be a Radon measure supported on a compact subset of $\mathbb G$ and let $K$ be a Borel subset of $\mathrm{supp}\phi$. Let $\vartheta,\gamma$ and $1\leq h\leq Q$ be natural numbers. Let $x\in E(\vartheta,\gamma)$, $0<r<1/\gamma$, and $0<\delta<\oldC{G1}$. Assume further that there exist $\Theta>0$ and $\mathbb V\in \G(h)$ such that
\begin{equation}\label{eqn:HypothesisControl}
F_{x,r}(\phi\llcorner K,\Theta\mathcal{C}^h\llcorner x\mathbb V)+F_{x,r}(\phi,\Theta\mathcal{C}^h\llcorner x\mathbb V) \leq 2\delta r^{h+1}.
\end{equation}
Then for any $w\in B(x,r/2)\cap x\mathbb V$ we have
$\phi(K\cap B(w,\delta^\frac{1}{h+2} r))>0$, and thus in particular $K\cap B(w,\delta^\frac{1}{h+2} r) \neq \emptyset$.
\end{proposizione}

\begin{proof}
From the hypothesis we have that $F_{x,r}(\phi,\Theta\mathcal{C}^h\llcorner x\mathbb V) \leq 2\delta r^{h+1}$. 
Define $g(x):=\min\{\dist(x,U(0,1)^c),\eta\}$, where $\eta$ is defined in \cref{def:C4}. From the very definition of the function $g$ and the choice of $\Theta$ above we deduce that
\begin{equation}
    \begin{split}
        \vartheta^{-1}(1-\eta)^h\eta r^{h+1}- \Theta \eta r^{h+1} &\leq\eta r\phi\big(B(x,(1-\eta)r)\big)-\eta r\Theta\mathcal{C}^h\llcorner x\mathbb V(B(x,r))\\
        &\leq \int r g(\delta_{1/r}(x^{-1}z)) d\phi(z)-\Theta\int r g(\delta_{1/r}(x^{-1}z)) d\mathcal{C}^h\llcorner x\mathbb V(z)\leq2\delta r^{h+1},
        \nonumber
    \end{split}
\end{equation}
where in the first inequality we are using that $x\in E(\vartheta,\gamma)$ and \cref{rem:Ch1}, and in the last inequality we are using that $rg(\delta_{1/r}(x^{-1}\cdot))\in\lip(B(x,r))$.
Simplifying and rearranging the above chain of inequalities, we infer that
$$
\Theta\geq \vartheta^{-1}(1-\eta)^{h}-2\delta/\eta\underset{(A)}{\geq}(2\vartheta)^{-1}(1-\eta)^{h}\underset{(B)}{=}(2\vartheta)^{-1}(1-1/(h+1))^h,$$
where (A) comes from the fact that $\delta<\oldC{G1}< ((1-\eta)^h\eta)/(4\vartheta)$, see \cref{def:C4}, and (B) comes from the definition of $\eta$, see \cref{def:C4}.
Since the function $h\mapsto(1-1/(h+1))^h$ is decreasing and bounded below by $e^{-1}$, we deduce, from the previous inequality, that $\Theta\geq 1/(2\vartheta e)$.

We now claim that for every $\lambda$ with $\delta^{1/(h+2)}\leq \lambda<1/2$ and every $w\in x\mathbb V\cap B(x,r/2)$ we have $\phi \big(B(w,\lambda r)\cap K\big)>0$. This will finish the proof. By contradiction assume there is $w\in x\mathbb V\cap B(x,r/2)$ such that $\phi \big(B(w,\lambda r)\cap K\big)=0$. This would imply that
\begin{equation}
\begin{split}
    \Theta \eta(1-\eta)^{h}\lambda^{h+1}r^{h+1}&=\Theta\eta \lambda r \mathcal{C}^{h}\llcorner x\mathbb V\big(B(w,(1-\eta)\lambda r)\big)\\
&\leq \Theta\int \lambda rg(\delta_{1/(\lambda r)}(w^{-1}z))d\mathcal{C}^{h}\llcorner x\mathbb V(z)\\
&=\Theta\int \lambda rg(\delta_{1/(\lambda r)}(w^{-1}z))d\mathcal{C}^{h}\llcorner x\mathbb V(z)-\int \lambda rg(\delta_{1/(\lambda r)}(w^{-1}z)) d\phi\llcorner K(z)
\leq2\delta r^{h+1},
\label{eq:n104}
\end{split}
\end{equation}
where the first equality comes from \cref{rem:Ch1}, and the last inequality comes from the choice of $\Theta$ as in the statement, and the fact that \\ $\lambda rg(\delta_{1/(\lambda r)}(w^{-1}\cdot))\in\lip(B(w,\lambda r)) \subseteq \lip(B(x,r))$ because $\lambda<1/2$ and $w\in B(x,r/2)$.
Thanks to \eqref{eq:n104}, the choice of $\lambda$, and the fact, proved some line above, that $1/(4e\vartheta)<\Theta$, we have that
\begin{equation}
  \frac{\delta^{\frac{h+1}{h+2}}}{4e\vartheta}\eta(1-\eta)^{h}<\Theta\lambda^{h+1} \eta(1-\eta)^{h}\leq2\delta, \quad \text{and then} \quad  \delta^{1/(h+2)} \geq \frac{\eta(1-\eta)^h}{8e\vartheta},
  \nonumber
\end{equation}
which is a contradiction since $\delta < \oldC{G1}= ((\eta(1-\eta)^h)/(32\vartheta))^{h+2}$, see \cref{def:C4}.
\end{proof}

\begin{proof}[Proof of \cref{prop:TANGENTSIntro}]
First of all, by reducing the measure $\phi$ to have compact support, e.g., considering the restriction on the balls with integer radii, and then by using \cref{prop::E}, we can assume without loss of generality that $K\subseteq E(\vartheta,\gamma)$ for some $\vartheta,\gamma \in\mathbb N$

Since $\phi$ is a $\mathscr{P}_h$-rectifiable measure, by using the locality of tangents with the density $\rho\equiv \chi_K$, see \cref{prop:LocalityOfTangent}, for $\phi$-almost every $x\in K$ we have that the following three conditions hold
\begin{itemize}
    \item[(i)] $\mathrm{Tan}_h(\phi,x)\subseteq \{\lambda\mathcal{S}^h\llcorner\mathbb V(x):\lambda>0\}$, where $\mathbb V(x)\in\G(h)$,
    \item[(ii)] $ 0<\Theta^h_*(\phi,x)\leq \Theta^{h,*}(\phi,x)<+\infty$. 
    \item[(iii)] if $r_i\to 0$ is such that there exists $\Theta>0$ with $r_i^{-h}T_{x,r_i}\phi \to \Theta\mathcal{C}^h\llcorner \mathbb V(x)$, then $ r_i^{-h}T_{x,r_i}(\phi\llcorner K)\to \Theta\mathcal{C}^h\llcorner \mathbb V(x)$.
\end{itemize}
From now on let us fix a point $x\in K$ for which the three conditions above hold. If we are able to prove the convergence in the statement for such a point then the proof of the proposition is concluded.

Thus, we have to show that for every $k>0$ the following holds
\begin{equation}\label{eqn:Hausdorff0}
\lim_{r\to 0} d_{H,\mathbb G}(\delta_{1/r}(x^{-1}\cdot K) \cap B(0,k), \mathbb V(x)\cap B(0,k))=0,
\end{equation}
where $d_{H,\mathbb G}$ is the Hausdorff distance between closed subsets in $\mathbb G$. For some compatibility with the statements that we already proved, we are going to prove \eqref{eqn:Hausdorff0} for $k=1/4$. The proof of \eqref{eqn:Hausdorff0} for an arbitrary $k>0$ can be achieved by changing accordingly the constants in the statements of \cref{prop1vsinfty} and \cref{prop:bil2}, that we are going to crucially use in this proof. We leave this generalization to the reader, as it will be clear from this proof. 

Let us fix $
\varepsilon<\min\{\delta_{\mathbb G},\oldC{G1}\}$, 
where $\delta_{\mathbb G}$ is defined in \cref{eqn:DefinitionOfDeltaG} and \oldC{G1} in \cref{def:C4}, and let us show that there exist an $r_0=r_0(\varepsilon)$ and a real function $f_1$ such that
\begin{equation}\label{eqn:Bound}
d_{H,\mathbb G}\left(\delta_{1/r}(x^{-1}\cdot K) \cap B(0,1/4),\mathbb V(x)\cap B(0,1/4)\right) \leq f_1(\varepsilon), \qquad \text{for all $0<r<r_0(\varepsilon)$},
\end{equation}
where
\begin{equation}
\qquad  f_1(\varepsilon):=\max\{\oldC{C:b1}\varepsilon^{1/(h+1)}+f_2(\varepsilon),3\varepsilon^{1/(h+2)}+f_3(\varepsilon)\},
\end{equation}
and where the constant $\oldC{C:b1}$ is defined in \cref{prop1vsinfty}, and the functions $f_2$, $f_3$ are introduced in \eqref{eqn:DefinitionF2} and \eqref{eqn:Definitionf3}, respectively. By the definition of $f_1,f_2,f_3$ it follows that $f_1(\varepsilon)\to 0$ as $\varepsilon\to 0$ and thus, if we prove \eqref{eqn:Bound}, we are done. 

In order to reach the proof of \eqref{eqn:Bound} let us add an intermediate step. We claim that there exists an $r_0:=r_0(\varepsilon)<1/\gamma$ such that the following holds
\begin{equation}\label{eqn:ConclusionContradiction}\begin{split}
    \text{for every $0<r<r_0$ there exists a $\Theta:=\Theta(r)$ for which}
    F_{x,r}(\phi\llcorner K,\Theta\mathcal{C}^h\llcorner x\mathbb V)+F_{x,r}(\phi,\Theta\mathcal{C}^h\llcorner x\mathbb V) \leq 2\varepsilon r^{h+1}.
\end{split}
\end{equation}
The conclusion in  \eqref{eqn:ConclusionContradiction} follows if we prove that 
\begin{equation}\label{eqn:ConclusionContradiction2}
\lim_{r\to 0}\inf_{\Theta>0}\frac{ F_{x,r}(\phi\llcorner K,\Theta\mathcal{C}^h\llcorner x\mathbb V)+F_{x,r}(\phi,\Theta\mathcal{C}^h\llcorner x\mathbb V)}{r^{h+1}} \to 0.
\end{equation}
We prove \eqref{eqn:ConclusionContradiction2} by contradiction. If \eqref{eqn:ConclusionContradiction2} was not true, there would exist an $\widetilde\varepsilon$ and an infinitesimal sequence $\{r_i\}_{i\in\mathbb N}$ such that 
\begin{equation}\label{eqn:ConclusionContradiction3}
\inf_{\Theta>0} \left(F_{x,r_i}(\phi\llcorner K,\Theta\mathcal{C}^h\llcorner x\mathbb V)+F_{x,r_i}(\phi,\Theta\mathcal{C}^h\llcorner x\mathbb V)\right) > \widetilde\varepsilon r_i^{h+1}, \quad \text{for every $i\in\mathbb N$.}
\end{equation}
Thus, from items (i) and (ii) above, and from \cite[Corollary 1.60]{AFP00}, we conclude that, up to a non re-labelled subsequence of $r_i$, there exists a $\Theta^*>0$ such that we have $r_i^{-h}T_{x,r_i}\phi\to \Theta^*\mathcal{C}^h\llcorner\mathbb V(x)$ as $r_i\to 0$. Then by exploiting the item (iii) above we get also that $r_i^{-h}T_{x,r_i}(\phi\llcorner K)\to \Theta^*\mathcal{C}^h\llcorner\mathbb V(x)$ as $r_i\to 0$. These two conclusions immediately imply, by exploiting \cref{rem:ScalinfFxr} and \eqref{prop:WeakConvergenceAndFk}, that 
$$
\lim_{i\to+\infty} r_i^{-(h+1)}\left(F_{x,r_i}(\phi\llcorner K,\Theta^*\mathcal{C}^h\llcorner x\mathbb V)+F_{x,r_i}(\phi,\Theta\mathcal{C}^h\llcorner x\mathbb V)\right) \to 0,
$$
which is a contradiction with \eqref{eqn:ConclusionContradiction3}. Thus, the conclusion in \eqref{eqn:ConclusionContradiction} holds. Let us continue the proof of \eqref{eqn:Bound}.
 
Taking into account the bound on $\varepsilon$ and \eqref{eqn:ConclusionContradiction} we can apply \cref{prop1vsinfty}, since $\mathbb V(x)\in \Pi_{\varepsilon}(x,r)$ for all $0<r<r_0$, and \cref{prop:bil2} to obtain, respectively, that for all $0<r<r_0$
\begin{equation}\label{eqn:ConclusionHausdorff}
    \begin{split}
        \sup_{p\in K\cap B(x,r/4)} \dist(p,x \mathbb V(x)) \leq \sup_{p\in E(\vartheta,\gamma)\cap B(x,r/4)} \dist(p,x \mathbb V(x)) \leq \oldC{C:b1}r\varepsilon^{1/(h+1)}, \\
        \text{for every $p\in B(x,r/2)\cap x\mathbb V(x)$ we have $B(p,\varepsilon^{1/(h+2)}r)\cap K\neq \emptyset$}.
    \end{split}
\end{equation}
Let us proceed with the proof of \eqref{eqn:Bound}. Fix $0<r<r_0$ and note that
 for any $w\in \delta_{1/r}(x^{-1}\cdot K)\cap B(0,1/4)$ there exists a point $p\in K\cap B(x,r/4)$ such that $w=:\delta_{1/r}(x^{-1}\cdot p)$. From the first line of \eqref{eqn:ConclusionHausdorff} we get that $\dist(x^{-1}\cdot p,\mathbb V(x))\leq \oldC{C:b1}r\varepsilon^{1/(h+1)}$ and thus there exists a $v\in \mathbb V(x)$ such that $d(x^{-1}\cdot p,v)\leq \oldC{C:b1}r\varepsilon^{1/(h+1)}$. This in particular means that $d(w,\delta_{1/r}v)\leq \oldC{C:b1}\varepsilon^{1/(h+1)}$ and then, since $w\in B(0,1/4)$, we get also that $\delta_{1/r}v \in \mathbb V(x)\cap B(0,1/4+\oldC{C:b1}\varepsilon^{1/(h+1)})$. Thus, we conclude that 
\begin{equation}\label{eqn:Finale1}
\dist(w,\mathbb V(x)\cap B(0,1/4+\oldC{C:b1}\varepsilon^{1/(h+1)}))\leq \oldC{C:b1}\varepsilon^{1/(h+1)},
\end{equation}
for all $w\in \delta_{1/r}(x^{-1}\cdot K)\cap B(0,1/4)$.
Define the following function
\begin{equation}\label{eqn:DefinitionF2}
f_2(\varepsilon):=\sup_{u\in \mathbb V(x)\cap \big(B(0,1/4+\oldC{C:b1}\varepsilon^{1/(h+1)})\setminus U(0,1/4)\big)} d(u,\delta_{4^{-1}\|u\|^{-1}}u),
\end{equation}
and notice that by compactness it is easy to see that $f_2(\varepsilon)\to 0$ as $\varepsilon\to 0$. With the previous definition of $f_2$ in hands, we can exploit \eqref{eqn:Finale1} and conclude that 
\begin{equation}\label{eqn:Finale1.1}
    \sup_{w\in \delta_{1/r}(x^{-1}\cdot K)\cap B(0,1/4)} \dist(w,\mathbb V(x)\cap B(0,1/4))\leq \oldC{C:b1}\varepsilon^{1/(h+1)}+f_2(\varepsilon).
\end{equation}

The latter estimate is the first piece of information we need to prove \eqref{eqn:Bound}. Let us now estimate $\dist(\delta_{1/r}(x^{-1}\cdot K)\cap B(0,1/4),v)$ for any $v\in \mathbb V(x)\cap B(0,1/4)$. If $u\in \mathbb V(x)\cap \big(B(0,1/4)\setminus U(0,1/4-\varepsilon^{1/(h+2)})\big)$, then there exists a unique $\mu=\mu(u)>0$ such that $\delta_{\mu(u)}u\in \mathbb V(x)\cap \partial B(0,1/4-\varepsilon^{1/(h+2)})$. Let us define 
\begin{equation}\label{eqn:Definitionf3}
f_3(\varepsilon):=\sup_{u\in \mathbb V(x)\cap\big(B(0,1/4)\setminus U(0,1/4-\varepsilon^{1/(h+2)})\big)} d(u,\delta_{\mu(u)}u),
\end{equation}
and by compactness it is easy to see that $f_3(\varepsilon)\to 0$ as $\varepsilon\to 0$. Let us now fix $v\in\mathbb V(x)\cap B(0,1/4)$.
Then $x\cdot \delta_{r}v\in B(x,r/4)\cap x\mathbb V(x)\subseteq B(x,r/2)\cap x\mathbb V(x)$.
We can use the second line of \eqref{eqn:ConclusionHausdorff} to conclude that there exists $w\in B(x\cdot \delta_r v,\varepsilon^{1/(h+2)}r)\cap K$. 
Thus $\widetilde w:=\delta_{1/r}(x^{-1}\cdot w) \in B(v,\varepsilon^{1/(h+2)})\cap \delta_{1/r}(x^{-1}\cdot K)$.
Now we have two cases
\begin{itemize}
    \item if $v$ was in $ B(0,1/4-\varepsilon^{1/(h+2)})$ we would get $\widetilde w \in B(0,1/4)$ and then 
\begin{equation}\label{eqn:FirstCase} \dist(\delta_{1/r}(x^{1}\cdot K)\cap B(0,1/4),v) \leq \varepsilon^{1/(h+2)};
\end{equation}
\item if instead $v\in \mathbb V(x)\cap \big(B(0,1/4)\setminus U(0,1/4-\varepsilon^{1/(h+2)})\big)$, we denote $v':=\delta_{\mu(v)}v$ the point that we have defined above and then we still have $x\cdot \delta_{r}v'\in B(x,r/2)\cap x\mathbb V(x)$. Thus we can again apply the second line of \eqref{eqn:ConclusionHausdorff} to deduce the existence of $w'\in B(x\cdot \delta_r v',\varepsilon^{1/(h+2)}r)\cap K$. Then we conclude $\widetilde w':=\delta_{1/r}(x^{-1}\cdot w') \in B(v',\varepsilon^{1/(h+2)})\cap \delta_{1/r}(x^{-1}\cdot K)$. Now we can estimate
\begin{equation}\label{eqn:widetildew'}
\begin{split}
    d(\widetilde w,\widetilde w') &= \frac{1}{r} d(w,w') \leq \frac{1}{r}\big(d(w,x\cdot \delta_r v)+d(x\cdot \delta _r v,x\cdot \delta_r v')+d(x\cdot\delta_r v',w')\big)\\
    &\leq 2\varepsilon^{1/(h+2)}+f_3(\varepsilon).
\end{split}
\end{equation}
Moreover, since $v'\in \partial B(0,1/4-\varepsilon^{1/(h+2)})$ and $\widetilde w'\in B(v',\varepsilon^{1/(h+2)})$ we get that $\widetilde w'\in B(0,1/4)\cap \delta_{1/r}(x^{-1}\cdot K)$. Then by the triangle inequality and \eqref{eqn:widetildew'} we conclude that, in this second case,
\begin{equation}
    d(\widetilde w',v)\leq 3\varepsilon^{1/(h+2)}+f_3(\varepsilon),
\end{equation}
and then
\begin{equation}\label{eqn:SecondCase}
\dist(\delta_{1/r}(x^{1}\cdot K)\cap B(0,1/4),v)\leq 3\varepsilon^{1/(h+2)}+f_3(\varepsilon).
\end{equation}
\end{itemize}

By joining together the conclusion of the two cases, see \eqref{eqn:FirstCase} and \eqref{eqn:SecondCase}, we conclude that 
\begin{equation}\label{eqn:FINAAAL}
\sup_{v\in \mathbb V(x)\cap B(0,1/4)} \dist(\delta_{1/r}(x^{1}\cdot K)\cap B(0,1/4),v) \leq 3\varepsilon^{1/(h+2)}+f_3(\varepsilon).
\end{equation}
The equations \eqref{eqn:Finale1.1} and \eqref{eqn:FINAAAL} imply \eqref{eqn:Bound} by the very definition of Hausdorff distance. Thus the proof is concluded.
\end{proof}
Let us now give the definition of intrinsically differentiable graph. 

\begin{definizione}[Intrinsically differentiable graph]\label{defiintrinsicdiffgraph}
   Let $\mathbb V$ and $\mathbb L$ be two complementary subgroups of a Carnot group $\mathbb G$. Let ${\varphi}:K\subseteq \mathbb V \to\mathbb L$ be a continuous function with $K$ compact in $\mathbb V$. Let $a_0\in K$. We say that $\mathrm{graph}(\varphi)$ is an {\em intrinsically differentiable graph} at $a_0\cdot\varphi(a_0)$ if there exists a homogeneous subgroup $\mathbb V(a_0)$ such that for all $k>0$
    \begin{equation}\label{eqn:TANGENTE}
\lim_{\lambda \to \infty } d_{H,\mathbb G}\left(\delta _\lambda ((a_0\cdot \varphi(a_0))^{-1}\cdot\mathrm{graph}(\varphi) )\cap B(0,k), \mathbb V(a_0)\cap B(0,k)\right)=0,
    \end{equation}
    where $d_{H,\mathbb G}$ is the Hausdorff distance between closed subsets of $\mathbb G$. \end{definizione}
We prove now that the support of a $\mathscr{P}_h^c$-rectifiable measure $\mathcal{S}^h\llcorner \Gamma$, where $\Gamma$ is compact, can be written as the countable union of almost everywhere intrinsically differentiable graphs.
\begin{teorema}\label{coroll:idiff}
For any $1\leq h\leq Q$, there exist a countable subfamily $\mathscr{F}:=\{\mathbb V_k\}_{k=1}^{+\infty}$ of $\G_c(h)$, and $\mathbb L_k$ complementary subgroups of $\mathbb V_k$ such that the following holds. 

Let $\Gamma$ be a compact subset of $\mathbb G$ such that $0<\mathcal{S}^h(\Gamma)<+\infty$, and $\mathcal{S}^h\llcorner \Gamma$ is a $\mathscr{P}^c_h$-rectifiable measure. Then for every $\alpha>0$, there are countably many compact $\Gamma_i$'s that are intrinsic graphs of functions $\varphi_i:P_{\mathbb V_i}(\Gamma_i)\to\mathbb L_i$, and that satisfy the following three conditions: $\Gamma_i$ are $C_{\mathbb V_i}(\alpha)$-sets, $\Gamma_i$ are intrinsically differentiable graphs at $a\cdot\varphi_i(a)$ for $\mathcal{S}^h\llcorner P_{\mathbb V_i}(\Gamma_i)$-almost every $a\in P_{\mathbb V_i}(\Gamma_i)$, and 
$$
\mathcal{S}^h(\Gamma\setminus\cup_{i=1}^{+\infty}\Gamma_i) = 0. 
$$
\end{teorema}
\begin{proof}
First of all let
$$
\bold F(\mathbb V,\mathbb L):=\oldep{ep:Cool}(\mathbb V,\mathbb L), \quad \text{for all $(\mathbb V,\mathbb L)\in\mathrm{Sub}(h)$},
$$
where $\mathrm{Sub}(h)$ is defined in \eqref{eqn:defBoldG}. Given the above defined function $\bold F$, we construct the family $\mathscr{F}:=\{\mathbb V_k\}_{k=1}^{+\infty}$ and choose $\mathbb L_k$ complementary subgroups of $\mathbb V_k$ as in the statement of \cref{thm:MainTheorem1}. Notice that this choice is dependent on the function $\bold F$ that we chose above. We claim that the family for which the statement holds is $\mathscr{F}$. 

Applying \cref{thm:MainTheorem1} with $\beta\equiv \min\{1/2,\alpha\}$ to the measure $\mathcal{S}^h\llcorner\Gamma$ we get countably many compact sets $\Gamma_i\subseteq \Gamma$ that are $C_{\mathbb V_i}(\min\{\bold F(\mathbb V_i,\mathbb L_i),\alpha\})$-sets and such that 
$$
\mathcal{S}^h(\Gamma\setminus\cup_{i=1}^{+\infty}\Gamma_i)=0.
$$
Since $\bold F(\mathbb V_i,\mathbb L_i)=\oldep{ep:Cool}(\mathbb V_i,\mathbb L_i)$, we conclude that each $\Gamma_i$ is also the intrinsic graph of a function $\varphi_i:P_{\mathbb V_i}(\Gamma_i)\to\mathbb L_i$, see \cref{prop:ConeAndGraph}. It is left to show that, for every $i\in\mathbb N$, $\mathrm{graph}(\varphi_i)$ is an intrinsically differentiable graph at $a\cdot\varphi_i(a)$ for  $\mathcal{S}^h\llcorner P_{\mathbb V}(\Gamma_i)$-almost every $a\in P_{\mathbb V_i}(\Gamma_i)$.

Indeed, since $\mathcal{S}^h\llcorner\Gamma$ is $\mathscr{P}_h^c$-rectifiable, we can apply \cref{prop:TANGENTSIntro} and,  for every $i\in\mathbb N$, we conclude that 
\begin{equation}\label{eqn:idiff}
\begin{split}
    \delta_{1/r}(x^{-1}\cdot \Gamma_i)\to \mathbb V(x), \quad \text{as $r$ goes to $0$,}
    \text{    for $\mathcal{S}^h\llcorner\Gamma_i$-almost every}\text{ $x\in \mathbb G$, where $\mathbb V(x)\in \G(h)$,}
\end{split}
\end{equation}
 in the sense of Hausdorff convergence on closed balls $\{B(0,k)\}_{k>0}$. Moreover, thanks to \cref{prop:mutuallyabs} and to Lebesgue differentiation theorem in \cref{prop:Lebesuge}, we infer that $(\Phi_i)_*\mathcal{S}^h\llcorner \mathbb V_i$ is mutually absolutely continuous with respect to $\mathcal{S}^h\llcorner\Gamma_i$, where $\Phi_i$ is the graph map of $\varphi_i$. Furthermore, since every point $x\in\Gamma_i$ can be written as $x=a\cdot\varphi_i(a)$, with $a\in P_{\mathbb V_i}(\Gamma_i)$, we conclude, from \eqref{eqn:idiff} and latter absolute continuity, that $\Gamma_i=\mathrm{graph}(\varphi_i)$ is an intrinsically differentiable graph at $a\cdot\varphi_i(a)$ for  $\mathcal{S}^h\llcorner P_{\mathbb V}(\Gamma_i)$-almost every $a\in P_{\mathbb V_i}(\Gamma_i)$, and this concludes the proof.
\end{proof}
In the following corollary we provide the proof of \cref{coroll:IdiffMeasureIntroNEW}.
\begin{corollario}\label{coroll:IdiffMeasure}
For any $1\leq h\leq Q$, there exist a countable subfamily $\mathscr{F}:=\{\mathbb V_k\}_{k=1}^{+\infty}$ of $\G_c(h)$, and $\mathbb L_k$ complementary subgroups of $\mathbb V_k$ such that the following holds. 

For any $\mathscr{P}^c_h$-rectifiable measure $\phi$ and for any $\alpha>0$, there exist countably many compact sets $\Gamma_i$'s that are $C_{\mathbb V_i}(\alpha)$-sets, that are intrinsic graphs of functions $\varphi_i:P_{\mathbb V_i}(\Gamma_i)\to\mathbb L_i$, and that satisfy the following conditions: $\Gamma_i$ are intrinsically differentiable graphs at $a\cdot\varphi_i(a)$ for $\mathcal{S}^h\llcorner P_{\mathbb V_i}(\Gamma_i)$-almost every $a\in P_{\mathbb V_i}(\Gamma_i)$, and 
$$
\phi(\mathbb G\setminus\cup_{i=1}^{+\infty}\Gamma_i) = 0. 
$$
\end{corollario}
\begin{proof}
By restricting on closed balls of integer radii we can assume without loss of generality that $\phi$ has compact support. Let us fix $\vartheta,\gamma\in\mathbb N$. We can infer this corollary by working on $\phi\llcorner E(\vartheta,\gamma)$, that is mutually absolutely continuous with respect to $\mathcal{S}^h\llcorner E(\vartheta,\gamma)$, see \cref{prop:MutuallyEthetaGamma},  and by using the previous \cref{coroll:idiff} together with \cref{prop::E}. The resulting strategy is identical to the one in \cref{coroll:DENSITY} so we omit the details.
\end{proof}

\begin{osservazione}[Uniformly intrinsically differentiable graphs and $C^1_{\mathrm H}(\mathbb G,\mathbb G')$-surfaces]\label{rem:C1GG'PhQuandoPossono}
By the recent work of the second named author, see \cite[Theorem 3]{MarstrandMattila20}, one can show that in an arbitrary Carnot group of homogeneous dimension $Q$, the support of a $\mathscr{P}^*_{Q-1}$-rectifiable measure can be covered by countably many $C^1_{\mathrm H}$-regular hypersurfaces. Moreover, it is known that a $C^1_{\mathrm H}$-regular hypersurface is characterized, locally, by being the graph of a \textbf{uniformly} intrinsically differentiable function, see \cite[Theorem 1.6]{ADDDLD20}. This means that, in some particular cases, as it is the codimension-one case, we can strenghten the conclusion in \cref{coroll:IdiffMeasure} by obtaining that the maps are uniformly intrinsically differentiable.

This latter observation gives raise to two questions, that in the co-horizontal case are the same thanks to \cite[Theorem 1.6]{ADDDLD20}, but in general could be different: is it always possible to improve the intrinsic differentiability in \cref{coroll:IdiffMeasure} to some kind of uniform intrinsic differentiability? Is it possible to prove that when a $\mathscr{P}_h$-rectifiable measure, or even a $\mathscr{P}_h^*$-rectifiable measure, on $\mathbb G$ admits only \textbf{complemented normal subgroups that have only complementary subgroups that are Carnot subgroups}, then we can write its support as the countable union of $C^1_{\mathrm H}(\mathbb G,\mathbb G')$-surfaces, see \cref{def:C1Hmanifold}? Let us stress that if one answers positively to the second question, this would mean, taking into account \cref{prop:GG'IsPh}, that whenever they can agree, see \cref{oss:LackOfGenerality}, the two notions of $\mathscr{P}$-rectifiable measure and $(\mathbb G,\mathbb G')$-rectifiable set agree. 

We do not address these questions in this paper, 
but we stress that with the results proved in \cite{antonelli2020rectifiable2}, we show that, at least in the co-horizontal case, the notion of $\mathscr{P}$-rectifiable measure and the notion of rectifiability given in terms of $(\mathbb G,\mathbb G')$-rectifiable sets coincide, see \cite[Corollary 5.3]{antonelli2020rectifiable2}.
\end{osservazione}

In the final part of this section we briefly discuss how the notion of intrinsically differentiable graph in \cref{defiintrinsicdiffgraph} is related to the already available notion of intrinsic differentiability, see \cite[Definition 3.2.1]{FMS14} and \cite[Definition 2.5]{AM20}. \textbf{Throughout the rest of this section $\mathbb V$ and $\mathbb L$ are two fixed complementary subgroups in a Carnot group $\mathbb G$.}
\begin{definizione}[Intrinsic translation of a function]\label{def:PhiQ}
	Given a function $\varphi\colon U\subseteq\mathbb V\to\mathbb L$, we define, for every $q\in\mathbb G$,
	\[
	{U}_q:=\{a\in\mathbb V: P_{\mathbb V}(q^{-1}\cdot a)\in {U}\},
	\]
	and ${\varphi}_q\colon{U}_q\subseteq \mathbb V\to\mathbb L$ by setting
	\begin{equation}\label{eqn:Phiq}
	{\varphi}_q(a):=\big(
	P_{\mathbb L}(q^{-1}\cdot a)\big)^{-1}\cdot {\varphi}\big(P_{\mathbb V}(q^{-1}\cdot a)\big).
	\end{equation}
\end{definizione}
\begin{definizione}[Intrinsically linear function]
    The map $\ell: \mathbb V \to \mathbb L$ is said to be {\em intrinsically linear} if $\mathrm{graph}(\ell)$ is a homogeneous subgroup of $\mathbb G$.
\end{definizione}
\begin{definizione}[Intrinsically differentiable function]\label{defiintrinsicdiff}
   Let ${\varphi}\colon{U}\subseteq \mathbb V \to\mathbb L$ be a function with $U$ Borel in $\mathbb V$. Fix a density point
   $a_0\in\mathcal D({U})$ of $U$, let $p_0:={\varphi}(a_0)^{-1}\cdot a_0^{-1}$ and denote with ${\varphi}_{p_0}\colon {U}_{p_0}\subseteq \mathbb V\to\mathbb L$ the shifted function introduced in \cref{def:PhiQ}. We say that ${\varphi}$ is {\em intrinsically differentiable} at $a_0$
	if there is an intrinsically linear map $d^{\varphi}\varphi_{a_0}\colon\mathbb V\to\mathbb L$ such that
	\begin{equation}\label{eqn:IdInCoordinates2}
	\lim_{b\to e,\, b\in U_{p_0}}\frac{\|d^{\varphi}\varphi_{a_0}[b]^{-1}\cdot{\varphi}_{p_0}(b)\|}{\|b\|}= 0.
	\end{equation}
	The function $d^{\varphi}\varphi_{a_0}$ is called the {\em intrinsic differential} of $\varphi$ at $a_0$. 
\end{definizione}

Let us fix $\varphi:U\subseteq \mathbb V\to \mathbb L$ with $U$ \textbf{open}. 
Whenever the intrinsic differential introduced in \cref{defiintrinsicdiff} exists, it is unique: see \cite[Theorem 3.2.8]{FMS14}. 
In \cite{FMS14} the authors prove the following result: a function $\varphi:U\subseteq \mathbb V\to \mathbb L$, with $U$ \textbf{open}, is intrinsically differentiable at $a_0$ if and only if $\mathrm{graph}(\varphi)$ is an intrinsically differentiable graph at $a_0\cdot\varphi(a_0)$ with the tangent $\mathbb V(a_0)$ complemented by $\mathbb L$, see \cref{defiintrinsicdiffgraph}, and moreover $\mathbb V(a_0)=\mathrm{graph}(d^\varphi \varphi_{a_0})$. 
In the setting we are dealing with, i.e., with maps $\varphi:U\subseteq\mathbb V\to\mathbb L$ with $U$ \textbf{compact}, the above equivalence still holds at density points of $U$. We do not give a proof of this last assertion since it follows by routine modifications of the argument in \cite{FMS14}, and moreover we do not need it in this paper.

\printbibliography

 \end{document}